\documentclass[journal]{IEEEtran}
\usepackage[x11names,dvipsnames,svgnames]{xcolor}
\usepackage{amsmath,amsfonts}
\usepackage{algorithmic}
\usepackage{algorithm}
\usepackage{array}
\usepackage[caption=false,font=normalsize,labelfont=sf,textfont=sf]{subfig}
\usepackage{textcomp}
\usepackage{stfloats}
\usepackage{url}
\usepackage{verbatim}
\usepackage{graphicx}
\usepackage{cite}
\usepackage{quiver}
\usepackage{amsthm}
\usepackage{stmaryrd}
\usepackage{mathtools}
\usepackage{tikz}
\usepackage{adjustbox}
\SetSymbolFont{stmry}{bold}{U}{stmry}{m}{n}
\usepackage{bm}
\usepackage[capitalize]{cleveref}
\usepackage[normalem]{ulem}

\newtheorem{theorem}{Theorem}[section]

\newtheorem{proposition}[theorem]{Proposition}

\newtheorem{assumption}{Assumption}

\newtheorem{example}[theorem]{Example}

\crefname{assumption}{assumption}{assumptions}
\Crefname{assumption}{Assumption}{Assumptions}

\theoremstyle{definition}
\newtheorem{definition}[theorem]{Definition}
\newtheorem{remark}[theorem]{Remark}
\newtheorem{construction}[theorem]{Construction}

\usepackage{globalvals}

\newcommand{\rrightarrow}{\to}

\newcommand{\namedCat}[1]{\textup{\textbf{#1}}}

\newcommand{\N}{\mathbb{N}}

\newcommand{\id}{\mathsf{id}}
\newcommand{\comp}{\mathop{\fatsemi}}

\newcommand{\iso}{\cong}

\newcommand{\Cat}{\namedCat{C}}

\renewcommand{\hom}{\mathrm{arrow}}
\newcommand{\ob}{\mathrm{ob}}

\newcommand{\RelPlus}{\namedCat{Rel}^+}
\newcommand{\SetMulti}{\namedCat{SetMulti}}

\newcommand{\FinStoch}{\namedCat{FinStoch}}

\newcommand{\into}{\to}

\newcommand{\inclusion}{\gamma}

\newcommand{\poss}{P}
\newcommand{\possplus}{\poss^+}
\newcommand{\prob}{D}

\newcommand{\update}{\mathsf{upd}}

\newcommand{\sys}[1]{\mathsf{#1}}

\newcommand{\Sys}{\mathbf{Sys}}

\newcommand{\states}[1]{#1}

\newcommand{\inputsx}{i}
\newcommand{\inputsX}{I}

\newcommand{\inputsY}{\inputsX'}

\newcommand{\inputsM}{J}

\newcommand{\obs}{{\mathsf{\inputsx}}}
\newcommand{\sts}{{\mathsf{s}}}

\newcommand{\sysx}{x}
\newcommand{\sysX}{X}
\newcommand{\SysX}{{\sys \sysX}}

\newcommand{\sysY}{X'}
\newcommand{\SysY}{{\sys \sysY}}
\newcommand{\sysm}{m}
\newcommand{\sysM}{M}
\newcommand{\SysM}{{\sys \sysM}}

\newcommand{\env}{E}
\newcommand{\Env}{{\sys \env}}
\newcommand{\plant}{P}
\newcommand{\Plant}{{\sys \plant}}
\newcommand{\cont}{C}
\newcommand{\Cont}{{\sys \cont}}
\newcommand{\contstar}{{\cont^*}}
\newcommand{\Contstar}{{\Cont^*}}
\newcommand{\contstaraut}{\contstar}
\newcommand{\Contstaraut}{{\Cont^*_{\mathsf{aut}}}}
\newcommand{\system}{S}
\newcommand{\System}{{\sys \system}}

\newcommand{\goal}{K}

\newcommand{\proj}{\pi}

\newcommand{\del}{\mathsf{del}}

\newcommand{\Copy}{\Delta}

\def\markto{\mathrel{\mkern2.5mu  \vcenter{\hbox{$\scriptstyle\bullet$}} \mkern-11.5mu{\to}}}

\newcommand{\belief}{\psi}
\newcommand{\model}{\mu}
\newcommand{\indmodel}{\model^{-1}}

\newcommand{\Unit}{I}
\newcommand{\parameter}{\theta}
\newcommand{\Parameter}{\Theta}
\newcommand{\hiddenstate}{x}
\newcommand{\Hiddenstate}{X}
\newcommand{\observation}{y}
\newcommand{\Observation}{Y}
\newcommand{\prior}{p}

\newcommand{\likelihood}{f}
\newcommand{\posterior}{\likelihood^\dagger}
\newcommand{\conj}{c}

\newcommand{\modelEnv}{\nu}
\newcommand{\modelenv}{\nu_\sts}
\newcommand{\modelSys}{\proj_{\Contstaraut}}
\newcommand{\modelsys}{\proj_{\contstaraut}}

\newcommand{\indmodelenv}{\modelenv^{-1}}

\newcommand{\indmodelsys}{\modelsys^{-1}}

\newcommand{\diamondupdate}[1]{\overline{\update_{\sys{#1}}}}

\newcommand{\reasoner}{\conj}
\newcommand{\hiddenmodel}{\kappa}
\newcommand{\posteriorFiltering}{\hiddenmodel^\dagger}

\usepackage{tikzit}

\tikzstyle{function clear}=[fill=white, draw=dgreen, shape=circle, minimum size=6 pt, inner sep=2 pt]
\tikzstyle{function black}=[fill=black, draw=black, shape=circle, inner sep=1, minimum size=2mm]
\tikzstyle{function red}=[fill=red, draw=red, shape=circle, inner sep=1, minimum size=2mm]
\tikzstyle{function blue}=[fill=blue, draw=blue, shape=circle, inner sep=1, minimum size=2mm]
\tikzstyle{function magenta}=[fill=magenta, draw=magenta, shape=circle, inner sep=1, minimum size=2mm]
\tikzstyle{function teal}=[fill=teal, draw=teal, shape=circle, inner sep=1, minimum size=2mm]
\tikzstyle{function brown}=[fill=brown, draw=brown, shape=circle, inner sep=1, minimum size=2mm]
\tikzstyle{function orange}=[fill=orange, draw=orange, shape=circle, inner sep=1, minimum size=2mm]
\tikzstyle{new style 0}=[fill={rgb,255: red,76; green,169; blue,170}, draw={rgb,255: red,76; green,169; blue,170}, shape=circle, inner sep=1]
\tikzstyle{box-.5-.5}=[fill=none, draw=black, shape=rectangle, minimum width=5mm, minimum height=5mm, thick]
\tikzstyle{box-.5-1}=[fill=none, draw=black, shape=rectangle, minimum width=5mm, minimum height=10mm, thick]
\tikzstyle{box-1-.5}=[fill=none, draw=black, shape=rectangle, minimum width=1cm, minimum height=5mm, thick]
\tikzstyle{fbox-1-.5-blue}=[fill={blue!10}, draw=black, shape=rectangle, minimum width=1cm, minimum height=5mm, thick]
\tikzstyle{box-1-.5-blue}=[fill=none, draw=blue, shape=rectangle, minimum width=1cm, minimum height=5mm, thick]
\tikzstyle{fbox-1-.5-red}=[fill={red!10}, draw=black, shape=rectangle, minimum width=1cm, minimum height=5mm, thick]
\tikzstyle{box-1-.5-red}=[fill=none, draw=red, shape=rectangle, minimum width=1cm, minimum height=5mm, thick]
\tikzstyle{box-1-1}=[fill=none, draw=black, shape=rectangle, minimum width=1cm, minimum height=1cm, thick]
\tikzstyle{box-1-1-red}=[fill=none, draw=red, shape=rectangle, minimum width=1cm, minimum height=1cm, thick]
\tikzstyle{box-1-2}=[fill=none, draw=black, shape=rectangle, minimum width=1cm, minimum height=2cm, thick]
\tikzstyle{box-1.5-.5}=[fill=none, draw=black, shape=rectangle, minimum width=1.5cm, minimum height=5mm, thick]
\tikzstyle{box-1.5-1}=[fill=none, draw=black, shape=rectangle, minimum width=1.5cm, minimum height=1cm, thick]
\tikzstyle{box-2-.5}=[fill=none, draw=black, shape=rectangle, minimum width=2cm, minimum height=.5cm, thick]
\tikzstyle{box-2-1}=[fill=none, draw=black, shape=rectangle, minimum width=2cm, minimum height=1cm, thick]
\tikzstyle{box-3-.5}=[fill=none, draw=black, shape=rectangle, minimum width=3cm, minimum height=.5cm, thick]
\tikzstyle{box-3-2}=[fill=none, draw=black, shape=rectangle, minimum width=3cm, minimum height=2cm, thick]
\tikzstyle{box-3-4}=[fill=none, draw=black, shape=rectangle, minimum width=3cm, minimum height=4cm, rounded corners, thick]
\tikzstyle{box-4-.5}=[fill=none, draw=black, shape=rectangle, minimum width=4cm, minimum height=.5cm, thick]
\tikzstyle{box-4-3}=[fill=none, draw=black, shape=rectangle, minimum width=4cm, minimum height=3cm, thick]
\tikzstyle{dashed-box-4-3}=[fill=none, draw=black, shape=rectangle, minimum width=4cm, minimum height=3cm, thick, rounded corners, dashed]
\tikzstyle{box-4-5}=[fill=none, draw=black, shape=rectangle, minimum width=4cm, minimum height=5cm, thick]
\tikzstyle{box-5-2}=[fill=none, draw=black, shape=rectangle, minimum width=5cm, minimum height=2cm, thick]
\tikzstyle{box-5-3}=[fill=none, draw=black, shape=rectangle, minimum width=5cm, minimum height=3cm, thick]
\tikzstyle{box-5-4}=[fill=none, draw=black, shape=rectangle, minimum width=5cm, minimum height=4cm, thick]
\tikzstyle{box-6-3}=[fill=none, draw=black, shape=rectangle, minimum width=6cm, minimum height=3cm, thick]
\tikzstyle{box-6-4}=[fill=none, draw=black, shape=rectangle, minimum width=6cm, minimum height=4cm, thick]
\tikzstyle{box-6-5}=[fill=none, draw=black, shape=rectangle, minimum width=6cm, minimum height=5cm, thick]
\tikzstyle{box-7-5}=[fill=none, draw=black, shape=rectangle, minimum width=7cm, minimum height=5cm, thick]
\tikzstyle{box-8-4}=[fill=none, draw=black, shape=rectangle, minimum width=8cm, minimum height=4cm, thick]
\tikzstyle{box-10-6}=[fill=none, draw=black, shape=rectangle, minimum width=10cm, minimum height=6cm, thick]
\tikzstyle{box-13-6}=[fill=none, draw=black, shape=rectangle, minimum width=13cm, minimum height=6cm, thick]
\tikzstyle{ellipse-20-12}=[fill=none, draw=red, shape=ellipse, minimum width=20cm, minimum height=12cm, thick]
\tikzstyle{ellipse-14-8}=[fill=none, draw=yellow, shape=ellipse, minimum width=14cm, minimum height=8cm, thick]
\tikzstyle{ellipse-8-4.5}=[fill=none, draw=green, shape=ellipse, minimum width=8cm, minimum height=4.5cm, thick]

\tikzstyle{object red}=[-, draw={rgb,255: red,191; green,0; blue,64}, thick]
\tikzstyle{object blue}=[-, draw=blue, thick]
\tikzstyle{object orange}=[-, draw={rgb,255: red,255; green,128; blue,0}, thick]
\tikzstyle{big dashes}=[-, dash pattern=on 4mm off 2mm, thick, dashed]
\tikzstyle{arrow}=[->]
\tikzstyle{thick-line}=[-, thick]
\tikzstyle{thick-blue-line}=[-, thick, draw=blue]
\tikzstyle{thick-red-line}=[-, thick, draw=red]
\tikzstyle{thick-double-line}=[-, thick, double]
\tikzstyle{thick-arrow}=[->, thick]
\tikzstyle{thick-blue-arrow}=[->, thick, draw=blue]
\tikzstyle{thick-red-arrow}=[->, thick, draw=red]
\tikzstyle{tube-blue}=[-, fill=none, draw={rgb,255: red,103; green,161; blue,255}, line width=6]
\tikzstyle{tube-green}=[-, draw={rgb,255: red,110; green,207; blue,108}, line width=6, fill=none]

\renewcommand{\textit}{\emph}

\begin{document}

\title{A Bayesian Interpretation\\of the Internal Model Principle}

\author{\textbf{Manuel Baltieri}, Araya Inc.$^\dagger$\thanks{$^\dagger$ Correspondence: manuel\_baltieri@araya.org} \\
    \textbf{Martin Biehl}, Cross Labs, Cross Compass Ltd.\\
    \textbf{Matteo Capucci}, University of Strathclyde and Independent Researcher\\
    \textbf{Nathaniel Virgo}, University of Hertfordshire and \\Earth-Life Science Institute, Institute of Science Tokyo}



\maketitle

\begin{abstract}
    The internal model principle, originally proposed in the theory of control of linear systems, nowadays represents a more general class of results in control theory and cybernetics.
    The central claim of these results is that, under suitable assumptions, if a system (a controller) can regulate against a class of external inputs (from the environment), it is because the system contains a model of the system causing these inputs, which can be used to generate signals counteracting them.
    Similar claims on the role of internal models appear also in cognitive science, especially in modern Bayesian treatments of cognitive agents, often suggesting that a system (a human subject, or some other agent) models its environment to adapt against disturbances and perform goal-directed behaviour.
    It is however unclear whether the Bayesian internal models discussed in cognitive science bear any formal relation to the internal models invoked in standard treatments of control theory.
    Here, we first review the internal model principle and present a precise formulation of it using concepts inspired by categorical systems theory.
    This leads to a formal definition of ``model'' generalising its use in the internal model principle.
    Although this notion of model is not \emph{a priori} related to the notion of Bayesian reasoning, we show that it can be seen as a special case of \emph{possibilistic} Bayesian filtering.
    This result is based on a recent line of work formalising, using Markov categories, a notion of \emph{interpretation}, describing when a system can be interpreted as performing Bayesian filtering on an outside world in a consistent way.

\end{abstract}


\section{Introduction}
\label{sec:intro}
A classic slogan in cybernetics states that ``every good regulator of a system must be a model of that system''~\cite{conant1970every}.
This idea, a relative of the ``law of requisite variety''~\cite{ashby1958requisite} is also influential in fields influenced by cybernetics, such as control theory, biology, artificial intelligence and cognitive science.

In control theory specifically, the ``internal model principle'' (IMP)~\cite{francis1976internal} refers to a general principle (a ``mold'', or a guide for a class of results as argued by~\cite{sontag2003adaptation, bin2022internal}) that in some sense formalises claims in~\cite{ashby1958requisite, conant1970every} by defining sufficient conditions for the existence of internal models of the environment in controllers for certain classes of regulation problems.

Similar concepts have been invoked in other fields. Internal models appear in biology based on the IMP and form the basis of work on homeostasis and (perfect) adaptation in living organisms at all scales, including microorganisms such as bacteria~\cite{yi2000robust, briat2016antithetic, khammash2021perfect}.

In artificial intelligence, the concept of \emph{world model}~\cite{ha2018world, wu2023daydreamer, taniguchi2023world}, closely related to the idea of an internal model, underlies a research programme with applications to reinforcement learning, robotics and deep learning, focusing on learning how to represent hidden properties of the environment~\cite{russell2009artificial}.

In cognitive science and neuroscience, internal models are broadly thought to constitute the computational basis of perception, motor control and high-level cognitive reasoning~\cite{grossberg1987competitive, kawato1999internal, ito2008control, baker2022three}, although there is no shortage of debate about this, e.g.~\cite{harvey1992untimed, Beer1995, van1995might, harvey2008misrepresentations}. In the context of neuroscience, internal models are often, though by no means universally, presented under a Bayesian framework.
According to the Bayesian view, brains or agents as whole systems, can be thought of as Bayesian reasoners and their cognitive processes as instances of Bayesian inference~\cite{knill1996perception, mcnamee2019internal, todorov2006optimal, seth2014cybernetic}.

While the label ``internal model,'' or just ``model'' is used across different disciplines, it is unclear whether it always refers to the same underlying formal concept.
If cognitive scientists propose internal models for the study of cognition, are they referring to the same kind of mathematical objects as control theorists working with internal models for regulation problems?
We do not fully answer these questions here, but take some steps towards answering them.

To do so, we structure this work in two main parts.
In the first part (\cref{sec:IMP}), we present the IMP developed by~\cite{wonham1976towards, hepburn1981internal, hepburn1984structurally, hepburn1984error, huang2018internal} using concepts inspired by categorical systems theory, a mathematical formalisation of systems and their interactions based on category theory~\cite{myers2021double, myers2023categorical, capucci2022towards}.
We take particular care in spelling out the assumptions that underpin work on the internal model principle, and along the way, we provide a definition of a ``model''.
This definition is inspired by and compatible with the one found in the IMP literature for closed and autonomous systems, but also applies to systems with inputs.
We then focus, in the remainder of the paper, on the case of autonomous systems, as treated by the standard IMP, and highlight one of the IMP's central aspects.
Its assumptions ensure that although the controller isn't autonomous, its dynamics are effectively described by an autonomous system that we call the ``autonomous attracting controller''.
This autonomous system is the one with a model of the environment.

In the second part, we first introduce the notion of a Bayesian filtering interpretation (\cref{sec:interpretationMap}).
This formally captures what it means for a system to support an interpretation as a Bayesian reasoner, and can also be considered as a formalisation of ``model'', albeit in a rather different sense than in the IMP.
Bayesian filtering interpretations were formulated in previous work~\cite{virgo2021interpreting} in the setting of Markov categories, a recent approach to synthetic probability theory~\cite{cho2019disintegration,fritz2020synthetic} situated in the tradition of {applied} category theory, and related to ideas on process theories~\cite{coecke2016mathematical, coecke2017picturing} and the graphical language of string diagrams.
Here we work in a particular Markov category, $\RelPlus$ (of sets and total relations), that captures \emph{possibilistic} rather than probabilistic non-determinism.

In \cref{sec:IMPInterpretation} we then show that, for autonomous systems, whenever there is an internal model in the IMP sense, there is also a corresponding Bayesian filtering interpretation.
To the best of our knowledge, this is the first result that establishes a formal connection between the concept of (internal) ``model'' used in the IMP, and the Bayesian inference/filtering literature.
This Bayesian filtering interpretation is however of a particular kind.
On the one hand it corresponds to a simplistic case of Bayesian filtering in which the system doing filtering never makes use of its inputs in any non-trivial way, meaning that although the prior changes over time in order to track the changing hidden state, the posterior for a given prior does not depend on the observation or input.
Because of this, we consider Bayesian filtering interpretations to be a more general notion of model than the one used in the IMP.
On the other hand, the Bayesian filtering interpretation employs a kind of approximation of the modelled system and not the modelled system directly.

\section{The Internal Model Principle}
\label{sec:IMP}
The IMP appears in the control theory literature as a series of different results that include, for instance, proposals using linear~\cite{francis1976internal} and nonlinear~\cite{isidori1990output} systems, and a focus on geometric~\cite{wonham1974linear, francis1976internal} or algebraic approaches~\cite{wonham1976towards}.
The latter are of particular interest in this work. The algebraic approach proposed by Wonham~\cite{wonham1976towards}, later refined in~\cite{hepburn1981internal, hepburn1984structurally, hepburn1984error}, paves the road for a more abstract version of the IMP, capturing its core assumptions for a broad class of system, without requiring explicit assumptions regarding the geometry (or \emph{any} geometry) of a system, see for instance~\cite{huang2018internal, wonham2019supervisory}.

In this section, we provide a self-contained account and critique of~\cite{hepburn1981internal, hepburn1984structurally, hepburn1984error}'s IMP.
In the remainder of this work, unless otherwise stated, ``IMP'' will refer to the one described in this section.
We start by first giving a definition of system that will be used throughout this work.

\subsection{Systems and their maps}
For the sake of simplicity, we will be working with discrete (in both time and space) dynamical systems, focusing on sets and functions.
Since we adopt a structural approach to our description, we expect that much of what follows can be stated at a higher level of abstraction, which would make it applicable to dynamical systems of other kinds (e.g.~smooth dynamical systems, as in~\cite{hepburn1981internal, hepburn1984structurally, hepburn1984error}, or topological/stochastic dynamical systems).\footnote{Indeed, all open dynamical systems as above fit in a common structural description, as shown in~\cite{myers2021double, myers2023categorical}.} In the present work we will use the following concrete definition:

\begin{definition}[System]
    \label{def:system}
    A \textit{system} (or more precisely, a \textit{fully observable system}) $\SysX$ is comprised of a set $\sysX$ of \textit{states}, a set $\inputsX$ of \textit{inputs} (or \textit{observations}), and an \textit{update} (or \textit{dynamics}) function:
    \begin{align}
        \begin{split}
            \update_{\SysX} &: \sysX \times \inputsX \to \sysX,
        \end{split}
    \end{align}
    The pair $\binom{\inputsX}{\sysX}$ is collectively referred to as the \textit{interface} of the system, and we write $\SysX : \Sys{\binom{\inputsX}{\sysX}}$ to mean $\SysX$ has such an interface.
\end{definition}

It might seem odd that the interface of a system includes its state space (as well as its inputs), however this is because we assume that the state space is exposed for other systems to observe.
Unless otherwise stated, we thus note that all of our uses of ``system'' in this work will specifically refer to the definition given above, meaning we will be dealing with fully observable systems.
Some of these systems will also have trivial inputs, meaning~$\inputsX$ is a singleton $1=\{\ast\}$.
We will call such systems \textit{autonomous}.

\begin{remark}
    We will denote a system $\SysX$ in a sans-serif font, and its state space $\sysX$ with the same letter but in regular font.
    A point we want to make in this section is that it is crucial to distinguish the two.
\end{remark}

Next we consider maps between systems.
This definition takes into account the fact that a system is comprised not only of a state space, but also of an interface and dynamics on its states.
So a map of systems is given by maps between their respective inputs and states which preserves their dynamics. Here and in the following we use the diagrammatic order for composition of functions i.e. if $f: X \to Y$ and $g:Y \to Z$ we write $(f\comp g):X \to Z$ for the composite function where $(f\comp g)(x) \coloneqq g \big( f(x) \big)$.

\begin{definition}[Map of systems]
    \label{def:mapsOfSystems}
    Let $\SysX : \Sys{\binom{\inputsX}{\sysX}}$ and $\SysY : \Sys{\binom{\inputsY}{\sysY}}$ be systems.
    A \textit{map of systems} $f:\SysX \to \SysY$ is comprised of two parts:
    \begin{enumerate}
        \item a \textit{map on states}, given by a function
        \begin{equation}
            f_\sts : \sysX \to \sysY,
        \end{equation}
        \item a \textit{map on inputs}, given by a function
        \begin{align}
            f_\obs &: \sysX \times \inputsX \to \inputsY,
        \end{align}
    \end{enumerate}
    such that the following diagram commutes:
    \begin{equation}
        \begin{matrix}
            \begin{tikzcd}[ampersand replacement=\&]
                {\sysX \times \inputsX} \&[10ex] {\sysY \times \inputsY} \\
                {\sysX} \& {\sysY}
                 \arrow["{(\proj_\sysX \comp f_\sts, f_\obs)}",from=1-1, to=1-2]
                \arrow["f_\sts", from=2-1, to=2-2]
                \arrow["{\update_{\SysX}}"', from=1-1, to=2-1]
                \arrow["{\update_{\SysY}}", from=1-2, to=2-2]
            \end{tikzcd}
        \end{matrix}
    \end{equation}
    meaning that, for every $\sysx \in \sysX, \inputsx \in \inputsX$, the following equation is satisfied:
    \begin{align}
        f_\sts \big( \update_{\sysX}(\sysx, \inputsx) \big) &= \update_{\sysY} \big( f_\sts(\sysx), f_\obs(\sysx, \inputsx) \big).
    \end{align}
\end{definition}

It's worth commenting on the map on inputs, since one might expect it to have type $\inputsX \to \inputsY$ rather than the more general $\sysX \times \inputsX \to \inputsY$.
This latter choice comes from the notion of \textit{chart} in categorical systems theory~\cite{myers2023categorical}, and allows a more general class of maps between systems.
We make use of this in \cref{rmk:projections}, since the maps described there would not exist if we used the more restrictive definition.
When $\SysX$ and $\SysY$ have the same set of inputs, a map between them can be given by just specifying the map on states.

\begin{construction}
\label{const:comp-of-sys-maps}
    In what follows, we will need to compose maps of systems, so we explain here how that happens.
    Given maps ${f:\sys X \to \sys X'}$ and ${g:\sys X' \to \sys X''}$, where ${\sys X : \Sys{\binom{I}{X}}}$, ${\sys X' : \Sys{\binom{I'}{X'}}}$ and ${\sys X'' : \Sys{\binom{I''}{X''}}}$, their composition
    ${f \comp g: \sys X \to \sys X''}$
    is given on states by the composition of the maps on states:
    \begin{equation}
        (f \comp g)_\sts =  f_\sts \comp g_\sts,
    \end{equation}
    while on inputs
    $(f \comp g)_\obs : X \times I \to I''$
    is defined as follows:
    \begin{equation}
        (f \comp g)_\obs(x, i) = g_\obs \big( f_\sts(x),f_\obs(x,i) \big).
    \end{equation}
\end{construction}

The notions of subsystem of a system and attracting subsystem will be crucial to define the regulation condition for the IMP, hence we introduce them here:

\begin{definition}[Subsystem]
    \label{def:subsystem}
    A \textit{subsystem} of $\SysX$ is a forward-invariant subset of $\sysX$ together with updates restricted to the subset, thus a map of systems $\sys \inclusion : \SysY \into \SysX$ given on states by an inclusion $\inclusion_\sts : \sysY \into \sysX$ and on inputs by projection $\inclusion_\obs = \proj_\inputsX : \sysY \times \inputsX \rrightarrow \inputsX$.
\end{definition}

Forward-invariant here means that if the state is in $\sysY$ then the next state will also be in $\sysY$, no matter what observation is received.
The dynamics of a subsystem are induced by the dynamics of its parent system. In fact one can say that a subsystem of $\SysX$ is given by a subset $\sysY \subseteq \sysX$ such that $\update_{\SysX}(\sysx', \inputsx) \in \sysY$ for all $\sysx' \in \sysY$ and $\inputsx \in \inputsX$. The update function of the subsystem is then just $\update_{\SysX}$ restricted to $\sysX' \times \inputsX$.

Note that the term ``subsystem'' is often used to mean a component, or a part, of a system.
This is not the sense in which we use the term here.
For example, using terminology from \cref{ass:factorisation} below, the controller is not a subsystem of the full system.

\begin{definition}[Attracting subsystem]
    \label{def:attractingSubsystem}
    An \textit{attracting subsystem} for $\SysX$ is a subsystem $\SysX^* \into \SysX$ such that, for each $\sysx \in \sysX$, there exists $n \in \N$ such that for all $t \geq n$ and $\underline \inputsx \in \inputsX^t$, we have $\update_{\SysX}^t(\sysx, \underline \inputsx) \in \sysX^*$ .
\end{definition}

Notice that we do not assume $\SysX^*$ to be the smallest subsystem of $\SysX$ meeting this criterion.
In part because of this, our definition is rather more general than the usual definition of attractor.
For example, $\SysX^*$ could be the basin of attraction of a fixed point, or it could contain multiple distinct orbits.
On the other hand, we stress that an attracting subsystem of a non-empty system must be non-empty.

\subsection{Systems that model other systems}
We now give the following definition which we believe to be novel, though in spirit very much like that proposed in~\cite{hepburn1984error} as well as in~\cite{conant1970every}.
This definition is a special map of systems, see \cref{def:mapsOfSystems}, one that describes when such a map can be seen as a model.

\begin{definition}[Model]
    \label{def:model}
    Given systems $\SysX : \Sys{\binom{\inputsX}{\sysX}}$ called the \emph{referent} and $\SysM : \Sys{\binom{\inputsM}{\sysM}}$ called the \emph{referrer}, a \emph{model} is a map $\model : \SysX \to \SysM$ such that
    \begin{enumerate}
        \item its part on states $\model_\sts : \sysX \rrightarrow \sysM$ is surjective, and
        \item its part on inputs $\model_\obs(\sysx,-): \inputsX \rrightarrow \inputsM$ is surjective for each $\sysx \in \sysX$.
    \end{enumerate}
\end{definition}

Often we will just say ``$\SysM$ models $\SysX$'', leaving $\model$ implicit.
The idea of such a definition is that the map $\model$ forgets some of the complexity of the system $\SysX$, which is mapped (surjectively) into the simpler system $\SysM$.
It is also connected to ``coarse-grainings''~\cite{noid2008multiscale}, ``variable aggregation''~\cite{simon1961aggregation}, ``state aggregation''~\cite{ren2002state}, ``lumpability''~\cite{kemeny1969finite}, ``model reduction''~\cite{moore1981principal}, ``dynamical consistency''~\cite{caines1995hierarchical} or other similar concepts, and reflected in standard ideas of ``supervisory control''~\cite{wonham2019supervisory}, for readers already familiar with any of these.

Importantly, however, our definition is not exactly the same as any of the standard definitions mentioned above since the map on inputs corresponds to a chart~\cite{myers2023categorical} for fully observable systems, as alluded to in the definition of maps of systems (\cref{def:mapsOfSystems}): it is of the form $\sysX \times \inputsX \to \inputsM$ and not simply $\inputsX \to \inputsM$.
This is important because it allows us to define a map of system between an autonomous system, e.g. $\System$, and one of its components as a non-autonomous one, e.g. $\Cont, \Plant$ or $\Env$, see \cref{rmk:projections}, something that wouldn't be possible with a map on inputs with the simpler type signature $\inputsX \to \inputsM$.

Having a model ${\mu:\SysX \to \SysM}$ means that for each state $\sysm \in \sysM$ we have a set of states $\indmodel(\sysm) \in \sysX$, called the \emph{fibre of $\sysm$}, which represents a subset of elements of $\sysX$ that are \emph{indistinguishable} from the perspective of the simpler system $\SysM$ as they all map to the same element $\sysm \in \sysM$ via the surjective function $\model_\sts$.
As $\sysm \in \sysM$ varies along the function $\update_{\SysM}$, this variation is consistent with the variation described by the function $\update_{\SysX}$ for each element $\sysx$ of the fibre $\indmodel(\sysm)$ of $\sysm$.
This will be further unpacked in \cref{const:model}.

\begin{remark}
\label{rmk:HWIMP}
    When applied to autonomous systems, a model reduces to the definition implicit in~\cite{hepburn1984error, hepburn1984structurally}, namely a surjective map of states commuting with the dynamics:
    \begin{equation}
    \label{eq:autonomous_model}
        \begin{tikzcd}[ampersand replacement=\&]
            \sysX \& \sysM \\
            \sysX \& \sysM
            \arrow["\update_\SysX"', from=1-1, to=2-1]
            \arrow["\model_\sts"', from=2-1, to=2-2] 
            \arrow["\model_\sts", from=1-1, to=1-2] 
            \arrow["\update_\SysM", from=1-2, to=2-2]
        \end{tikzcd}
    \end{equation}
    This is because the map on inputs of a model is necessarily of the form $\sysX \times 1 \rrightarrow 1$ for autonomous systems, due to the surjectivity condition. Therefore systems that model autonomous systems must also be autonomous, explaining the need for \Cref{ass:autonomousController} later on.

    In \cite{hepburn1984error, hepburn1984structurally} this definition is mixed with other assumptions, particularly \Cref{ass:allEnvironmentStatesExactlyOnce}, which we believe are not strictly necessary (and potentially problematic) for a notion of model.
    Our definition is also in terms of sets as in~\cite{huang2018internal, wonham2019supervisory} (rather than in terms of manifolds as in~\cite{hepburn1984error, hepburn1984structurally}).
\end{remark}

\begin{remark}
    Like any good definition, \cref{def:model} admits a trivial instance.
    A \emph{trivial model} is one where ${\model:\sysX \to \sysM}$ is a product projection, which means there exists a system $\sys F \in \Sys{\binom{H}{F}}$ such that the state space and inputs of the system $\SysX$ factor as $X = M \times F$ and $I=J \times H$ respectively, \emph{and} its dynamics decomposes in that of $\SysM$ and $\sys F$:
    \begin{equation}
    \label{eq:indep-dyn}
        \update_\SysX \big( (m,f),(j,h) \big) = \big( \update_\SysM(m,j), \, \update_{\sys F}(f,h) \big).
    \end{equation}
    Thus, in a trivial model, knowledge of $\SysM$ does not afford knowledge about the rest of $\SysX$, because $\SysM$ and $\sys F$ are not coupled.
    We stress that what makes a trivial model ``trivial'' is \cref{eq:indep-dyn}, rather than $\model_\sts$ and (fibrewise) $\model_\obs$ being product projections.
\end{remark}

\subsection{The IMP assumptions}
Having fixed what we mean by ``systems'', maps between them, and when such maps can be seen as models, let us define the systems of interest for the IMP.
The following is a typical control theoretic setup, and is used in particular in the work by Hepburn and Wonham~\cite{hepburn1981internal, hepburn1984error, hepburn1984structurally}, where it is presented a bit differently.
Throughout this section we introduce the assumptions used by Hepburn and Wonham to derive the internal model principle.
We note in advance that one of them, \Cref{ass:allEnvironmentStatesExactlyOnce}, seems to us rather difficult to motivate.
We begin with the following:

\begin{assumption}[Environment, plant, controller]
    \label{ass:factorisation}
    The following three \textit{components} are so defined:
    \begin{enumerate}
        \item the \textit{environment} $\Env : \Sys{\binom{1}{\env}}$ is an autonomous system
        \begin{align}
            \begin{split}
                \update_\Env & : \env \to \env,
                \label{eqn:updateEnv}
            \end{split}
        \end{align}
        \item the \textit{plant} $\Plant : \Sys{ \binom{\env \times \cont}{\plant}}$ is a system
        \begin{align}
            \begin{split}
                \update_\Plant & : \plant \times \env \times \cont \to \plant,
            \end{split}
        \end{align}
        \item the \textit{controller} $\Cont : \Sys{\binom{\plant}{\cont}}$ is a system
        \begin{align}
            \begin{split}
                \update_\Cont & : \cont \times \plant \to \cont.
                \label{eqn:updateCont}
            \end{split}
        \end{align}
    \end{enumerate}
    The \textit{full system} $\System : \Sys{\binom{1}{\env \times \plant \times \cont}}$ is the following composite autonomous system:
    \begin{align}
        \begin{split}
            \update_\System : \env \times \plant \times \cont \longrightarrow & \env \times \plant \times \cont \\
            (s_\Env, s_\Plant, s_\Cont) \longmapsto & \big( \update_\Env(s_\Env), \update_{\Plant}(s_\Plant, s_\Env, s_\Cont), \\
            & \update_{\Cont}(s_\Cont, s_\Plant) \big).
        \end{split}
    \end{align}
    Let ${\states{\system} = \env \times \plant \times \cont}$ denote the state space of $\System$.
\end{assumption}

We next look at the maps of systems that can be defined between these components thanks to the definition of maps of inputs in \cref{def:mapsOfSystems}.

\begin{remark}
    \label{rmk:projections}
    There are maps of systems ${\proj_\Env: \System \rrightarrow \Env}$, ${\proj_\Plant: \System \rrightarrow \Plant}$ and ${\proj_\Cont: \System \rrightarrow \Cont}$ induced by projecting out states of $\System$:
    \begin{equation}
        \begin{gathered}
            \begin{tikzcd}[ampersand replacement=\&]
                {\system \times 1} \&[5ex] {\env \times 1} \\
                \system \& \env
                \arrow["{(\proj_\env, \id_1)}", from=1-1, to=1-2] 
                \arrow["{\update_\System}"', from=1-1, to=2-1]
                \arrow["{\update_{\Env}}", from=1-2, to=2-2]
                \arrow["{\proj_\env}"', from=2-1, to=2-2] 
            \end{tikzcd}
            \hspace*{2ex}
            \begin{tikzcd}[ampersand replacement=\&]
                {\system \times 1} \&[3ex] {\cont \times \plant} \\
                \system \& \cont
                \arrow["{(\proj_{\cont}, \proj_{\plant})}", from=1-1, to=1-2] 
                \arrow["{\update_\System}"', from=1-1, to=2-1]
                \arrow["{\update_\Cont}", from=1-2, to=2-2]
                \arrow["{\proj_\cont}"', from=2-1, to=2-2] 
            \end{tikzcd}
            \\
            \begin{tikzcd}[ampersand replacement=\&]
                \system \times 1 \&[6ex] {\plant \times \env \times \cont} \\
                \system \& \plant
                \arrow["{(\proj_\plant, \proj_{\env \times \cont})}", from=1-1, to=1-2] 
                \arrow["{\update_\System}"', from=1-1, to=2-1]
                \arrow["{\update_\Plant}", from=1-2, to=2-2]
                \arrow["{\proj_\plant}"', from=2-1, to=2-2] 
            \end{tikzcd}
        \end{gathered}
    \end{equation}
\end{remark}

In this setting, the environment's role is to produce undesirable signals, or disturbances, the controller and plant must adapt to.
The plant-controller subsystem has no control over it, since $\Env$ is an autonomous system, i.e. it does not receive inputs from controller, plant or any other system.
In~\cite{wonham1976towards} Wonham calls the environment a ``convenient fiction'' which can be used as a placeholder for disturbances that $\Cont$ is designed to compensate.
We will comment more on this point later.

The controller's job is to \emph{regulate} plant and environment, according to some notion of regulation. Formally, we define:

\begin{definition}[Regulation problem]
    \label{def:regulaitonProblem}
    A \textit{regulation problem} (or \textit{reguland}~\cite{conant1970every}) is a triple $(\Env, \Plant, \Tilde \goal \subseteq \env \times \plant)$.
\end{definition}

$\Tilde\goal$ is a set of \textit{target states} (also known as goal states, or references). This set specifies what the controller is supposed to achieve.
We define it as a subset of $\env \times \plant$ rather than just $\plant$, which means that the desired states of the plant can also depend on the state of the environment.
We then define the set $\goal = \cont \times \tilde{\goal}$ of states of the complete system in which the plant-environment components are in a target state.
By slight abuse of terminology, we also refer to the elements of $\goal$ as target states.
A regulation problem defines all the data we need to formulate the IMP.
However, for the IMP to hold we need some additional assumptions, which amount to properties that we impose on the regulation problem.

We should note at this point that the dynamics of $\System$ might exit and enter the set of target states $\goal$.
Thus, in order to say that $\Cont$ actually regulates $\System$, we need to make assumptions regarding the way $\goal$ fits in the dynamics of $\System$:

\begin{assumption}[Regulation condition and attracting full system]
    \label{ass:regulationCondition}
    $\Cont$ solves its regulation problem, meaning there exists an attracting subsystem (\Cref{def:attractingSubsystem}), that we shall call the \textit{attracting full system} $\System^* \into \System$ such that, on states, $\system^* \subseteq \goal$.
\end{assumption}

The idea of this assumption is that no matter what state the whole system starts in, there eventually will be a time at which the environment-plant system is in $\goal$ \emph{and} it will remain there indefinitely.
Note that there is always a maximal subsystem of $\System$ contained in $\goal$ (just take the union of all such subsystems), but (1) it might be empty and (2) it might not be attracting.
\Cref{ass:regulationCondition} guarantees that neither of these is the case.

The next assumptions buy some extra structure so that we can ultimately find a way to compare $\Cont$ and $\Env$.

\begin{construction}[Attracting controller]
    \label{const:attractingController}
    Even though $\goal$ contains all states of $\Cont$, this might not be true anymore for $\system^*$.
    We define the set of \textit{attracting control states} as
    \begin{equation*}
        \contstar \coloneqq \{ s_\Cont \in \cont \mid \exists s_\Env \in \env, s_\Plant \in \plant, (s_\Env, s_\Plant, s_\Cont) \in \system^* \}.
    \end{equation*}
    It is obtained along with a surjection $\proj_{\contstar}: \system^* \rrightarrow \contstar$ by taking the image of $\system^*$ under the projection map ${\proj_\cont : \system \rrightarrow \cont}$:
    \begin{equation}
        \label{eq:active.part.proj}
        \begin{tikzcd}[ampersand replacement=\&]
            {\system^*} \& {\contstar} \\
            \system \& \cont
            \arrow["{\proj_\cont}"', from=2-1, to=2-2] 
            \arrow[from=1-1, to=2-1] 
            \arrow[dashed, from=1-2, to=2-2] 
            \arrow["{\proj_{\contstar}}", dashed, from=1-1, to=1-2] 
        \end{tikzcd}
    \end{equation}
    $\cont^*$ supports a closed dynamics given by that of $\cont$ restricted to the states in $\cont^*$, and the inputs from $\plant^*$ defined as:
    \begin{equation}
        \plant^* \coloneqq \{ s_\Plant \in \plant \mid \exists s_\Env \in \env, s_\Cont \in \cont, (s_\Env, s_\Plant, s_\Cont) \in \system^* \}.
    \end{equation}
    This gives a system, the \textit{attracting controller} $\Contstar : \Sys{\binom{\plant^*}{\cont^*}}$, with interface $\binom{\plant^*}{\cont^*}$ and dynamics
    \begin{align}
    \label{eqn:attractingControllerDynamics}
        \update_\Contstar : \cont^* \times \plant^* & \to \cont^*\\
        (s_\Cont, s_\Plant) &\mapsto \update_\Cont(s_\Cont, s_\Plant) \nonumber
    \end{align}
    which is guaranteed to be well-defined since $\System^*$ is a closed system.
    This allow us to promote $\pi_{\Cont^*}$ to a full map of systems:
    \begin{equation}
        \begin{tikzcd}[ampersand replacement=\&]
            {\System^*} \& {\Contstar}
            \arrow["\proj_{\Contstar}", from=1-1, to=1-2] 
        \end{tikzcd}
    \end{equation}
\end{construction}

We would like to show next that the attracting controller $\Contstar$ tracks $\Env$ (or a subsystem of $\Env$ whose states are contained in $\goal$), but $\Env$ is autonomous, while $\Contstar$ isn't, so $\Contstar$ can't model $\Env$ according to \cref{rmk:HWIMP}.

To resolve this issue, following~\cite{wonham1976towards, hepburn1981internal, hepburn1984structurally, hepburn1984error}, we assume that the controller $\Cont$ operates in a particular way.
When the system is outside the target states, the controller is in a \emph{closed loop} regime in which its evolution depends on the state of the plant, by virtue of the inputs received from it.
It is however allowed to switch to an \emph{open loop} strategy.
In particular, when the system is in a target state, we assume that the controller operates purely in an open loop setting, in which its outputs and state changes don't depend on the state of the plant.

This assumption comes from thinking about the controller as performing \emph{error feedback}: the controller measures how far the plant-environment system is from a target state, and counteracts accordingly, so that when inside the set of target states, the controller essentially operates in an open loop regime.
This assumption plays an important role in certain areas of control theory, for instance in the ``disturbance decoupling problem''~\cite{isidori1981nonlinear}, and cognitive science, where it corresponds to the idea of ``decouplability'' or ``detachability''~\cite{thobani2024triviality}.
However, it is also quite a strong assumption and there are plenty of control strategies that violate it.
This includes the case where the controller operates by error feedback in the presence of noise and the target set includes some margin for error around the reference point.
In this case the controller must react to the noise in order to keep the system inside the margin, so it is not autonomous.

\begin{assumption}[Autonomous attracting controller]
    \label{ass:autonomousController}
    There exists an autonomous system $\Contstaraut : \Sys{\binom{1}{\contstaraut}}$, that we call the \textit{autonomous attracting controller}, admitting a map $\Contstar \to \Contstaraut$ which is the identity on states and the projection $!:\cont^* \times \plant^* \to 1$ on inputs:
    \begin{equation}
        \begin{tikzcd}
            {\Cont^*} & \Contstaraut
        	\arrow["{(\id, !)}", from=1-1, to=1-2]
        \end{tikzcd}
    \end{equation}
    We still get a projection down from $\System^*$ as the dashed arrow below, which we note is surjective on inputs:
    \begin{equation}
        \begin{tikzcd}
        	{\System^*} \\
        	{\Cont^*} & \Contstaraut
        	\arrow["{\proj_{\Cont^*}}"', from=1-1, to=2-1]
        	\arrow["{\proj_{\Contstaraut}}", dashed, from=1-1, to=2-2]
        	\arrow["{(\id, !)}"', from=2-1, to=2-2]
        \end{tikzcd}
    \end{equation}
\end{assumption}

Given \cref{def:model}, the following thus holds:

\begin{theorem}[\textit{Internal Model Principle}]
    \label{th:impFullSystem}
    Let $\sys S$ be a system subject to \Cref{ass:factorisation,ass:regulationCondition,ass:autonomousController}.
    Then $\Contstaraut$ models the attracting full system $\System^*$ via the projection $\proj_{\Contstaraut}$.
\end{theorem}

\begin{proof}
     \Cref{ass:factorisation} implies the existence of a map ${\proj_\Cont : \System \to \Cont}$.
     \Cref{ass:regulationCondition} prescribes that it restricts to $\proj_{\Contstar} : \System^* \to \Contstar$ (see \cref{const:attractingController}).
     However this is \emph{not} a model.
     \Cref{ass:autonomousController} says that $\proj_{\Contstaraut} : \System^* \to \Contstaraut$ is a model by assuming the existence of an autonomous $\Contstaraut$, making its part on inputs
     surjective.
\end{proof}

We note that this version of the IMP is not the same given by \cite{wonham1976towards, hepburn1981internal, hepburn1984structurally, hepburn1984error, huang2018internal, wonham2019supervisory} because it involves a map of systems between the attracting full system $\System^*$ and the autonomous attracting controller $\Contstaraut$.
It is however a necessary condition for the results in \cite{hepburn1981internal, hepburn1984structurally, hepburn1984error}, as we will see shortly, and we believe that it gives a fresh and mostly novel perspective on the classical IMP, where the model doesn't represent knowledge about just the environment, but about the full system that contains environment, plant and the controller itself, thus \emph{embedding}, in the sense of \cite{demski2019embedded}, the (autonomous attracting) controller in the (attracting) full system it models.

Furthermore, the kind of model obtained through \Cref{th:impFullSystem} might not be what the name ``internal model principle'' immediately suggests, especially when considering more traditional accounts such as \cite{hepburn1984error}.
In fact, a model as in \Cref{def:model} is a map between systems, therefore existing \emph{externally} from the systems themselves, at the level of their structural descriptions given by an external designer, or control theorist.
It follows that $\pi_{\Contstaraut}$ witnesses the fact that the way $\Contstaraut$ works necessarily reflects the way $\System^*$ does, indicating that $\Cont$ must have been designed to somehow predict the way $\System$ works under a regulation condition (\Cref{ass:regulationCondition}).
This terminology is used in contrast to a more explicit ``internal model'' expressed in terms of an actual \emph{mechanism}, part of a given system.

To obtain what Hepburn and Wonham called the IMP on the other hand, we need another assumption.
This is quite strong, and hard to motivate in practice as far as we can tell.
First however, we note that ust like $\pi_\Cont$ restricts to the attracting subsystems as $\pi_\Contstar$, so does the projection ${\proj_\Env: \System \rrightarrow \Env}$ described in~\Cref{rmk:projections}.
\begin{remark}[Attracting environment]
    The attracting full system $\System^*$ induces an analogous global attractor $\Env^*$ in $\Env$ (again, obtained by restricting $\Env$ to those states which are part of at least one state of $\System^*$).
    We will call the system $\Env^*$ the \textit{attracting environment}, by analogy with the autonomous attracting controller in \Cref{const:attractingController}.
\end{remark}
The next assumption is arguably the core of the IMP as described in~\cite{hepburn1981internal, hepburn1984structurally, hepburn1984error}.

\begin{assumption}
    \label{ass:allEnvironmentStatesExactlyOnce}
    There is an isomorphism of systems $\System^* \iso \Env^*$, meaning that for each environment state $s_\Env \in \env^*$, there is \emph{exactly one} $s \in \system^*$ such that $\proj_\env(s) = s_\Env$.
\end{assumption}

This assumption can be used to obtain a map of systems between the autonomous attracting controller $\Contstaraut$ and the attracting environment $\Env^*$, which leads to a complete reformulation of~\cite{hepburn1984error}:

\begin{theorem}[\textit{Internal Model Principle (Hepburn and Wonham)}]
    \label{th:imp}
    Let $\sys S$ be a system subject to \Cref{ass:factorisation,ass:regulationCondition,ass:autonomousController,ass:allEnvironmentStatesExactlyOnce}.
    Then $\Contstaraut$ models the attracting environment $\Env^*$ via the dashed map below:\footnote{Such a composite is defined by \cref{const:comp-of-sys-maps}.}
    \begin{equation}
        \begin{tikzcd}[ampersand replacement=\&]
            \Env^* \&[2ex]\&[2ex] {\Contstaraut} \\[-2ex]
            \& {\System^*}
            \arrow["\modelEnv", dashed, from=1-1, to=1-3] 
            \arrow["\sim"{rotate=-20,pos=0.35}, "{\text{\Cref{ass:allEnvironmentStatesExactlyOnce}}}"', curve={height=6pt}, from=1-1, to=2-2]
            \arrow["{\modelSys}", "{\text{\Cref{ass:autonomousController}}}"', curve={height=6pt}, from=2-2, to=1-3] 
        \end{tikzcd}
        \label{eqn:IMPPartiallyObservable}
    \end{equation}
\end{theorem}

This is the version of the IMP presented by, for instance~\cite{hepburn1981internal, hepburn1984structurally, hepburn1984error, huang2018internal, wonham2019supervisory}.
As seen in the diagram, this relies on \cref{th:impFullSystem}, via \Cref{ass:autonomousController}.
We thus argue that the stronger result is the one given by \Cref{th:impFullSystem}, since \Cref{th:imp} is only obtained after adding the rather strong \Cref{ass:allEnvironmentStatesExactlyOnce}.

In the next section, we will link the notion of model given by \cref{def:model} to that of Bayesian interpretation put forward by~\cite{virgo2021interpreting, biehl2022interpreting}.
This is somewhat surprising, because models in Bayesian statistics and models according to definition \cref{def:model} are quite different things and it's not a priori obvious that they would be connected at all.

\section{Bayesian updates and interpretations}
\label{sec:interpretationMap}
The notion of interpretation map was introduced in~\cite{virgo2021interpreting} as a map between the state of a system and probability distributions on the states of the external world describing the former's beliefs about the latter.
The concept has since then been developed further in a category theory context in~\cite{virgo2023unifilar}.
The goal was to understand what it means to \emph{interpret} a physical system as performing Bayesian inference; that is, what properties must a physical system have in order to be able to make the claim that it \emph{has a Bayesian prior} over some hidden variable?
This question, and the approach taken to it, is similar in spirit to~\cite{horsman2014does}, with the main difference being the focus on Bayesian reasoning rather than computation.

This notion has also been applied to partially observable Markov decision processes~\cite{biehl2022interpreting}, sharing some mathematical background with the state-space control theoretic approach behind the internal model principle.
In this section we briefly summarise some of the results in~\cite{virgo2021interpreting, biehl2022interpreting} that will help us later establish connections between the internal model principle and Bayesian reasoning ideas.

\subsection{Markov categories}
Markov categories are a synthetic approach to probability theory and formalise the compositional structure of non-deterministic processes that behave like Markov kernels.
The canonical reference for Markov categories is~\cite{fritz2020synthetic}.
We begin with the definition of a category.

\begin{definition}[Category]
    A category $\Cat$ consists of:
    \begin{itemize}
        \item a class of \textit{objects}, $\ob(\Cat)$, e.g. $A, B, C, \dots$,
        \item a class of \textit{maps}, \textit{arrows} or \textit{morphisms}, $\hom(\Cat)$ (these three terms are interchangeable),
        \item for each arrow in $\hom(\Cat)$, a \textit{source} and a \textit{target}, which are objects, i.e. elements of $\ob(\Cat)$,
        if an arrow $f$ has source $A$ and target $B$ then we often write it as $f:A\to B$, and we say that $A\to B$ is the arrow's \textit{type},
        \item for each object of $\ob(\Cat)$, an \textit{identity morphism} $\id_A: A \to A$,
        \item a binary operation $\comp$ on arrows called the \textit{composition rule}, such that given morphisms $f:A\to B$ and $g: B\to C$, their composite $f \comp g$ is an arrow with type $A \to C$; composition is defined when (and only when) the target of one arrow equals the source of another, and must obey the following laws:
        \begin{itemize}
            \item \textit{associativity}: given morphisms $f:A\to B$, $g:B\to C$ and $h:C\to D$, we must have $f \comp (g \comp h) = (f \comp g) \comp h$,
            \item left and right \textit{unit laws}: for every pair of objects $A, B$ and morphism $f : A \to B$, we must have $\id_A \comp f = f = f \comp \id_B$.
        \end{itemize}
        In other works one will often find composition written as $g\circ f$ instead of $f\comp g$.
    \end{itemize}
\end{definition}

Next we introduce a graphical notation, string diagrams, for a special class of categories with extra structure, known as symmetric monoidal categories, that will be used throughout this work.
We give an abridged definition that gives enough information to understand the rest of the work, skipping over some important details (the so-called {coherence laws}) that are needed for the graphical language to be formally valid~\cite{borceuxHandbookCategoricalAlgebraVol21994}.
For a concise formal treatment of symmetric monoidal categories and their graphical language, a good reference is~\cite{baez2010physics}; see also~\cite{selinger2010survey} for more on string diagrams and~\cite{coecke2017picturing} for a much more comprehensive beginner-friendly introduction.

\begin{definition}[Symmetric monoidal category]
    A \textit{symmetric monoidal category} (also known as a \textit{process theory}~\cite{coecke2017picturing}) is a category $\Cat$ with the following additional structure:
    \begin{itemize}
        \item objects as wires, or strings, $A, B, C, \dots$
            \begin{equation}
                \scalebox{.8}{\begin{tikzpicture}
	\begin{pgfonlayer}{nodelayer}
		\node [style=none] (0) at (-5.5, 0) {};
		\node [style=none] (1) at (0.25, 0) {};
		\node [style=none] (2) at (-4.25, 0.5) {$A$};
		\node [style=none] (3) at (-0.75, 0.5) {$B$};
		\node [style=none] (4) at (-2.75, -0.075) {,};
		\node [style=none] (5) at (-3.25, 0) {};
		\node [style=none] (6) at (-2, 0) {};
		\node [style=none] (7) at (3.75, 0) {};
		\node [style=none] (8) at (2.75, 0.5) {$C$};
		\node [style=none] (9) at (0.75, -0.075) {,};
		\node [style=none] (10) at (1.5, 0) {};
		\node [style=none] (11) at (4.75, -0.075) {, $\dots$};
	\end{pgfonlayer}
	\begin{pgfonlayer}{edgelayer}
		\draw [style=thick-line] (0.center) to (5.center);
		\draw [style=thick-line] (6.center) to (1.center);
		\draw [style=thick-line] (10.center) to (7.center);
	\end{pgfonlayer}
\end{tikzpicture}}
            \end{equation}

        \item morphisms as boxes, $f: A \to B$, $g: B \to C$, $h: C \to D, \dots$, converting objects to new objects and forming processes (combinations of objects and morphisms to track their changes)
            \begin{equation}
                \scalebox{.8}{\begin{tikzpicture}
	\begin{pgfonlayer}{nodelayer}
		\node [style=box-1-.5] (0) at (-7.25, -0.25) {$f$};
		\node [style=none] (1) at (-9.75, -0.25) {};
		\node [style=none] (2) at (-4.75, -0.25) {};
		\node [style=none] (3) at (-9, 0.25) {$A$};
		\node [style=none] (4) at (-5.5, 0.25) {$B$};
		\node [style=none] (5) at (-4.25, -0.325) {,};
		\node [style=box-1-.5] (6) at (-1, -0.25) {$g$};
		\node [style=none] (7) at (-3.5, -0.25) {};
		\node [style=none] (8) at (1.5, -0.25) {};
		\node [style=none] (9) at (-2.75, 0.25) {$B$};
		\node [style=none] (10) at (0.75, 0.25) {$C$};
		\node [style=none] (11) at (2, -0.325) {,};
		\node [style=box-1-.5] (12) at (5.25, -0.25) {$h$};
		\node [style=none] (13) at (2.75, -0.25) {};
		\node [style=none] (14) at (7.75, -0.25) {};
		\node [style=none] (15) at (3.5, 0.25) {$C$};
		\node [style=none] (16) at (7, 0.25) {$D$};
		\node [style=none] (17) at (8.75, -0.325) {,  $\dots$};
	\end{pgfonlayer}
	\begin{pgfonlayer}{edgelayer}
		\draw [style=thick-line] (1.center) to (0);
		\draw [style=thick-line] (0) to (2.center);
		\draw [style=thick-line] (7.center) to (6);
		\draw [style=thick-line] (6) to (8.center);
		\draw [style=thick-line] (13.center) to (12);
		\draw [style=thick-line] (12) to (14.center);
	\end{pgfonlayer}
\end{tikzpicture}}
            \end{equation}

        \item composition appears in two forms, sequential and parallel; sequential composition, denoted by $\comp$, corresponds to the connection of wires with the appropriate type ($A$'s with $A$'s, $B$'s with $B$'s, etc.) such that
              \begin{equation}
                  \scalebox{0.8}{\begin{tikzpicture}
	\begin{pgfonlayer}{nodelayer}
		\node [style=box-1-.5] (0) at (-5.75, -0.25) {$f$};
		\node [style=none] (1) at (-8.25, -0.25) {};
		\node [style=none] (3) at (-7.5, 0.25) {$A$};
		\node [style=none] (4) at (-4, 0.25) {$B$};
		\node [style=box-1-.5] (6) at (-0.75, -0.25) {$g$};
		\node [style=none] (7) at (-3.25, -0.25) {};
		\node [style=none] (8) at (1.75, -0.25) {};
		\node [style=none] (9) at (-2.5, 0.25) {$B$};
		\node [style=none] (10) at (1, 0.25) {$C$};
		\node [style=none] (11) at (2.75, -0.325) {$=$};
		\node [style=box-1-.5] (12) at (6.25, -0.25) {$f \comp g$};
		\node [style=none] (13) at (3.75, -0.25) {};
		\node [style=none] (14) at (8.75, -0.25) {};
		\node [style=none] (15) at (4.5, 0.25) {$A$};
		\node [style=none] (16) at (8, 0.25) {$C$};
	\end{pgfonlayer}
	\begin{pgfonlayer}{edgelayer}
		\draw [style=thick-line] (1.center) to (0);
		\draw [style=thick-line] (7.center) to (6);
		\draw [style=thick-line] (6) to (8.center);
		\draw [style=thick-line] (13.center) to (12);
		\draw [style=thick-line] (12) to (14.center);
		\draw [style=thick-line] (0) to (7.center);
	\end{pgfonlayer}
\end{tikzpicture}}
              \end{equation}

              while parallel composition,
              denoted by $\otimes$, is simply depicted as
              \begin{equation}
                  \scalebox{0.8}{\begin{tikzpicture}
	\begin{pgfonlayer}{nodelayer}
		\node [style=box-1-.5] (0) at (-4.25, 1) {$f$};
		\node [style=none] (1) at (-6.75, 1) {};
		\node [style=none] (2) at (-1.75, 1) {};
		\node [style=none] (3) at (-6, 1.5) {$A$};
		\node [style=none] (4) at (-2.5, 1.5) {$B$};
		\node [style=box-1-.5] (5) at (-4.25, -1) {$h$};
		\node [style=none] (6) at (-6.75, -1) {};
		\node [style=none] (7) at (-1.75, -1) {};
		\node [style=none] (8) at (-6, -0.5) {$C$};
		\node [style=none] (9) at (-2.5, -0.5) {$D$};
		\node [style=none] (10) at (-1, -0.075) {$=$};
		\node [style=box-1-.5] (11) at (3.5, 0) {$f \otimes h$};
		\node [style=none] (12) at (-0.25, 0) {};
		\node [style=none] (13) at (7.25, 0) {};
		\node [style=none] (14) at (1, 0.5) {$A \otimes C$};
		\node [style=none] (15) at (6, 0.5) {$B \otimes D$};
	\end{pgfonlayer}
	\begin{pgfonlayer}{edgelayer}
		\draw [style=thick-line] (1.center) to (0);
		\draw [style=thick-line] (0) to (2.center);
		\draw [style=thick-line] (6.center) to (5);
		\draw [style=thick-line] (5) to (7.center);
		\draw [style=thick-line] (12.center) to (11);
		\draw [style=thick-line] (11) to (13.center);
	\end{pgfonlayer}
\end{tikzpicture}}
              \end{equation}
    \end{itemize}
    The composition operations must obey:
    \begin{itemize}
        \item associativity, for sequential composition given by (equalities) $f \comp (g \comp h) = (f \comp g) \comp h = f \comp g \comp h$ and for parallel composition by (isomorphism) $f \otimes (g \otimes h) \cong (f \otimes g) \otimes h \cong f \otimes g \otimes h$,
        \item identity, for sequential composition in the form of an identity box (usually not drawn, as on the right hand side below)
            \begin{equation}
                \scalebox{0.8}{\begin{tikzpicture}
	\begin{pgfonlayer}{nodelayer}
		\node [style=box-1-.5] (0) at (-2, -0.25) {$\text{id}_A$};
		\node [style=none] (1) at (-4.5, -0.25) {};
		\node [style=none] (2) at (0.5, -0.25) {};
		\node [style=none] (3) at (-3.75, 0.25) {$A$};
		\node [style=none] (4) at (-0.25, 0.25) {$A$};
		\node [style=none] (5) at (5, -0.25) {};
		\node [style=none] (6) at (3.5, 0.25) {$A$};
		\node [style=none] (7) at (1.25, -0.325) {$=$};
		\node [style=none] (8) at (2, -0.25) {};
	\end{pgfonlayer}
	\begin{pgfonlayer}{edgelayer}
		\draw [style=thick-line] (1.center) to (0);
		\draw [style=thick-line] (0) to (2.center);
		\draw [style=thick-line] (8.center) to (5.center);
	\end{pgfonlayer}
\end{tikzpicture}}
            \end{equation}
            and for parallel composition in the form of a distinguished wire $I$, often drawn as no wire,
            \begin{equation}
                \label{eqn:monoidalUnit}
                \scalebox{1}{\begin{tikzpicture}
	\begin{pgfonlayer}{nodelayer}
		\node [style=none] (0) at (-5.25, 0) {};
		\node [style=none] (1) at (-2, 0) {};
		\node [style=none] (2) at (-3.5, 0.5) {$I$};
		\node [style=none] (3) at (0, 0) {=};
		\node [style=box-1.5-1, dashed] (4) at (3.25, 0) {};
	\end{pgfonlayer}
	\begin{pgfonlayer}{edgelayer}
		\draw [style=thick-line] (0.center) to (1.center);
	\end{pgfonlayer}
\end{tikzpicture}}
            \end{equation}
            such that for every object $A \in \Cat$, $I \otimes A \cong A \cong A \otimes I$, i.e. in string diagrammatic terms, any wire can be seen as having infinitely many not-drawn identity wires in parallel,
        \item and symmetry, in the form of a morphism $\sigma_{A,B} : A \otimes B \to B \otimes A$, drawn as two wires crossing:
        \begin{equation}
            \scalebox{0.8}{\begin{tikzpicture}
	\begin{pgfonlayer}{nodelayer}
		\node [style=none] (0) at (-3.25, 0.5) {};
		\node [style=none] (1) at (-1.5, 0.5) {};
		\node [style=none] (2) at (1.5, 0.5) {};
		\node [style=none] (3) at (3.25, 0.5) {};
		\node [style=none] (4) at (-3.25, -0.75) {};
		\node [style=none] (5) at (-1.5, -0.75) {};
		\node [style=none] (6) at (1.5, -0.75) {};
		\node [style=none] (7) at (3.25, -0.75) {};
		\node [style=none] (8) at (-2.25, 1) {$A$};
		\node [style=none] (9) at (2.25, -0.25) {$A$};
		\node [style=none] (10) at (2.25, 1) {$B$};
		\node [style=none] (11) at (-2.25, -0.25) {$B$};
	\end{pgfonlayer}
	\begin{pgfonlayer}{edgelayer}
		\draw [style=thick-line] (0.center) to (1.center);
		\draw [style=thick-line] (4.center) to (5.center);
		\draw [style=thick-line] (2.center) to (3.center);
		\draw [style=thick-line] (6.center) to (7.center);
		\draw [style=thick-line, in=-180, out=0] (1.center) to (6.center);
		\draw [style=thick-line, in=-180, out=0] (5.center) to (2.center);
	\end{pgfonlayer}
\end{tikzpicture}}
            \label{fig:CatSymmetric}
        \end{equation}
        such that $\sigma_{A,B} \comp \sigma_{B,A} = \mathrm{id}_{A \otimes B}$.
    \end{itemize}
    The isomorphisms witnessing associativity, unitality and symmetry satisfy coherence laws~\cite[Definition~6.1.1]{borceuxHandbookCategoricalAlgebraVol21994}, as mentioned before, will not be of concern here. We will however state the following properties in the graphical notation, as they will be regularly used in string diagrams manipulations later on:
    \begin{itemize}
        \item interchange law
        \begin{equation}
            \scalebox{0.7}{\begin{tikzpicture}
	\begin{pgfonlayer}{nodelayer}
		\node [style=box-1-.5] (0) at (0, 1.25) {$f$};
		\node [style=none] (1) at (-2.5, 1.25) {};
		\node [style=none] (2) at (2.5, 1.25) {};
		\node [style=none] (3) at (-1.75, 1.75) {$A$};
		\node [style=none] (4) at (1.75, 1.75) {$B$};
		\node [style=box-1-.5] (5) at (0, -1.25) {$h$};
		\node [style=none] (6) at (-2.5, -1.25) {};
		\node [style=none] (7) at (2.5, -1.25) {};
		\node [style=none] (8) at (-1.75, -0.75) {$C$};
		\node [style=none] (9) at (1.75, -0.75) {$D$};
		\node [style=box-1-.5] (10) at (-8.25, 1.25) {$f$};
		\node [style=none] (11) at (-10.25, 1.25) {};
		\node [style=none] (12) at (-4, 1.25) {};
		\node [style=none] (13) at (-9.75, 1.75) {$A$};
		\node [style=none] (14) at (-4.5, 1.75) {$B$};
		\node [style=box-1-.5] (15) at (-6, -1.25) {$h$};
		\node [style=none] (16) at (-10.25, -1.25) {};
		\node [style=none] (17) at (-4, -1.25) {};
		\node [style=none] (18) at (-9.75, -0.75) {$C$};
		\node [style=none] (19) at (-4.5, -0.75) {$D$};
		\node [style=box-1-.5] (20) at (8.25, 1.25) {$f$};
		\node [style=none] (21) at (4, 1.25) {};
		\node [style=none] (22) at (10.25, 1.25) {};
		\node [style=none] (23) at (4.5, 1.75) {$A$};
		\node [style=none] (24) at (9.75, 1.75) {$B$};
		\node [style=box-1-.5] (25) at (6, -1.25) {$h$};
		\node [style=none] (26) at (4, -1.25) {};
		\node [style=none] (27) at (10.25, -1.25) {};
		\node [style=none] (28) at (4.5, -0.75) {$C$};
		\node [style=none] (29) at (9.75, -0.75) {$D$};
		\node [style=none] (30) at (-3.25, 0) {=};
		\node [style=none] (31) at (3.25, 0) {=};
	\end{pgfonlayer}
	\begin{pgfonlayer}{edgelayer}
		\draw [style=thick-line] (1.center) to (0);
		\draw [style=thick-line] (0) to (2.center);
		\draw [style=thick-line] (6.center) to (5);
		\draw [style=thick-line] (5) to (7.center);
		\draw [style=thick-line] (11.center) to (10);
		\draw [style=thick-line] (10) to (12.center);
		\draw [style=thick-line] (16.center) to (15);
		\draw [style=thick-line] (15) to (17.center);
		\draw [style=thick-line] (21.center) to (20);
		\draw [style=thick-line] (20) to (22.center);
		\draw [style=thick-line] (26.center) to (25);
		\draw [style=thick-line] (25) to (27.center);
	\end{pgfonlayer}
\end{tikzpicture}}
            \label{fig:CatSliding}
        \end{equation}
        \item naturality of the swap map
        \begin{equation}
            \scalebox{0.7}{\begin{tikzpicture}
	\begin{pgfonlayer}{nodelayer}
		\node [style=box-1-.5] (0) at (-6.5, 1) {$f$};
		\node [style=none] (1) at (-8.5, 1) {};
		\node [style=none] (2) at (-4.5, 1) {};
		\node [style=none] (3) at (-8, 1.5) {$A$};
		\node [style=none] (4) at (-5, 1.5) {$B$};
		\node [style=box-1-.5] (5) at (-6.5, -1) {$h$};
		\node [style=none] (6) at (-8.5, -1) {};
		\node [style=none] (7) at (-4.5, -1) {};
		\node [style=none] (8) at (-8, -0.5) {$C$};
		\node [style=none] (9) at (-5, -0.5) {$D$};
		\node [style=none] (10) at (-2, 1) {};
		\node [style=none] (11) at (-1, 1) {};
		\node [style=none] (12) at (-2, -1) {};
		\node [style=none] (13) at (-1, -1) {};
		\node [style=none] (14) at (-1.5, 1.5) {$D$};
		\node [style=none] (15) at (-1.5, -0.5) {$B$};
		\node [style=none] (16) at (0, 0) {=};
		\node [style=box-1-.5] (17) at (6.5, -1) {$h$};
		\node [style=none] (19) at (8.5, -1) {};
		\node [style=none] (20) at (5, -0.5) {$A$};
		\node [style=none] (21) at (8, -0.5) {$B$};
		\node [style=box-1-.5] (22) at (6.5, 1) {$g$};
		\node [style=none] (24) at (8.5, 1) {};
		\node [style=none] (25) at (5, 1.5) {$C$};
		\node [style=none] (26) at (8, 1.5) {$D$};
		\node [style=none] (27) at (2, 1) {};
		\node [style=none] (28) at (2, -1) {};
		\node [style=none] (29) at (4.5, 1) {};
		\node [style=none] (30) at (4.5, -1) {};
		\node [style=none] (31) at (1, 1) {};
		\node [style=none] (32) at (1, -1) {};
		\node [style=none] (33) at (1.5, 1.5) {$A$};
		\node [style=none] (34) at (1.5, -0.5) {$C$};
	\end{pgfonlayer}
	\begin{pgfonlayer}{edgelayer}
		\draw [style=thick-line] (1.center) to (0);
		\draw [style=thick-line] (0) to (2.center);
		\draw [style=thick-line] (6.center) to (5);
		\draw [style=thick-line] (5) to (7.center);
		\draw [style=thick-line] (10.center) to (11.center);
		\draw [style=thick-line] (12.center) to (13.center);
		\draw [style=thick-line, in=-180, out=0, looseness=0.75] (7.center) to (10.center);
		\draw [style=thick-line, in=-180, out=0, looseness=0.75] (2.center) to (12.center);
		\draw [style=thick-line] (17) to (19.center);
		\draw [style=thick-line] (22) to (24.center);
		\draw [style=thick-line, in=-180, out=0, looseness=0.75] (28.center) to (29.center);
		\draw [style=thick-line, in=-180, out=0, looseness=0.75] (27.center) to (30.center);
		\draw [style=thick-line] (32.center) to (28.center);
		\draw [style=thick-line] (31.center) to (27.center);
		\draw [style=thick-line] (29.center) to (22);
		\draw [style=thick-line] (30.center) to (17);
	\end{pgfonlayer}
\end{tikzpicture}}
            \label{fig:CatNaturalitySwap}
        \end{equation}
    \end{itemize}
\end{definition}

\begin{definition}[Markov category]
    Markov categories (or affine cd-categories~\cite{cho2019disintegration}) are symmetric monoidal categories where every object is equipped with extra structure that reproduces, synthetically, the algebra of \textit{non-deterministic} processes resembling (normalised) Markov kernels. This structure includes \textit{copy} and \textit{delete} operations doing what their names explicitly suggest: the former creates two copies of its input, while the latter deletes its input:
    \begin{equation}
        \scalebox{0.8}{\begin{tikzpicture}
	\begin{pgfonlayer}{nodelayer}
		\node [style=none] (0) at (-5, 0) {};
		\node [style=none] (1) at (-1.75, 1) {};
		\node [style=none] (2) at (-1.75, -1) {};
		\node [style=none] (3) at (1.25, 0) {};
		\node [style=function black] (4) at (-3, 0) {};
		\node [style=function black] (5) at (3.5, 0) {};
		\node [style=none] (6) at (-4, 0.5) {$A$};
		\node [style=none] (7) at (-1.75, 1.5) {$A$};
		\node [style=none] (8) at (-1.75, -0.5) {$A$};
		\node [style=none] (9) at (2.5, 0.5) {$A$};
		\node [style=none] (10) at (-1.25, 1) {};
		\node [style=none] (11) at (-1.25, -1) {};
		\node [style=none] (12) at (0, 0) {,};
	\end{pgfonlayer}
	\begin{pgfonlayer}{edgelayer}
		\draw [style=thick-line] (0.center) to (4);
		\draw [style=thick-line, in=180, out=-75] (4) to (2.center);
		\draw [style=thick-line, in=-180, out=75] (4) to (1.center);
		\draw [style=thick-line] (3.center) to (5);
		\draw [style=thick-line] (2.center) to (11.center);
		\draw [style=thick-line] (1.center) to (10.center);
	\end{pgfonlayer}
\end{tikzpicture}}
        \label{eqn:CatCopyDelete}
    \end{equation}

    Copy and delete obey the following laws (for readability, strings aren't labelled here, but imagine them all having the same label):
    \begin{equation}
        \scalebox{0.8}{\begin{tikzpicture}
	\begin{pgfonlayer}{nodelayer}
		\node [style=function black] (0) at (-9.25, 1) {};
		\node [style=none] (2) at (-10.5, 1) {};
		\node [style=none] (5) at (-8, 0) {};
		\node [style=function black] (6) at (-9.25, 1) {};
		\node [style=none] (10) at (-6.75, 2.75) {};
		\node [style=function black] (13) at (-8, 2) {};
		\node [style=none] (14) at (-6.75, 1.25) {};
		\node [style=none] (15) at (-6.75, 0) {};
		\node [style=function black] (16) at (-3.5, 1.75) {};
		\node [style=none] (17) at (-4.75, 1.75) {};
		\node [style=none] (18) at (-2.25, 2.75) {};
		\node [style=function black] (19) at (-3.5, 1.75) {};
		\node [style=none] (20) at (-1, 0) {};
		\node [style=function black] (21) at (-2.25, 0.75) {};
		\node [style=none] (22) at (-1, 1.5) {};
		\node [style=none] (23) at (-1, 2.75) {};
		\node [style=none] (24) at (-5.75, 1.25) {=};
		\node [style=none] (26) at (1, 1.75) {};
		\node [style=none] (27) at (3.5, 0.75) {};
		\node [style=function black] (30) at (2.25, 1.75) {};
		\node [style=none] (31) at (3.5, 2.75) {};
		\node [style=none] (34) at (5.5, 1.75) {};
		\node [style=none] (35) at (7.5, 0.75) {};
		\node [style=function black] (37) at (6.75, 1.75) {};
		\node [style=none] (38) at (7.5, 2.75) {};
		\node [style=none] (40) at (9.25, 0.75) {};
		\node [style=none] (41) at (9.25, 2.75) {};
		\node [style=none] (42) at (4.5, 1.75) {=};
		\node [style=function black] (43) at (-5.25, -2.75) {};
		\node [style=none] (44) at (-6.5, -2.75) {};
		\node [style=function black] (45) at (-4, -3.75) {};
		\node [style=function black] (46) at (-5.25, -2.75) {};
		\node [style=function black] (47) at (-5.25, -2.75) {};
		\node [style=none] (48) at (-4, -1.75) {};
		\node [style=function black] (49) at (-5.25, -2.75) {};
		\node [style=none] (50) at (-3, -1.75) {};
		\node [style=function black] (51) at (2.75, -2.75) {};
		\node [style=none] (52) at (1.5, -2.75) {};
		\node [style=function black] (53) at (4, -1.75) {};
		\node [style=function black] (54) at (2.75, -2.75) {};
		\node [style=function black] (55) at (2.75, -2.75) {};
		\node [style=none] (56) at (4, -3.75) {};
		\node [style=function black] (57) at (2.75, -2.75) {};
		\node [style=none] (58) at (5, -3.75) {};
		\node [style=none] (59) at (-2.25, -2.75) {=};
		\node [style=none] (60) at (0.75, -2.75) {=};
		\node [style=none] (61) at (-1.5, -2.75) {};
		\node [style=none] (62) at (0, -2.75) {};
		\node [style=none] (69) at (0, 1.675) {,};
	\end{pgfonlayer}
	\begin{pgfonlayer}{edgelayer}
		\draw [style=thick-line] (2.center) to (0);
		\draw [style=thick-line, in=-180, out=-75] (0) to (5.center);
		\draw [style=thick-line, in=180, out=-75] (13) to (14.center);
		\draw [style=thick-line, in=-180, out=75] (6) to (13);
		\draw [style=thick-line, in=-180, out=75] (13) to (10.center);
		\draw [style=thick-line] (5.center) to (15.center);
		\draw [style=thick-line] (17.center) to (16);
		\draw [style=thick-line, in=180, out=75] (16) to (18.center);
		\draw [style=thick-line, in=-180, out=75] (21) to (22.center);
		\draw [style=thick-line, in=180, out=-75] (19) to (21);
		\draw [style=thick-line, in=180, out=-75] (21) to (20.center);
		\draw [style=thick-line] (18.center) to (23.center);
		\draw [style=thick-line, in=180, out=75] (30) to (31.center);
		\draw [style=thick-line, in=180, out=75] (37) to (38.center);
		\draw [style=thick-line, in=-180, out=0, looseness=0.75] (38.center) to (40.center);
		\draw [style=thick-line, in=-180, out=0, looseness=0.75] (35.center) to (41.center);
		\draw [style=thick-line] (44.center) to (43);
		\draw [style=thick-line, in=-180, out=-75] (43) to (45);
		\draw [style=thick-line, in=180, out=75] (47) to (48.center);
		\draw [style=thick-line] (48.center) to (50.center);
		\draw [style=thick-line] (52.center) to (51);
		\draw [style=thick-line, in=180, out=75] (51) to (53);
		\draw [style=thick-line, in=-180, out=-75] (55) to (56.center);
		\draw [style=thick-line] (56.center) to (58.center);
		\draw [style=thick-line] (61.center) to (62.center);
		\draw [style=thick-line] (26.center) to (30);
		\draw [style=thick-line, in=-180, out=-75] (30) to (27.center);
		\draw [style=thick-line] (34.center) to (37);
		\draw [style=thick-line, in=-180, out=-75] (37) to (35.center);
	\end{pgfonlayer}
\end{tikzpicture}}
    \end{equation}
\end{definition}

For extra clarity, we adopt the convention introduced by~\cite{virgo2021interpreting} to denote deterministic maps as boxes:
\begin{equation}
    \scalebox{0.8}{\begin{tikzpicture}
	\begin{pgfonlayer}{nodelayer}
		\node [style=box-1-.5] (0) at (0, 0) {$f$};
		\node [style=none] (1) at (-2.5, 0) {};
		\node [style=none] (2) at (2.5, 0) {};
		\node [style=none] (3) at (-1.75, 0.5) {$A$};
		\node [style=none] (4) at (1.75, 0.5) {$B$};
	\end{pgfonlayer}
	\begin{pgfonlayer}{edgelayer}
		\draw [style=thick-line] (1.center) to (0);
		\draw [style=thick-line] (0) to (2.center);
	\end{pgfonlayer}
\end{tikzpicture}}
\end{equation}
and draw general (possibly non-deterministic) ones as boxes with a curved right edge:
\begin{equation}
    \scalebox{0.8}{\begin{tikzpicture}
	\begin{pgfonlayer}{nodelayer}
		\node [style=none] (0) at (-2.5, 0) {};
		\node [style=none] (1) at (2.5, 0) {};
		\node [style=none] (2) at (-1.75, 0.5) {$A$};
		\node [style=none] (3) at (1.75, 0.5) {$B$};
		\node [style=none] (4) at (-1, 0.5) {};
		\node [style=none] (5) at (0, 0.5) {};
		\node [style=none] (6) at (-1, -0.5) {};
		\node [style=none] (7) at (0, -0.5) {};
		\node [style=none] (8) at (1, 0) {};
		\node [style=none] (9) at (0, 0) {$f$};
		\node [style=none] (10) at (-1, 0) {};
	\end{pgfonlayer}
	\begin{pgfonlayer}{edgelayer}
		\draw [style=thick-line] (4.center) to (5.center);
		\draw [style=thick-line] (4.center) to (6.center);
		\draw [style=thick-line] (6.center) to (7.center);
		\draw [style=thick-line, in=90, out=0] (5.center) to (8.center);
		\draw [style=thick-line, in=-90, out=0] (7.center) to (8.center);
		\draw [style=thick-line] (0.center) to (10.center);
		\draw [style=thick-line] (8.center) to (1.center);
	\end{pgfonlayer}
\end{tikzpicture}}
\end{equation}
In text, we write the type of a deterministic morphism as $f:A \to B$, whereas for a general (possibly non-deterministic) morphism we use the special notation $f:A\markto B$.

The defining property of deterministic maps in Markov categories is
\begin{equation}
    \scalebox{0.8}{\begin{tikzpicture}
	\begin{pgfonlayer}{nodelayer}
		\node [style=box-1-.5] (0) at (-5.5, 0) {$f$};
		\node [style=none] (1) at (-8, 0) {};
		\node [style=none] (3) at (-7.25, 0.5) {$A$};
		\node [style=none] (4) at (-3.75, 0.5) {$B$};
		\node [style=none] (12) at (1.25, 0) {};
		\node [style=none] (15) at (2, 0.5) {$A$};
		\node [style=box-1-.5] (21) at (5.75, 1) {$f$};
		\node [style=box-1-.5] (22) at (5.75, -1) {$f$};
		\node [style=none] (23) at (8.25, 1) {};
		\node [style=none] (24) at (8.25, -1) {};
		\node [style=none] (25) at (7.75, 1.75) {$B$};
		\node [style=none] (26) at (7.75, -0.5) {$B$};
		\node [style=none] (27) at (0, 0) {=};
		\node [style=none] (37) at (-1.75, 1) {};
		\node [style=none] (38) at (-1.75, -1) {};
		\node [style=function black] (39) at (-3, 0) {};
		\node [style=none] (40) at (-1.75, 1.5) {$B$};
		\node [style=none] (41) at (-1.75, -0.5) {$B$};
		\node [style=none] (42) at (-1.25, 1) {};
		\node [style=none] (43) at (-1.25, -1) {};
		\node [style=none] (44) at (4.25, 1) {};
		\node [style=none] (45) at (4.25, -1) {};
		\node [style=function black] (46) at (3, 0) {};
		\node [style=none] (47) at (4.25, 1.5) {$A$};
		\node [style=none] (48) at (4.25, -0.5) {$A$};
		\node [style=none] (49) at (4.75, 1) {};
		\node [style=none] (50) at (4.75, -1) {};
	\end{pgfonlayer}
	\begin{pgfonlayer}{edgelayer}
		\draw [style=thick-line] (1.center) to (0);
		\draw [style=thick-line] (22) to (24.center);
		\draw [style=thick-line] (21) to (23.center);
		\draw [style=thick-line, in=180, out=-75] (39) to (38.center);
		\draw [style=thick-line, in=-180, out=75] (39) to (37.center);
		\draw [style=thick-line] (38.center) to (43.center);
		\draw [style=thick-line] (37.center) to (42.center);
		\draw [style=thick-line] (0) to (39);
		\draw [style=thick-line, in=180, out=-75] (46) to (45.center);
		\draw [style=thick-line, in=-180, out=75] (46) to (44.center);
		\draw [style=thick-line] (45.center) to (50.center);
		\draw [style=thick-line] (44.center) to (49.center);
		\draw [style=thick-line] (12.center) to (46);
	\end{pgfonlayer}
\end{tikzpicture}}
\end{equation}
which is sometimes called the \textit{naturality of copy}.
Naturality doesn't apply to non-deterministic maps, intuitively because copying the result of a die roll is not the same as rolling two dice.
On the other hand, all maps in a Markov category satisfy the normalisation law, or \textit{naturality of delete}, stating that mapping an input to an output and then deleting the output, is the same as deleting the input:
\begin{equation}
    \label{eqn:normalisation}
    \scalebox{0.8}{\begin{tikzpicture}
	\begin{pgfonlayer}{nodelayer}
		\node [style=none] (0) at (-4.5, 0) {};
		\node [style=function black] (2) at (0.5, 0) {};
		\node [style=none] (3) at (1.5, 0) {=};
		\node [style=none] (4) at (2.25, 0) {};
		\node [style=function black] (5) at (3.75, 0) {};
		\node [style=none] (6) at (-3.75, 0.5) {$A$};
		\node [style=none] (7) at (3, 0.5) {$A$};
		\node [style=none] (8) at (-0.25, 0.5) {$B$};
		\node [style=none] (9) at (-3, 0.5) {};
		\node [style=none] (10) at (-2, 0.5) {};
		\node [style=none] (11) at (-3, -0.5) {};
		\node [style=none] (12) at (-2, -0.5) {};
		\node [style=none] (13) at (-1, 0) {};
		\node [style=none] (14) at (-2, 0) {$f$};
		\node [style=none] (15) at (-3, 0) {};
	\end{pgfonlayer}
	\begin{pgfonlayer}{edgelayer}
		\draw [style=thick-line] (4.center) to (5);
		\draw [style=thick-line] (9.center) to (10.center);
		\draw [style=thick-line] (9.center) to (11.center);
		\draw [style=thick-line] (11.center) to (12.center);
		\draw [style=thick-line, in=90, out=0] (10.center) to (13.center);
		\draw [style=thick-line, in=-90, out=0] (12.center) to (13.center);
		\draw [style=thick-line] (0.center) to (15.center);
		\draw [style=thick-line] (13.center) to (2);
	\end{pgfonlayer}
\end{tikzpicture}}
\end{equation}

For further technical details, and discussions, we refer the reader to~\cite{fritz2020synthetic, perrone2023markov}.

\begin{example}[$\FinStoch$~\cite{fritz2020synthetic}]
    For a given set $Y$, denote by $\prob(Y)$ the set of all finitely-supported probability distributions over it.
    There is a Markov category $\FinStoch$ whose objects are finite sets and whose morphisms from $X$ to $Y$, $X \markto Y$, called \textit{Markov kernels}, are defined as functions $X \to \prob(Y)$.
    For a Markov kernel $f:X\markto Y$ we write $f(y \mid x)$ for $f(x)(y)$, the probability assigned to $y$ when the kernel is given $x$ as input.

    Sequential composition of Markov kernels $f: X \markto Y$, $g: Y \markto Z$ is given by the Markov kernel so defined (i.e.~the Chapman--Kolmogorov equation):
    \begin{align}
        f \comp g : X & \markto Z\nonumber \\
        (f\comp g)(z \mid x) & \coloneqq \sum_{y \in Y} g(z \mid y) f(y \mid x)
    \end{align}
    while parallel composition is, on objects, the Cartesian product of sets $X \otimes Y \coloneqq X \times Y$, and for morphisms $f : X \markto Y$ and $f' : M \markto Y'$ given by the Markov kernel so defined:
    \begin{align}
        f \otimes f' :X \times X' & \markto Y \times Y'
        \nonumber \\
        (f\otimes f')(y,y'\mid x,x') & \coloneqq f(y\mid x) f'(y'\mid x').
    \end{align}

    The unit for parallel composition is the singleton $\Unit = 1 = \{\ast\}$.
    The copy and delete morphisms for a set $X$ are given, respectively, by the Markov kernel $\Copy_X: X \to X \times X$ mapping $x \in X$ to the Dirac distribution $\delta_{(x,x)}$, and the unique map $\del_X : X \to 1$, mapping $x\in X$ to the only normalised probability distribution that exists over $\{*\}$.

    Deterministic morphisms in this Markov category are those that are indeed deterministic, i.e.~they map every element of their domain to a Dirac distribution. In other words, the deterministic maps of $\FinStoch$ are all and only the usual set-theoretic functions, see~\cite{fritz2020synthetic} for details.
\end{example}

The following example will be used in our results:

\begin{example}[$\RelPlus$ (or $\SetMulti$)~\cite{fritz2020synthetic, stein2023probabilistic}]
    For a set $Y$, denote by $\possplus(Y)$ the set of all non-empty subsets of $Y$.
    There is a Markov category $\RelPlus$ whose objects are sets and whose morphisms from $X$ to $Y$ are functions $X \to \possplus(Y)$, corresponding to \textit{left-total relations}.
    Sequential composition of left-total relations $f: X \markto Y$, $g: Y \markto Z$ is given by the left-total relation so defined:
    \begin{align}
    \label{eq:relpluscomp}
        f \comp g: X &\longrightarrow \possplus(Z) \nonumber \\
        x &\longmapsto \{ z \in Z \mid \exists y \in Y, y \in f(x) \text{ and } z \in g(y)\},
    \end{align}
    while parallel composition is, on objects, the Cartesian product of sets $X \otimes Y \coloneqq X \times Y$, and for morphisms $f : X \markto Y$ and $f' : M \markto Y'$ given by the left-total relation so defined:
    \begin{align}
        f \otimes f' : X \times M \longrightarrow & \possplus(Y \times Y') \nonumber \\
        (x,x') \longmapsto & \{(y,y') \in Y \times Y' \mid  \nonumber \\& \qquad y \in f(x) \text{ and }
        y' \in f'(x')\}.
    \end{align}
    The unit for parallel composition is the singleton $\Unit = 1 = \{\ast\}$.

    The copy and delete morphisms for a set $X$ are given, respectively, by the left-total relation $\Copy_X: X \to X \times X$ mapping $x \in X$ to the singleton $\{(x,x)\}$, and the unique map $\del_X : X \to \possplus(1)$.

    Deterministic morphisms in this Markov category are those that map every element of their domain to a singleton. In other words, the deterministic maps of $\RelPlus$ are all and only the usual set-theoretic functions~\cite{fritz2020synthetic}.
\end{example}

The morphisms in $\RelPlus$ can be seen as \textit{possibilistic} Markov kernels.
For each element in the codomain we specify a set of possible elements of the domain that it might map to, but we don't assign probabilities to these possibilities.
Possibilistic representations of uncertainty are less common than probabilistic ones. Nonetheless they are well established and have a long history in different fields, including control theory~\cite{byrnes2003limit, rawlings2017model, aubin2011viability, tanwani2018well, brogliato2020dynamical}, artificial intelligence~\cite{denoeux2010representing} and automata theory~\cite{rabin1959finite}.

\subsection{Bayesian reasoning and interpretations in Markov categories}

We now introduce the main ideas involved in the notion of Bayesian interpretations.
To aid intuition, we refer to maps of type $\prior: \Unit \markto \Hiddenstate$ for any object $\Hiddenstate$ as \textit{distributions} and to deterministic maps $\hiddenstate : \Unit \to \Hiddenstate$ of that type as \textit{elements} of $\Hiddenstate$.
We also consider maps $\belief: \Parameter \markto \Hiddenstate$ to correspond to \textit{parametrised families of distributions} (also called ``channels'' in~\cite{cho2019disintegration, jacobs2020channel}), with each element $\parameter: \Unit \to \Parameter$ a parameter determining a distribution $\belief(\Hiddenstate \mid \parameter)$ or $\belief(\parameter)$ over $\Hiddenstate$.

Furthermore, given a map $\likelihood: \Hiddenstate \markto \Observation$ and an element $x: \Unit \to \Observation$ we call  $\hiddenstate \comp \likelihood : \Unit \markto \Observation$ the distribution over $\Observation$ assigned to the element $\hiddenstate \in \Hiddenstate$ by $\likelihood$.
We write it as $\likelihood(\hiddenstate)$ or $\likelihood (\Observation \mid \hiddenstate)$.
In $\FinStoch$ this terminology and the notation coincide with its common usage in probability theory.

Next we recall the definitions of Bayesian inversion and conjugate prior for Markov categories with conditionals given, e.g. in~\cite{jacobs2020channel, fritz2020synthetic}.
We then generalise these two concepts, which brings us closer to the definition of a Bayesian interpretation.

\begin{definition}[Bayesian inversion]
    \label{def:bayesInference}
    In a Markov category, a Bayesian inversion~\cite{cho2019disintegration} of a map $\likelihood : \Hiddenstate \markto \Observation$ with respect to a distribution $\prior : \Unit \markto \Hiddenstate$ is a map $\posterior : \Observation \markto \Hiddenstate$ satisfying the following equation:
    \begin{equation}
        \label{eqn:bayesInference}
        \scalebox{0.8}{\begin{tikzpicture}
	\begin{pgfonlayer}{nodelayer}
		\node [style=none] (0) at (0.25, 3.5) {};
		\node [style=none] (1) at (1.25, 3.5) {};
		\node [style=none] (2) at (0.25, 2.5) {};
		\node [style=none] (3) at (1.25, 2.5) {};
		\node [style=none] (4) at (2.25, 3) {};
		\node [style=none] (5) at (1.25, 3) {$\likelihood$};
		\node [style=none] (6) at (0.25, 3) {};
		\node [style=function black] (7) at (-0.75, 2) {};
		\node [style=none] (9) at (-1.5, 2.5) {$\Hiddenstate$};
		\node [style=none] (10) at (0.25, 1) {};
		\node [style=none] (11) at (3.25, 1) {};
		\node [style=none] (12) at (3.25, 3) {};
		\node [style=none] (13) at (2.75, 3.5) {$\Observation$};
		\node [style=none] (14) at (2.75, 1.5) {$\Hiddenstate$};
		\node [style=none] (15) at (3.5, -3.275) {};
		\node [style=none] (16) at (4.5, -3.275) {};
		\node [style=none] (17) at (3.5, -2.275) {};
		\node [style=none] (18) at (4.5, -2.275) {};
		\node [style=none] (19) at (5.5, -2.775) {};
		\node [style=none] (20) at (4.5, -2.775) {$\posterior$};
		\node [style=none] (21) at (3.5, -2.75) {};
		\node [style=function black] (22) at (2.5, -1.75) {};
		\node [style=none] (24) at (1.75, -1.25) {$\Observation$};
		\node [style=none] (25) at (3.5, -0.75) {};
		\node [style=none] (26) at (6.5, -0.75) {};
		\node [style=none] (27) at (6.5, -2.775) {};
		\node [style=none] (28) at (6, -2.275) {$\Hiddenstate$};
		\node [style=none] (29) at (6, -0.25) {$\Observation$};
		\node [style=none] (30) at (-0.75, -1.25) {};
		\node [style=none] (31) at (0.25, -1.25) {};
		\node [style=none] (32) at (-0.75, -2.25) {};
		\node [style=none] (33) at (0.25, -2.25) {};
		\node [style=none] (34) at (1.25, -1.75) {};
		\node [style=none] (35) at (0.25, -1.75) {$\likelihood$};
		\node [style=none] (36) at (-0.75, -1.75) {};
		\node [style=none] (38) at (-1.5, -1.25) {$\Hiddenstate$};
		\node [style=none] (39) at (-4.25, 2.5) {};
		\node [style=none] (40) at (-3.25, 2.5) {};
		\node [style=none] (41) at (-4.25, 1.5) {};
		\node [style=none] (42) at (-3.25, 1.5) {};
		\node [style=none] (43) at (-2.25, 2) {};
		\node [style=none] (44) at (-3.25, 2) {$\prior$};
		\node [style=none] (45) at (-4.25, 2) {};
		\node [style=none] (48) at (-4.25, -1.25) {};
		\node [style=none] (49) at (-3.25, -1.25) {};
		\node [style=none] (50) at (-4.25, -2.25) {};
		\node [style=none] (51) at (-3.25, -2.25) {};
		\node [style=none] (52) at (-2.25, -1.75) {};
		\node [style=none] (53) at (-3.25, -1.75) {$\prior$};
		\node [style=none] (54) at (-4.25, -1.75) {};
		\node [style=none] (60) at (4.75, 2) {=};
	\end{pgfonlayer}
	\begin{pgfonlayer}{edgelayer}
		\draw [style=thick-line] (0.center) to (1.center);
		\draw [style=thick-line] (0.center) to (2.center);
		\draw [style=thick-line] (2.center) to (3.center);
		\draw [style=thick-line, in=90, out=0] (1.center) to (4.center);
		\draw [style=thick-line, in=-90, out=0] (3.center) to (4.center);
		\draw [style=thick-line, in=180, out=75] (7) to (6.center);
		\draw [style=thick-line, in=180, out=-75] (7) to (10.center);
		\draw [style=thick-line] (10.center) to (11.center);
		\draw [style=thick-line] (4.center) to (12.center);
		\draw [style=thick-line] (15.center) to (16.center);
		\draw [style=thick-line] (15.center) to (17.center);
		\draw [style=thick-line] (17.center) to (18.center);
		\draw [style=thick-line, in=-90, out=0] (16.center) to (19.center);
		\draw [style=thick-line, in=90, out=0] (18.center) to (19.center);
		\draw [style=thick-line, in=-180, out=-75] (22) to (21.center);
		\draw [style=thick-line, in=-180, out=75] (22) to (25.center);
		\draw [style=thick-line] (25.center) to (26.center);
		\draw [style=thick-line] (19.center) to (27.center);
		\draw [style=thick-line] (30.center) to (31.center);
		\draw [style=thick-line] (30.center) to (32.center);
		\draw [style=thick-line] (32.center) to (33.center);
		\draw [style=thick-line, in=90, out=0] (31.center) to (34.center);
		\draw [style=thick-line, in=-90, out=0] (33.center) to (34.center);
		\draw [style=thick-line] (34.center) to (22);
		\draw [style=thick-line] (39.center) to (40.center);
		\draw [style=thick-line] (39.center) to (41.center);
		\draw [style=thick-line] (41.center) to (42.center);
		\draw [style=thick-line, in=90, out=0] (40.center) to (43.center);
		\draw [style=thick-line, in=-90, out=0] (42.center) to (43.center);
		\draw [style=thick-line] (43.center) to (7);
		\draw [style=thick-line] (48.center) to (49.center);
		\draw [style=thick-line] (48.center) to (50.center);
		\draw [style=thick-line] (50.center) to (51.center);
		\draw [style=thick-line, in=90, out=0] (49.center) to (52.center);
		\draw [style=thick-line, in=-90, out=0] (51.center) to (52.center);
		\draw [style=thick-line] (52.center) to (36.center);
	\end{pgfonlayer}
\end{tikzpicture}}
    \end{equation}
\end{definition}
In a general Markov category with $\likelihood$ and $\prior$ as above, a Bayesian inverse $\posterior$ is not guaranteed to exist.
However, many Markov of interest have a property known as ``having (all) conditionals''~\cite[Definition 11.5]{fritz2020synthetic}, which guarantees that they always exist, for every $\likelihood$ and $\prior$. This is the case in both $\FinStoch$ and $\RelPlus$.
However, although they always exist in our categories of interest, Bayesian inverses are not generally unique, meaning that for given $p$ and $f$ there might be multiple morphisms $f^\dagger$ satisfying \cref{eqn:bayesInference}. We discuss the reasons for this shortly.

To get a more concrete feeling for Bayesian inverses, we can examine the form \cref{eqn:bayesInference} takes when our Markov category is $\FinStoch$, where it corresponds to the equation
\begin{align}
    \label{eqn:bayesianInverseDefinitionInFinstoch}
    \prior(\hiddenstate) \likelihood(\observation \mid \hiddenstate)
    & = \likelihood (\observation \mid \prior(\hiddenstate)) \posterior(\hiddenstate \mid \observation) \nonumber \\
    & = \sum_{\hiddenstate' \in \Hiddenstate} \prior(\hiddenstate') \likelihood(\observation \mid \hiddenstate') \posterior(\hiddenstate \mid \observation),
\end{align}
from which we can derive the standard Bayes rule by dividing by $\likelihood (\observation \mid \prior(\hiddenstate))$, assuming it's positive ($\likelihood (\observation \mid \prior(\hiddenstate)) > 0)$:
\begin{align}
    \posterior(\hiddenstate \mid \observation)
    & = \frac{\prior(\hiddenstate) \likelihood(\observation \mid \hiddenstate)}{\likelihood (\observation \mid \prior(\hiddenstate))} \nonumber \\
    & = \frac{\prior(\hiddenstate) \likelihood(\observation \mid \hiddenstate)}{\sum_{\hiddenstate' \in \Hiddenstate} \prior(\hiddenstate') \likelihood(\observation \mid \hiddenstate')}.
\end{align}
As mentioned, the Bayesian inverse $\posterior$ is generally not unique.
In $\FinStoch$ this is because \cref{eqn:bayesianInverseDefinitionInFinstoch} doesn't constrain the value of $\posterior(\hiddenstate \mid \observation)$ in cases where $\likelihood (\observation \mid \prior(\hiddenstate)) = 0$.

Bayesian inversions can be used to inspire a notion of \emph{updating} of distributions over a (hidden) variable in response to observations~\cite{caticha2021entropy}.
For this, we consider $\Observation$ as observations generated from hidden variable $\Hiddenstate$ by $\likelihood : \Hiddenstate \markto \Observation$ (often called a \textit{statistical model} in a probabilistic context, including for instance in $\FinStoch$), the distribution $\prior : \Unit \markto \Hiddenstate$ above as a \textit{prior} distribution before an observation, and the map $\posterior : \Observation \markto \Hiddenstate$ as assigning new, \textit{posterior} distributions to the observations $\Observation$ generated via $\likelihood$.
In the case of $\FinStoch$, the map $f^\dagger$ can be seen as multiplying the prior by the \textit{likelihood} (i.e. the values of $f(y\mid x)$ for the given \textit{data} $y$) and then dividing by the \textit{evidence} $\sum_{\hiddenstate' \in \Hiddenstate} \prior(\hiddenstate') \likelihood(\observation \mid \hiddenstate')$ to obtain the posterior.

It is common to refer to distributions that are updated according to this process as \textit{(Bayesian) beliefs}, and to the process itself as \textit{(Bayesian) belief updating}.

In some cases, in place of the prior $\prior: \Unit \markto \Hiddenstate$ we might want a \textit{parametrised family} of priors $\belief: \Parameter \markto \Hiddenstate$.
The idea is that the prior is assumed to come from some family of distributions, say for instance Gaussians, and so we consider one prior for each value of the parameters $\Parameter$.
In this case, the Bayesian inverse $\posterior$ must depend on the parameter as well, because in general the Bayesian inverse depends on the prior.
This gives rise to the following definition.
\begin{definition}
    \label{def:parametrisedBayes}
    In a Markov category, we say that a map $\posterior : \Observation \otimes \Parameter \markto \Hiddenstate$ is a Bayesian inversion of a map $\likelihood : \Hiddenstate \markto \Observation$ with respect to a parametrised family $\belief: \Theta \markto \Hiddenstate$ if:
    \begin{equation}
        \label{eqn:bayesWithInputs}
        \scalebox{0.8}{\begin{tikzpicture}
	\begin{pgfonlayer}{nodelayer}
		\node [style=none] (0) at (13.75, 6.25) {};
		\node [style=none] (1) at (14.75, 6.25) {};
		\node [style=none] (2) at (13.75, 5.25) {};
		\node [style=none] (3) at (14.75, 5.25) {};
		\node [style=none] (4) at (15.75, 5.75) {};
		\node [style=none] (5) at (14.75, 5.75) {$f$};
		\node [style=none] (6) at (13.75, 5.75) {};
		\node [style=function black] (7) at (12.75, 4.75) {};
		\node [style=none] (9) at (12, 5.25) {$X$};
		\node [style=none] (10) at (13.75, 3.75) {};
		\node [style=none] (11) at (16.75, 3.75) {};
		\node [style=none] (12) at (16.75, 5.75) {};
		\node [style=none] (13) at (16.25, 6.25) {$Y$};
		\node [style=none] (14) at (16.25, 4.25) {$X$};
		\node [style=none] (15) at (17, -0.625) {};
		\node [style=none] (16) at (18, -0.625) {};
		\node [style=none] (17) at (17, 0.375) {};
		\node [style=none] (18) at (18, 0.375) {};
		\node [style=none] (19) at (19, -0.125) {};
		\node [style=none] (20) at (18, -0.125) {$f^\dagger$};
		\node [style=none] (21) at (17, 0) {};
		\node [style=function black] (22) at (16, 1) {};
		\node [style=none] (24) at (15.25, 1.5) {$Y$};
		\node [style=none] (25) at (17, 2) {};
		\node [style=none] (26) at (20, 2) {};
		\node [style=none] (27) at (20, -0.125) {};
		\node [style=none] (28) at (19.5, 0.375) {$X$};
		\node [style=none] (29) at (19.5, 2.5) {$Y$};
		\node [style=none] (30) at (12.75, 1.5) {};
		\node [style=none] (31) at (13.75, 1.5) {};
		\node [style=none] (32) at (12.75, 0.5) {};
		\node [style=none] (33) at (13.75, 0.5) {};
		\node [style=none] (34) at (14.75, 1) {};
		\node [style=none] (35) at (13.75, 1) {$f$};
		\node [style=none] (36) at (12.75, 1) {};
		\node [style=none] (38) at (12, 1.5) {$X$};
		\node [style=none] (39) at (9.25, 5.25) {};
		\node [style=none] (40) at (10.25, 5.25) {};
		\node [style=none] (41) at (9.25, 4.25) {};
		\node [style=none] (42) at (10.25, 4.25) {};
		\node [style=none] (43) at (11.25, 4.75) {};
		\node [style=none] (44) at (10.25, 4.75) {$\belief$};
		\node [style=none] (45) at (9.25, 4.75) {};
		\node [style=none] (46) at (7.75, 4.75) {};
		\node [style=none] (47) at (8.25, 5.25) {$\Theta$};
		\node [style=none] (48) at (9.25, 1.5) {};
		\node [style=none] (49) at (10.25, 1.5) {};
		\node [style=none] (50) at (9.25, 0.5) {};
		\node [style=none] (51) at (10.25, 0.5) {};
		\node [style=none] (52) at (11.25, 1) {};
		\node [style=none] (53) at (10.25, 1) {$\belief$};
		\node [style=none] (54) at (9.25, 1) {};
		\node [style=none] (55) at (7.75, 1) {};
		\node [style=none] (56) at (8.25, 1.5) {$\Theta$};
		\node [style=function black] (57) at (8.5, 1) {};
		\node [style=none] (58) at (9.25, -0.25) {};
		\node [style=none] (59) at (17, -0.25) {};
		\node [style=none] (60) at (18.25, 4.75) {=};
	\end{pgfonlayer}
	\begin{pgfonlayer}{edgelayer}
		\draw [style=thick-line] (0.center) to (1.center);
		\draw [style=thick-line] (0.center) to (2.center);
		\draw [style=thick-line] (2.center) to (3.center);
		\draw [style=thick-line, in=90, out=0] (1.center) to (4.center);
		\draw [style=thick-line, in=-90, out=0] (3.center) to (4.center);
		\draw [style=thick-line, in=180, out=75] (7) to (6.center);
		\draw [style=thick-line, in=180, out=-75] (7) to (10.center);
		\draw [style=thick-line] (10.center) to (11.center);
		\draw [style=thick-line] (4.center) to (12.center);
		\draw [style=thick-line] (15.center) to (16.center);
		\draw [style=thick-line] (15.center) to (17.center);
		\draw [style=thick-line] (17.center) to (18.center);
		\draw [style=thick-line, in=-90, out=0] (16.center) to (19.center);
		\draw [style=thick-line, in=90, out=0] (18.center) to (19.center);
		\draw [style=thick-line, in=-180, out=-75] (22) to (21.center);
		\draw [style=thick-line, in=-180, out=75] (22) to (25.center);
		\draw [style=thick-line] (25.center) to (26.center);
		\draw [style=thick-line] (19.center) to (27.center);
		\draw [style=thick-line] (30.center) to (31.center);
		\draw [style=thick-line] (30.center) to (32.center);
		\draw [style=thick-line] (32.center) to (33.center);
		\draw [style=thick-line, in=90, out=0] (31.center) to (34.center);
		\draw [style=thick-line, in=-90, out=0] (33.center) to (34.center);
		\draw [style=thick-line] (34.center) to (22);
		\draw [style=thick-line] (39.center) to (40.center);
		\draw [style=thick-line] (39.center) to (41.center);
		\draw [style=thick-line] (41.center) to (42.center);
		\draw [style=thick-line, in=90, out=0] (40.center) to (43.center);
		\draw [style=thick-line, in=-90, out=0] (42.center) to (43.center);
		\draw [style=thick-line] (43.center) to (7);
		\draw [style=thick-line] (46.center) to (45.center);
		\draw [style=thick-line] (48.center) to (49.center);
		\draw [style=thick-line] (48.center) to (50.center);
		\draw [style=thick-line] (50.center) to (51.center);
		\draw [style=thick-line, in=90, out=0] (49.center) to (52.center);
		\draw [style=thick-line, in=-90, out=0] (51.center) to (52.center);
		\draw [style=thick-line] (55.center) to (54.center);
		\draw [style=thick-line] (52.center) to (36.center);
		\draw [style=thick-line] (58.center) to (59.center);
		\draw [style=thick-line, in=165, out=-90, looseness=1.25] (57) to (58.center);
	\end{pgfonlayer}
\end{tikzpicture}}
    \end{equation}
\end{definition}

The next concept, called a \textit{conjugate prior}, expresses an important property that such a parametrised Bayesian inverse might have, namely that it factors through the original family of distributions $\belief$.
This is of particular importance to us because the map $\conj$ below can be seen as explicitly implementing the belief update, by updating the parameters such that the prior's parameter is updated to the posterior's parameter.
Conjugate priors were first expressed in string diagrams in this form in~\cite{jacobs2020channel}.

\begin{definition}[Conjugate prior for Bayesian inference]
    \label{def:conjugatePriorInference}
    We call $\belief$ a conjugate prior~\cite{jacobs2020channel} to $\likelihood: \Hiddenstate \markto \Observation$ if there exists a deterministic map $\conj: \Observation \otimes \Parameter \to \Parameter$ that satisfies the following:
    \begin{equation}
        \scalebox{0.8}{\begin{tikzpicture}
	\begin{pgfonlayer}{nodelayer}
		\node [style=none] (0) at (-0.25, 4) {};
		\node [style=none] (1) at (0.75, 4) {};
		\node [style=none] (2) at (-0.25, 3) {};
		\node [style=none] (3) at (0.75, 3) {};
		\node [style=none] (4) at (1.75, 3.5) {};
		\node [style=none] (5) at (0.75, 3.5) {$\likelihood$};
		\node [style=none] (6) at (-0.25, 3.5) {};
		\node [style=function black] (7) at (-1.25, 2.5) {};
		\node [style=none] (8) at (-2, 3) {$\Hiddenstate$};
		\node [style=none] (9) at (-0.25, 1.5) {};
		\node [style=none] (10) at (2.75, 1.5) {};
		\node [style=none] (11) at (2.75, 3.5) {};
		\node [style=none] (12) at (2.25, 4) {$\Observation$};
		\node [style=none] (13) at (2.25, 2) {$\Hiddenstate$};
		\node [style=none] (41) at (-4.75, 3) {};
		\node [style=none] (42) at (-3.75, 3) {};
		\node [style=none] (43) at (-4.75, 2) {};
		\node [style=none] (44) at (-3.75, 2) {};
		\node [style=none] (45) at (-2.75, 2.5) {};
		\node [style=none] (46) at (-3.75, 2.5) {$\belief$};
		\node [style=none] (48) at (-4.75, 2.5) {};
		\node [style=none] (49) at (-6.25, 2.5) {};
		\node [style=none] (50) at (-5.75, 3) {$\Parameter$};
		\node [style=none] (51) at (4.25, -2.75) {};
		\node [style=none] (52) at (5.5, -2.75) {};
		\node [style=none] (53) at (4.25, -1.75) {};
		\node [style=none] (54) at (5.25, -1.75) {};
		\node [style=none] (55) at (6.25, -2.25) {};
		\node [style=none] (56) at (5.25, -2.25) {$\belief$};
		\node [style=none] (57) at (2.25, -2.25) {};
		\node [style=function black] (58) at (1.25, -1.25) {};
		\node [style=none] (59) at (-2.25, -0.75) {$\Hiddenstate$};
		\node [style=none] (60) at (2.25, -0.25) {};
		\node [style=none] (61) at (7.25, -0.25) {};
		\node [style=none] (62) at (7.25, -2.25) {};
		\node [style=none] (63) at (6.75, 0.25) {$\Observation$};
		\node [style=none] (64) at (6.75, -1.75) {$\Hiddenstate$};
		\node [style=none] (65) at (-4.75, -0.75) {};
		\node [style=none] (66) at (-3.75, -0.75) {};
		\node [style=none] (67) at (-4.75, -1.75) {};
		\node [style=none] (68) at (-3.75, -1.75) {};
		\node [style=none] (69) at (-2.75, -1.25) {};
		\node [style=none] (70) at (-3.75, -1.25) {$\belief$};
		\node [style=none] (71) at (-4.75, -1.25) {};
		\node [style=none] (72) at (-6.25, -1.25) {};
		\node [style=none] (73) at (-5.75, -0.75) {$\Parameter$};
		\node [style=box-.5-.5] (74) at (2.75, -2.25) {$\conj$};
		\node [style=none] (75) at (4.25, -2.25) {};
		\node [style=none] (76) at (-1.75, -0.75) {};
		\node [style=none] (77) at (-0.75, -0.75) {};
		\node [style=none] (78) at (-1.75, -1.75) {};
		\node [style=none] (79) at (-0.75, -1.75) {};
		\node [style=none] (80) at (0.25, -1.25) {};
		\node [style=none] (81) at (-0.75, -1.25) {$\likelihood$};
		\node [style=none] (82) at (-1.75, -1.25) {};
		\node [style=none] (83) at (0.75, -0.75) {$\Observation$};
		\node [style=none] (84) at (5.25, 2.5) {=};
		\node [style=none] (85) at (3.75, -1.75) {$\Parameter$};
		\node [style=function black] (86) at (-5.5, -1.25) {};
		\node [style=none] (87) at (-4.75, -2.5) {};
		\node [style=none] (88) at (2.25, -2.5) {};
	\end{pgfonlayer}
	\begin{pgfonlayer}{edgelayer}
		\draw [style=thick-line] (0.center) to (1.center);
		\draw [style=thick-line] (0.center) to (2.center);
		\draw [style=thick-line] (2.center) to (3.center);
		\draw [style=thick-line, in=90, out=0] (1.center) to (4.center);
		\draw [style=thick-line, in=-90, out=0] (3.center) to (4.center);
		\draw [style=thick-line, in=180, out=75] (7) to (6.center);
		\draw [style=thick-line, in=180, out=-75] (7) to (9.center);
		\draw [style=thick-line] (9.center) to (10.center);
		\draw [style=thick-line] (4.center) to (11.center);
		\draw [style=thick-line] (41.center) to (42.center);
		\draw [style=thick-line] (41.center) to (43.center);
		\draw [style=thick-line] (43.center) to (44.center);
		\draw [style=thick-line, in=90, out=0] (42.center) to (45.center);
		\draw [style=thick-line, in=-90, out=0] (44.center) to (45.center);
		\draw [style=thick-line] (45.center) to (7);
		\draw [style=thick-line] (49.center) to (48.center);
		\draw [style=thick-line] (51.center) to (52.center);
		\draw [style=thick-line] (51.center) to (53.center);
		\draw [style=thick-line] (53.center) to (54.center);
		\draw [style=thick-line, in=-90, out=0] (52.center) to (55.center);
		\draw [style=thick-line, in=90, out=0] (54.center) to (55.center);
		\draw [style=thick-line, in=-180, out=-75] (58) to (57.center);
		\draw [style=thick-line, in=-180, out=75] (58) to (60.center);
		\draw [style=thick-line] (60.center) to (61.center);
		\draw [style=thick-line] (55.center) to (62.center);
		\draw [style=thick-line] (65.center) to (66.center);
		\draw [style=thick-line] (65.center) to (67.center);
		\draw [style=thick-line] (67.center) to (68.center);
		\draw [style=thick-line, in=90, out=0] (66.center) to (69.center);
		\draw [style=thick-line, in=-90, out=0] (68.center) to (69.center);
		\draw [style=thick-line] (72.center) to (71.center);
		\draw [style=thick-line] (74) to (75.center);
		\draw [style=thick-line] (76.center) to (77.center);
		\draw [style=thick-line] (76.center) to (78.center);
		\draw [style=thick-line] (78.center) to (79.center);
		\draw [style=thick-line, in=90, out=0] (77.center) to (80.center);
		\draw [style=thick-line, in=-90, out=0] (79.center) to (80.center);
		\draw [style=thick-line] (69.center) to (82.center);
		\draw [style=thick-line] (80.center) to (58);
		\draw [style=thick-line] (87.center) to (88.center);
		\draw [style=thick-line, in=180, out=-90] (86) to (87.center);
	\end{pgfonlayer}
\end{tikzpicture}}
        \label{eqn:conjugatePriorInference}
    \end{equation}
\end{definition}
Note that the map $\conj \comp \belief$ is a Bayesian inverse of $\likelihood$ with respect to $\belief$, in the sense of \cref{def:parametrisedBayes}.
In Bayesian statistics, $\parameter\in\Parameter$ is sometimes called a \textit{hyperparameter} and $\hiddenstate\in \Hiddenstate$ a \textit{parameter}, but we will generally avoid this usage.

If we have a conjugate prior $\belief: \Parameter \markto \Hiddenstate$ to $\likelihood : \Hiddenstate \markto \Observation$, then every parameter $\parameter: \Unit \to \Parameter$ determines a Bayesian inversion $\belief(\conj(-,\parameter)): \Observation \markto \Hiddenstate$ of $\likelihood : \Hiddenstate \markto \Observation$ for $\prior(\parameter)$~\cite[Theorem 6.3]{jacobs2020channel}.
The role of $\conj$ is to turn parameters of prior distributions into parameters of posterior distributions by taking into account observations generated by $\likelihood$ according to the prior $\belief$.
In this way, the map $\conj$ implements belief updating. The usual informal notion of a prior being conjugate to $\likelihood$ if the posterior distribution is in the same family as the prior's is captured by having the map $\belief$ appear twice: once when mapping the initial parameter to the prior and again when mapping the updated parameter to the posterior.

The version of belief updating introduced so far can be seen as a version of \textit{Bayesian inference}, which means that the updated beliefs are about a constant hidden variable.
To see this, note that $\Hiddenstate$ gets copied, but never gets changed in \cref{eqn:bayesInference,eqn:bayesWithInputs,eqn:conjugatePriorInference}.

We need then a form of belief updating that corresponds to \textit{Bayesian filtering}, where the hidden variable also changes.
For this we can replace $\Copy_\Hiddenstate \comp \likelihood \otimes \id_\Hiddenstate: \Hiddenstate \markto \Observation \otimes \Hiddenstate$ with a map $\hiddenmodel : \Hiddenstate \markto \Hiddenstate \otimes \Observation$ that produces an observation $\Observation$ and may change $\Hiddenstate$ instead of just copying it.
The analogue of the Bayesian inversion (\cref{def:bayesInference}) with respect to a distribution $\prior: \Unit \markto \Hiddenstate$ is thus a map $\posteriorFiltering : \Observation \markto \Hiddenstate$ satisfying
    \begin{equation}
        \label{eqn:bayesFiltering}
        \scalebox{0.8}{\begin{tikzpicture}
	\begin{pgfonlayer}{nodelayer}
		\node [style=none] (0) at (-1, 2.75) {};
		\node [style=none] (1) at (0, 2.75) {};
		\node [style=none] (2) at (-1, 0.75) {};
		\node [style=none] (3) at (0, 0.75) {};
		\node [style=none] (4) at (1, 1.75) {};
		\node [style=none] (5) at (0, 1.75) {$\hiddenmodel$};
		\node [style=none] (7) at (-1, 1.75) {};
		\node [style=none] (8) at (-1.75, 2.25) {$\Hiddenstate$};
		\node [style=none] (9) at (0.65, 1) {};
		\node [style=none] (10) at (2, 1) {};
		\node [style=none] (12) at (1.5, 3) {$\Observation$};
		\node [style=none] (13) at (1.5, 1.5) {$\Hiddenstate$};
		\node [style=none] (36) at (-4.5, 2.25) {};
		\node [style=none] (37) at (-3.5, 2.25) {};
		\node [style=none] (38) at (-4.5, 1.25) {};
		\node [style=none] (39) at (-3.5, 1.25) {};
		\node [style=none] (40) at (-2.5, 1.75) {};
		\node [style=none] (41) at (-3.5, 1.75) {$\prior$};
		\node [style=none] (42) at (-4.5, 1.75) {};
		\node [style=none] (50) at (4.5, 1.75) {=};
		\node [style=none] (51) at (2, 2.5) {};
		\node [style=none] (52) at (0.65, 2.5) {};
		\node [style=none] (53) at (-1, -0.25) {};
		\node [style=none] (54) at (0, -0.25) {};
		\node [style=none] (55) at (-1, -2.25) {};
		\node [style=none] (56) at (0, -2.25) {};
		\node [style=none] (57) at (1, -1.25) {};
		\node [style=none] (58) at (0, -1.25) {$\hiddenmodel$};
		\node [style=none] (59) at (-1, -1.25) {};
		\node [style=none] (60) at (-1.75, -0.75) {$\Hiddenstate$};
		\node [style=none] (61) at (0.65, -2) {};
		\node [style=function black] (62) at (2, -2) {};
		\node [style=none] (63) at (6, 0) {$\Observation$};
		\node [style=none] (64) at (1.5, -1.5) {$\Hiddenstate$};
		\node [style=none] (65) at (-4.5, -0.75) {};
		\node [style=none] (66) at (-3.5, -0.75) {};
		\node [style=none] (67) at (-4.5, -1.75) {};
		\node [style=none] (68) at (-3.5, -1.75) {};
		\node [style=none] (69) at (-2.5, -1.25) {};
		\node [style=none] (70) at (-3.5, -1.25) {$\prior$};
		\node [style=none] (71) at (-4.5, -1.25) {};
		\node [style=none] (72) at (6.5, -0.5) {};
		\node [style=none] (73) at (0.65, -0.5) {};
		\node [style=none] (74) at (3.5, -2.775) {};
		\node [style=none] (75) at (4.5, -2.775) {};
		\node [style=none] (76) at (3.5, -1.775) {};
		\node [style=none] (77) at (4.5, -1.775) {};
		\node [style=none] (78) at (5.5, -2.275) {};
		\node [style=none] (79) at (4.5, -2.275) {$\posteriorFiltering$};
		\node [style=none] (80) at (3.5, -2.25) {};
		\node [style=function black] (81) at (2.75, -0.5) {};
		\node [style=none] (82) at (6.5, -2.275) {};
		\node [style=none] (83) at (6, -1.775) {$\Hiddenstate$};
	\end{pgfonlayer}
	\begin{pgfonlayer}{edgelayer}
		\draw [style=thick-line] (0.center) to (1.center);
		\draw [style=thick-line] (0.center) to (2.center);
		\draw [style=thick-line] (2.center) to (3.center);
		\draw [style=thick-line, in=90, out=0] (1.center) to (4.center);
		\draw [style=thick-line, in=-90, out=0] (3.center) to (4.center);
		\draw [style=thick-line] (9.center) to (10.center);
		\draw [style=thick-line] (36.center) to (37.center);
		\draw [style=thick-line] (36.center) to (38.center);
		\draw [style=thick-line] (38.center) to (39.center);
		\draw [style=thick-line, in=90, out=0] (37.center) to (40.center);
		\draw [style=thick-line, in=-90, out=0] (39.center) to (40.center);
		\draw [style=thick-line] (40.center) to (7.center);
		\draw [style=thick-line] (52.center) to (51.center);
		\draw [style=thick-line] (53.center) to (54.center);
		\draw [style=thick-line] (53.center) to (55.center);
		\draw [style=thick-line] (55.center) to (56.center);
		\draw [style=thick-line, in=90, out=0] (54.center) to (57.center);
		\draw [style=thick-line, in=-90, out=0] (56.center) to (57.center);
		\draw [style=thick-line] (61.center) to (62);
		\draw [style=thick-line] (65.center) to (66.center);
		\draw [style=thick-line] (65.center) to (67.center);
		\draw [style=thick-line] (67.center) to (68.center);
		\draw [style=thick-line, in=90, out=0] (66.center) to (69.center);
		\draw [style=thick-line, in=-90, out=0] (68.center) to (69.center);
		\draw [style=thick-line] (69.center) to (59.center);
		\draw [style=thick-line] (73.center) to (72.center);
		\draw [style=thick-line] (74.center) to (75.center);
		\draw [style=thick-line] (74.center) to (76.center);
		\draw [style=thick-line] (76.center) to (77.center);
		\draw [style=thick-line, in=-90, out=0] (75.center) to (78.center);
		\draw [style=thick-line, in=90, out=0] (77.center) to (78.center);
		\draw [style=thick-line, in=-180, out=-90] (81) to (80.center);
		\draw [style=thick-line] (78.center) to (82.center);
	\end{pgfonlayer}
\end{tikzpicture}}
    \end{equation}
As with ordinary Bayesian inversions (\cref{def:bayesInference}), the map $\posteriorFiltering$ always exists in any Markov category with conditionals, including thus $\FinStoch$ and $\RelPlus$.
Note that setting $\hiddenmodel = \Copy_\Hiddenstate \comp \likelihood \otimes \id_\Hiddenstate : \Hiddenstate \markto \Observation \otimes \Hiddenstate$ recovers \cref{def:bayesInference}.

In Bayesian filtering, the idea is that $\hiddenmodel$ updates $\Hiddenstate$ and generates an observation $\Observation$.
Similarly to the Bayesian inference inversion, we can view $\prior$ as a prior distribution before the update of $\Hiddenstate$ and observation of $\Observation$, and $\posteriorFiltering : \Observation \markto \Hiddenstate$ as assigning posterior distributions over the updated $\Hiddenstate$ to the observations $\Observation$ generated via $\hiddenmodel$.
In $\FinStoch$, this corresponds to the following equation:
\begin{equation}
    \posteriorFiltering (\hiddenstate \mid \observation) = \frac{\sum_{\hiddenstate' \in \Hiddenstate} \prior(\hiddenstate') \hiddenmodel(\observation, \hiddenstate \mid \hiddenstate')}{\sum_{\hiddenstate', \hiddenstate'' \in \Hiddenstate} \prior(\hiddenstate') \hiddenmodel(\observation, \hiddenstate'' \mid \hiddenstate')}.
\end{equation}

The analogue to a conjugate prior to $\hiddenmodel: \Hiddenstate \markto \Hiddenstate \otimes \Observation$ in the filtering case is then the following~\cite{virgo2021interpreting}.

\begin{definition}[Conjugate prior for Bayesian filtering]
    \label{def:conjugatePriorFiltering}
    We call $\belief$ a conjugate prior for Bayesian filtering to $\hiddenmodel: \Hiddenstate \markto \Hiddenstate \otimes \Observation$ if there exists a deterministic map $\conj: \Observation \otimes \Parameter \to \Parameter$ that satisfies the following:
    \begin{equation}
        \scalebox{0.8}{\begin{tikzpicture}
	\begin{pgfonlayer}{nodelayer}
		\node [style=none] (61) at (-3.775, 2.75) {};
		\node [style=none] (62) at (-2.775, 2.75) {};
		\node [style=none] (63) at (-3.775, 0.75) {};
		\node [style=none] (64) at (-2.775, 0.75) {};
		\node [style=none] (65) at (-1.775, 1.75) {};
		\node [style=none] (66) at (-2.775, 1.75) {$\hiddenmodel$};
		\node [style=none] (67) at (-3.775, 1.75) {};
		\node [style=none] (68) at (-4.525, 2.25) {$\Hiddenstate$};
		\node [style=none] (69) at (-2.125, 1) {};
		\node [style=none] (70) at (-0.775, 1) {};
		\node [style=none] (71) at (-1.275, 3) {$\Observation$};
		\node [style=none] (72) at (-1.275, 1.5) {$\Hiddenstate$};
		\node [style=none] (73) at (-7.275, 2.25) {};
		\node [style=none] (74) at (-6.275, 2.25) {};
		\node [style=none] (75) at (-7.275, 1.25) {};
		\node [style=none] (76) at (-6.275, 1.25) {};
		\node [style=none] (77) at (-5.275, 1.75) {};
		\node [style=none] (78) at (-6.275, 1.75) {$\belief$};
		\node [style=none] (80) at (1.725, 1.75) {=};
		\node [style=none] (81) at (-0.775, 2.5) {};
		\node [style=none] (82) at (-2.125, 2.5) {};
		\node [style=none] (83) at (-3.775, -0.25) {};
		\node [style=none] (84) at (-2.775, -0.25) {};
		\node [style=none] (85) at (-3.775, -2.25) {};
		\node [style=none] (86) at (-2.775, -2.25) {};
		\node [style=none] (87) at (-1.775, -1.25) {};
		\node [style=none] (88) at (-2.775, -1.25) {$\hiddenmodel$};
		\node [style=none] (89) at (-3.775, -1.25) {};
		\node [style=none] (90) at (-4.525, -0.75) {$\Hiddenstate$};
		\node [style=none] (91) at (-2.125, -2) {};
		\node [style=function black] (92) at (-0.775, -2) {};
		\node [style=none] (93) at (5.225, 0) {$\Observation$};
		\node [style=none] (94) at (-1.275, -1.5) {$\Hiddenstate$};
		\node [style=none] (95) at (-7.275, -0.75) {};
		\node [style=none] (96) at (-6.275, -0.75) {};
		\node [style=none] (97) at (-7.275, -1.75) {};
		\node [style=none] (98) at (-6.275, -1.75) {};
		\node [style=none] (99) at (-5.275, -1.25) {};
		\node [style=none] (100) at (-6.275, -1.25) {$\belief$};
		\node [style=none] (102) at (5.725, -0.5) {};
		\node [style=none] (103) at (-2.125, -0.5) {};
		\node [style=none] (110) at (0.725, -2.25) {};
		\node [style=function black] (111) at (-0.025, -0.5) {};
		\node [style=none] (114) at (-7.275, -1.25) {};
		\node [style=none] (115) at (-8.775, -1.25) {};
		\node [style=none] (116) at (-8.275, -0.75) {$\Parameter$};
		\node [style=function black] (117) at (-8.025, -1.25) {};
		\node [style=none] (118) at (-7.275, -2.5) {};
		\node [style=none] (119) at (0.725, -2.5) {};
		\node [style=none] (120) at (-7.275, 1.75) {};
		\node [style=none] (121) at (-8.775, 1.75) {};
		\node [style=none] (122) at (-8.275, 2.25) {$\Parameter$};
		\node [style=none] (123) at (2.725, -3) {};
		\node [style=none] (124) at (3.975, -3) {};
		\node [style=none] (125) at (2.725, -2) {};
		\node [style=none] (126) at (3.725, -2) {};
		\node [style=none] (127) at (4.725, -2.5) {};
		\node [style=none] (128) at (3.725, -2.5) {$\belief$};
		\node [style=none] (129) at (5.725, -2.5) {};
		\node [style=none] (130) at (5.225, -2) {$\Hiddenstate$};
		\node [style=box-.5-.5] (131) at (1.225, -2.5) {$\conj$};
		\node [style=none] (132) at (2.725, -2.5) {};
		\node [style=none] (133) at (2.225, -2) {$\Parameter$};
	\end{pgfonlayer}
	\begin{pgfonlayer}{edgelayer}
		\draw [style=thick-line] (61.center) to (62.center);
		\draw [style=thick-line] (61.center) to (63.center);
		\draw [style=thick-line] (63.center) to (64.center);
		\draw [style=thick-line, in=90, out=0] (62.center) to (65.center);
		\draw [style=thick-line, in=-90, out=0] (64.center) to (65.center);
		\draw [style=thick-line] (69.center) to (70.center);
		\draw [style=thick-line] (73.center) to (74.center);
		\draw [style=thick-line] (73.center) to (75.center);
		\draw [style=thick-line] (75.center) to (76.center);
		\draw [style=thick-line, in=90, out=0] (74.center) to (77.center);
		\draw [style=thick-line, in=-90, out=0] (76.center) to (77.center);
		\draw [style=thick-line] (77.center) to (67.center);
		\draw [style=thick-line] (82.center) to (81.center);
		\draw [style=thick-line] (83.center) to (84.center);
		\draw [style=thick-line] (83.center) to (85.center);
		\draw [style=thick-line] (85.center) to (86.center);
		\draw [style=thick-line, in=90, out=0] (84.center) to (87.center);
		\draw [style=thick-line, in=-90, out=0] (86.center) to (87.center);
		\draw [style=thick-line] (91.center) to (92);
		\draw [style=thick-line] (95.center) to (96.center);
		\draw [style=thick-line] (95.center) to (97.center);
		\draw [style=thick-line] (97.center) to (98.center);
		\draw [style=thick-line, in=90, out=0] (96.center) to (99.center);
		\draw [style=thick-line, in=-90, out=0] (98.center) to (99.center);
		\draw [style=thick-line] (99.center) to (89.center);
		\draw [style=thick-line] (103.center) to (102.center);
		\draw [style=thick-line, in=180, out=-90] (111) to (110.center);
		\draw [style=thick-line] (115.center) to (114.center);
		\draw [style=thick-line] (118.center) to (119.center);
		\draw [style=thick-line, in=180, out=-90] (117) to (118.center);
		\draw [style=thick-line] (121.center) to (120.center);
		\draw [style=thick-line] (123.center) to (124.center);
		\draw [style=thick-line] (123.center) to (125.center);
		\draw [style=thick-line] (125.center) to (126.center);
		\draw [style=thick-line, in=-90, out=0] (124.center) to (127.center);
		\draw [style=thick-line, in=90, out=0] (126.center) to (127.center);
		\draw [style=thick-line] (127.center) to (129.center);
		\draw [style=thick-line] (131) to (132.center);
	\end{pgfonlayer}
\end{tikzpicture}}
        \label{eqn:conjugatePriorFiltering}
    \end{equation}
\end{definition}

The map $\conj : \Observation \otimes \Parameter \to \Parameter$ maps parameters of prior distributions over $\Hiddenstate$ to parameters of posterior distributions over \emph{updated} $\Hiddenstate$ (cf. \cref{def:conjugatePriorInference} where $\Hiddenstate$ is instead fixed), while taking into account observations generated by $\hiddenmodel$.
In a probabilistic context (including in $\FinStoch$), the map $\hiddenmodel$ where $\Hiddenstate$ is updated is often called a \textit{hidden Markov model} (or a \textit{discrete state-space model}), as opposed to the statistical model $\likelihood$ where $\Hiddenstate$ is fixed.

Reference~\cite{virgo2021interpreting} inverts this account, by starting with a map $\conj : \Observation \otimes \Parameter \to \Parameter$ and asking whether $\belief$ and $\hiddenmodel$ exist such that \cref{eqn:conjugatePriorFiltering} holds.
This leads to the following definition.
\begin{definition}[Bayesian filtering interpretation~\cite{virgo2021interpreting}]
    \label{def:interpretation}
    Given a deterministic map $\conj : \Observation \otimes \Parameter \to \Parameter$, a \textit{Bayesian filtering interpretation} of $\conj$ consists of a map $\belief: \Parameter \markto \Hiddenstate$ called the \textit{interpretation map}, together with a map $\hiddenmodel: \Hiddenstate \markto \Hiddenstate \otimes \Observation$ called \textit{hidden Markov model}, such that \cref{eqn:conjugatePriorFiltering} holds.
    In this context, \cref{eqn:conjugatePriorFiltering} is called the \textit{consistency equation}.
    A map $\conj : \Observation \otimes \Parameter \to \Parameter$ together with an interpretation $(\belief, \hiddenmodel)$ is called a \textit{reasoner}.
\end{definition}
Technically, this is a slight simplification of the main definition of~\cite{virgo2021interpreting}, since that paper allows $\conj$ to be stochastic rather than restricting it to be deterministic.
This terminology stems from the proposal to look at a the map $\conj$ as a physical system, whose states parametrise a Bayesian prior. The interpretation map specifies this parametrisation.
The consistency equation, \cref{eqn:conjugatePriorFiltering}, guarantees that when the prior updates, it does so in a way that is consistent with Bayesian filtering.
We note also that it is possible for the same map $\conj$ to admit many different Bayesian filtering interpretations.

Although the hidden Markov ``model'' $\hiddenmodel$ is a different kind of thing from the ``model'' $\model$ in the internal model principle, they are nevertheless related.
In the next section we will formally show this connection,
providing a Bayesian interpretation of the IMP and showing that interpretation maps can be seen as a generalisation of models in \cref{def:model}.

\section{Bayesian filtering interpretations from models}
\label{sec:IMPInterpretation}
To explain the connection between the notion of models in \cref{def:model} and Bayesian filtering interpretations, we will now show that every model induces a Bayesian filtering interpretation.
This means, in turn, that under the hypotheses of \cref{th:impFullSystem,th:imp}, the autonomous attracting controller $\Contstaraut$ has (among possibly many others) an interpretation involving the attracting full system $\System^*$ and one involving the attracting environment $\Env^*$.

\subsection{A possibilistic perspective on the IMP}
The existence of power sets allows to turn models (\cref{def:model}) \emph{upside down}, bringing us closer, as we shall briefly see, to the Bayesian filtering interpretation we are after.
We start from the following construction:
\begin{construction}
    \label{const:model}
    Given any surjective function $f : A \rrightarrow B$ we can define a left-total relation $f^{-1} : B \to \possplus(A)$ landing in the set $\possplus(A) = \poss(A) \setminus \{\varnothing\}$ of inhabited subsets of $A$:
    \begin{align}
    \label{eq:preimage}
        f^{-1}(b) \coloneqq \{a \in A \mid f(a)=b\} \quad \text{for } b \in B.
    \end{align}
    For a given $b \in B$ the subset $f^{-1}(b)$ is called the \textit{fibre of $f:A \to B$ over $b \in B$}.
    Both $f:A \rrightarrow B$ and $f^{-1}:B \to \possplus(A)$ can be seen as maps in $\RelPlus$ with types $f:A \to B$ and $f^{-1}:B \markto A$.
    Composing these two yields a closure map $\diamondsuit_f : A \markto A$, which \emph{closes} an element $a \in A$ by mapping it to the set of all $a' \in A$ that generate the same image:
    \begin{equation}
        \label{eq:closure}
        \diamondsuit_f(a) \coloneqq (f^{-1} \comp f)(a) = \{a' \in A \mid f(a') = f(a)\}.
    \end{equation}
\end{construction}

Instantiated for the case of a model $\model$ between autonomous systems, with map of states $\model_\sts : \sysX \to \sysM$, we can thus interpret the objects constructed above as follows:
\begin{enumerate}
    \item $\indmodel_\sts : \sysM \markto \sysX$ explicitly assigns to each state $\sysm \in \sysM$ the (inhabited) set of states of the system $\SysX$ it models.
    We can think of this map as encoding a \textit{belief}: $\indmodel_\sts(\sysm)$ is the belief (in the form of a set of possibilities) someone using the model $\model$ would have about the state of the system $\SysX$ knowing just $\sysm$.
    \item $\diamondsuit_{\model_\sts}$ completes an element $x \in \sysX$ with all the other states in $\sysX$ which can't be distinguished from $x$ by the model $\model$, i.e. states \emph{equicredible} with $x$.
\end{enumerate}

We now state the following proposition, which will be used in \cref{th:IMPmodelimpliesBayesmodel}.

\begin{proposition}
\label{th:IndexedModel}
    For autonomous systems $\SysM$ and $\SysX$, if $\SysM$ models $\SysX$ with respect to the map of systems $\model$, the following diagram commutes in $\RelPlus$:
    \begin{equation}
    \label{eqn:beliefModel}
        \begin{tikzcd}[ampersand replacement=\&]
            \sysM \& \sysX \\
            \sysM \& \sysX
            \arrow["{\indmodel_\sts}", "\bullet"{marking}, from=1-1, to=1-2]
            \arrow["{\update_\SysM}"', from=1-1, to=2-1]
            \arrow["{\diamondupdate{\SysX}}", "\bullet"{marking}, from=1-2, to=2-2]
            \arrow["{\indmodel_\sts}"', "\bullet"{marking}, from=2-1, to=2-2]
        \end{tikzcd}
    \end{equation}
    where
    \begin{align}
        \label{eqn:coarsedUpdates}
        \diamondupdate{\SysX} \coloneqq  \update_\SysX \comp \diamondsuit_{\model_\sts}.
    \end{align}
\end{proposition}
\begin{proof}
    See \cref{prf:IndexedModel}.
\end{proof}

The function $\diamondupdate{\SysX}$ can be thought of as \textit{the dynamics of $\SysX$ from the point of view of the map $\model: \SysX \to \SysM$} of a system $\SysM$ modelling a system $\SysX$.
We will see shortly that the map $\indmodel_\sts$ can be seen as an interpretation map and $\diamondupdate{\SysX}$ as the according model of a Bayesian filtering interpretation.
Before proceeding with the main theorem, we briefly note the following as relevant examples of \cref{th:IndexedModel}.

\begin{example}
    $\Contstaraut$ models $\System^*$ with the map $\modelSys$, according to \cref{th:impFullSystem}.
    According to \cref{th:IndexedModel} then, that the following diagram commutes:
    \begin{equation}
        \begin{tikzcd}
            {\contstaraut} & {\system^*}\\
            {\contstaraut} & {\system^*}
            \arrow["\indmodelsys", "\bullet"{marking}, from=1-1, to=1-2]
            \arrow["{\update_{\Contstaraut}}"', from=1-1, to=2-1]
            \arrow["{\diamondupdate{\system^*}}", "\bullet"{marking}, from=1-2, to=2-2]
            \arrow["\indmodelsys"', "\bullet"{marking}, from=2-1, to=2-2]
        \end{tikzcd}
    \end{equation}
\end{example}

\begin{example}
    $\Contstaraut$ models $\Env^*$ with the map $\modelEnv$, according to \cref{th:imp}.
    Thus \cref{th:IndexedModel} tells us that the following diagram commutes:
    \begin{equation}
        \begin{tikzcd}
            {\contstaraut} & {\env^*}\\
            {\contstaraut} & {\env^*}
            \arrow["\indmodelenv", "\bullet"{marking}, from=1-1, to=1-2]
            \arrow["{\update_{\Contstaraut}}"', from=1-1, to=2-1]
            \arrow["{\diamondupdate{\env^*}}", "\bullet"{marking}, from=1-2, to=2-2]
            \arrow["\indmodelenv"', "\bullet"{marking}, from=2-1, to=2-2]
        \end{tikzcd}
    \end{equation}
\end{example}

\subsection{Models imply Bayesian filtering interpretations}

Having obtained a possibilistic version of the IMP with a notion of beliefs given by the map $\indmodel_\sts$, we now show how a model from \cref{def:model} induces a Bayesian filtering interpretation.

\begin{theorem}
    \label{th:IMPmodelimpliesBayesmodel}
    Let $\SysM$ model $\SysX$ with $\model: \SysX \to \SysM$, and assume $\SysM$ and $\SysX$ are autonomous.
    Define $\reasoner: \sysX \otimes \sysM \to \sysM$ as
    \begin{equation}
        \label{eqn:IMPReasoner}
        \scalebox{0.8}{\begin{tikzpicture}
	\begin{pgfonlayer}{nodelayer}
		\node [style=none] (0) at (-8.5, 0.75) {};
		\node [style=none] (1) at (-6.25, 0.75) {};
		\node [style=none] (2) at (-0.2, 0) {};
		\node [style=none] (3) at (-6.25, -0.75) {};
		\node [style=none] (4) at (-7.5, -0.25) {$\sysM$};
		\node [style=none] (5) at (-1.25, 0.5) {$\sysM$};
		\node [style=none] (6) at (-7.5, 1.25) {$\sysX$};
		\node [style=none] (7) at (-8.5, -0.75) {};
		\node [style=none] (8) at (1.25, 0) {$\coloneqq$};
		\node [style=none] (9) at (2.75, 0.5) {};
		\node [style=function black] (10) at (4.5, 0.5) {};
		\node [style=none] (11) at (11.05, 0) {};
		\node [style=none] (12) at (5, -0.5) {};
		\node [style=none] (13) at (3.75, 0) {$\sysM$};
		\node [style=none] (14) at (10, 0.5) {$\sysM$};
		\node [style=none] (15) at (3.75, 1) {$\sysX$};
		\node [style=none] (16) at (9, 0) {};
		\node [style=box-2-1] (17) at (7, 0) {$\update_\SysM$};
		\node [style=none] (18) at (2.75, -0.5) {};
		\node [style=none] (19) at (9, 0) {};
		\node [style=none] (20) at (-6.25, 1.25) {};
		\node [style=none] (21) at (-2.45, 1.25) {};
		\node [style=none] (22) at (-6.25, -1.25) {};
		\node [style=none] (23) at (-2.45, -1.25) {};
		\node [style=none] (24) at (-2.45, 0) {};
		\node [style=none] (25) at (-4.25, 0) {$\reasoner$};
	\end{pgfonlayer}
	\begin{pgfonlayer}{edgelayer}
		\draw [style=thick-line] (0.center) to (1.center);
		\draw [style=thick-line] (7.center) to (3.center);
		\draw [style=thick-line] (9.center) to (10);
		\draw [style=thick-line] (18.center) to (12.center);
		\draw [style=thick-line] (11.center) to (19.center);
		\draw [style=thick-line] (20.center) to (21.center);
		\draw [style=thick-line] (20.center) to (22.center);
		\draw [style=thick-line] (22.center) to (23.center);
		\draw [style=thick-line] (2.center) to (24.center);
		\draw [style=thick-line] (21.center) to (23.center);
	\end{pgfonlayer}
\end{tikzpicture}}
    \end{equation}
    and $\hiddenmodel: \sysX \markto \sysX \otimes \sysX$ as:
    \begin{equation}
        \label{eqn:IMPModel}
        \scalebox{0.8}{\begin{tikzpicture}
	\begin{pgfonlayer}{nodelayer}
		\node [style=none] (0) at (1.75, -0.5) {};
		\node [style=none] (1) at (4.15, 0) {};
		\node [style=none] (2) at (5.9, 0) {};
		\node [style=none] (3) at (4.15, -1) {};
		\node [style=none] (4) at (5.9, -1) {};
		\node [style=none] (5) at (7.15, -0.5) {};
		\node [style=none] (6) at (5.525, -0.5) {$\diamondupdate{\SysX}$};
		\node [style=none] (7) at (4.15, -0.5) {};
		\node [style=function black] (8) at (3.25, -0.5) {};
		\node [style=none] (9) at (4.15, 1) {};
		\node [style=none] (10) at (8.75, -0.5) {};
		\node [style=none] (11) at (8, 1.5) {$\sysX$};
		\node [style=none] (12) at (8, 0) {$\sysX$};
		\node [style=none] (13) at (8.75, 1) {};
		\node [style=none] (14) at (2.5, 0) {$\sysX$};
		\node [style=none] (15) at (-9.5, 0.25) {};
		\node [style=none] (16) at (-7.25, 0.25) {};
		\node [style=none] (17) at (-1.25, 1) {};
		\node [style=none] (18) at (-1.25, -0.5) {};
		\node [style=none] (19) at (-8.5, 0.75) {$\sysX$};
		\node [style=none] (20) at (-2, 1.5) {$\sysX$};
		\node [style=none] (21) at (-2, 0) {$\sysX$};
		\node [style=none] (22) at (-7.25, 1.5) {};
		\node [style=none] (23) at (-4.75, 1.5) {};
		\node [style=none] (24) at (-7.25, -1) {};
		\node [style=none] (25) at (-4.75, -1) {};
		\node [style=none] (26) at (-3.25, 0.25) {};
		\node [style=none] (27) at (-5.25, 0.25) {$\hiddenmodel$};
		\node [style=none] (28) at (-3.55, -0.5) {};
		\node [style=none] (29) at (-3.55, 1) {};
		\node [style=none] (30) at (0.35, 0.25) {$\coloneqq$};
	\end{pgfonlayer}
	\begin{pgfonlayer}{edgelayer}
		\draw [style=thick-line] (1.center) to (2.center);
		\draw [style=thick-line] (1.center) to (3.center);
		\draw [style=thick-line] (3.center) to (4.center);
		\draw [style=thick-line, in=105, out=0] (2.center) to (5.center);
		\draw [style=thick-line, in=-105, out=0] (4.center) to (5.center);
		\draw [style=thick-line, in=180, out=90, looseness=0.75] (8) to (9.center);
		\draw [style=thick-line] (15.center) to (16.center);
		\draw [style=thick-line] (22.center) to (23.center);
		\draw [style=thick-line] (22.center) to (24.center);
		\draw [style=thick-line] (24.center) to (25.center);
		\draw [style=thick-line, in=90, out=0] (23.center) to (26.center);
		\draw [style=thick-line, in=-90, out=0] (25.center) to (26.center);
		\draw [style=thick-line] (28.center) to (18.center);
		\draw [style=thick-line] (17.center) to (29.center);
		\draw [style=thick-line] (0.center) to (7.center);
		\draw [style=thick-line] (5.center) to (10.center);
		\draw [style=thick-line] (9.center) to (13.center);
	\end{pgfonlayer}
\end{tikzpicture}}
    \end{equation}
    Then $\hiddenmodel$ is the hidden Markov model, and $\indmodel_\sts: \sysM \markto \sysX$ the interpretation map of a Bayesian filtering interpretation of $\reasoner$, i.e. we have:
    \begin{equation}
        \scalebox{0.7}{\begin{tikzpicture}
	\begin{pgfonlayer}{nodelayer}
		\node [style=none] (0) at (-5, 3) {};
		\node [style=none] (1) at (-3.5, 3) {};
		\node [style=none] (2) at (-5, 2) {};
		\node [style=none] (3) at (-3.5, 2) {};
		\node [style=none] (4) at (-2.5, 2.5) {};
		\node [style=none] (5) at (-3.75, 2.5) {$\diamondupdate{\SysX}$};
		\node [style=none] (6) at (-5, 2.5) {};
		\node [style=none] (7) at (-6.75, 3) {$\sysX$};
		\node [style=none] (10) at (1.25, 3) {$\sysX$};
		\node [style=none] (12) at (-9.25, 3) {};
		\node [style=none] (13) at (-8.25, 3) {};
		\node [style=none] (14) at (-9.25, 2) {};
		\node [style=none] (15) at (-8.25, 2) {};
		\node [style=none] (16) at (-7.25, 2.5) {};
		\node [style=none] (17) at (-8.25, 2.5) {$\indmodel_\sts$};
		\node [style=none] (18) at (3.25, 3) {=};
		\node [style=none] (19) at (1.75, 2.5) {};
		\node [style=none] (49) at (-9.25, 2.5) {};
		\node [style=none] (50) at (-10.75, 2.5) {};
		\node [style=none] (51) at (-10.25, 3) {$\sysM$};
		\node [style=none] (156) at (-6.25, 5) {};
		\node [style=none] (157) at (-3.25, 5) {};
		\node [style=none] (158) at (-6.25, 1.5) {};
		\node [style=none] (159) at (-3.25, 1.5) {};
		\node [style=none] (160) at (-1.5, 3.25) {};
		\node [style=none] (161) at (1.25, 4.5) {$\sysX$};
		\node [style=none] (162) at (1.75, 4) {};
		\node [style=function black] (163) at (-5.75, 2.5) {};
		\node [style=none] (164) at (-4.85, 4) {};
		\node [style=none] (165) at (-5, -1.5) {};
		\node [style=none] (166) at (-3.5, -1.5) {};
		\node [style=none] (167) at (-5, -2.5) {};
		\node [style=none] (168) at (-3.5, -2.5) {};
		\node [style=none] (169) at (-2.5, -2) {};
		\node [style=none] (170) at (-3.75, -2) {$\diamondupdate{\SysX}$};
		\node [style=none] (171) at (-5, -2) {};
		\node [style=none] (172) at (-6.75, -1.5) {$\sysX$};
		\node [style=none] (173) at (-1, -1.5) {$\sysX$};
		\node [style=none] (174) at (-9.25, -1.5) {};
		\node [style=none] (175) at (-8.25, -1.5) {};
		\node [style=none] (176) at (-9.25, -2.5) {};
		\node [style=none] (177) at (-8.25, -2.5) {};
		\node [style=none] (178) at (-7.25, -2) {};
		\node [style=none] (179) at (-8.25, -2) {$\indmodel_\sts$};
		\node [style=function black] (180) at (-0.5, -2) {};
		\node [style=none] (181) at (-9.25, -2) {};
		\node [style=none] (182) at (-10.75, -2) {};
		\node [style=none] (183) at (-10.25, -1.5) {$\sysM$};
		\node [style=none] (184) at (-6.25, 0.5) {};
		\node [style=none] (185) at (-3.25, 0.5) {};
		\node [style=none] (186) at (-6.25, -3) {};
		\node [style=none] (187) at (-3.25, -3) {};
		\node [style=none] (188) at (-1.5, -1.25) {};
		\node [style=none] (189) at (8.75, 0) {$\sysX$};
		\node [style=none] (190) at (9.25, -0.5) {};
		\node [style=function black] (191) at (-5.75, -2) {};
		\node [style=none] (192) at (-4.85, -0.5) {};
		\node [style=box-1-.5] (193) at (2.875, -3.5) {$\update_\SysM$};
		\node [style=none] (194) at (0.75, -2.5) {};
		\node [style=none] (195) at (5, -2.5) {};
		\node [style=none] (196) at (0.75, -4.5) {};
		\node [style=none] (197) at (5, -4.5) {};
		\node [style=none] (198) at (6.25, -3) {};
		\node [style=none] (199) at (7.25, -3) {};
		\node [style=none] (200) at (6.25, -4) {};
		\node [style=none] (201) at (7.25, -4) {};
		\node [style=none] (202) at (8.25, -3.5) {};
		\node [style=none] (203) at (7.25, -3.5) {$\indmodel_\sts$};
		\node [style=none] (204) at (8.75, -3) {$\sysX$};
		\node [style=none] (205) at (6.25, -3.5) {};
		\node [style=none] (206) at (9.25, -3.5) {};
		\node [style=function black] (207) at (0, -0.5) {};
		\node [style=function black] (208) at (1.275, -3.025) {};
		\node [style=none] (210) at (0.775, -3.025) {};
		\node [style=function black] (211) at (-10, -2) {};
		\node [style=none] (212) at (-6.75, -3.5) {};
		\node [style=none] (213) at (5.5, -3) {$\sysM$};
	\end{pgfonlayer}
	\begin{pgfonlayer}{edgelayer}
		\draw [style=thick-line] (0.center) to (1.center);
		\draw [style=thick-line] (0.center) to (2.center);
		\draw [style=thick-line] (2.center) to (3.center);
		\draw [style=thick-line, in=90, out=0] (1.center) to (4.center);
		\draw [style=thick-line, in=-90, out=0] (3.center) to (4.center);
		\draw [style=thick-line] (12.center) to (13.center);
		\draw [style=thick-line] (12.center) to (14.center);
		\draw [style=thick-line] (14.center) to (15.center);
		\draw [style=thick-line, in=90, out=0] (13.center) to (16.center);
		\draw [style=thick-line, in=-90, out=0] (15.center) to (16.center);
		\draw [style=thick-line] (16.center) to (6.center);
		\draw [style=thick-line] (50.center) to (49.center);
		\draw [style=big dashes] (158.center)
			 to (156.center)
			 to (157.center)
			 to [in=90, out=0] (160.center)
			 to [in=0, out=-90] (159.center)
			 to cycle;
		\draw [style=thick-line] (4.center) to (19.center);
		\draw [style=thick-line, in=180, out=90, looseness=0.75] (163) to (164.center);
		\draw [style=thick-line] (164.center) to (162.center);
		\draw [style=thick-line] (165.center) to (166.center);
		\draw [style=thick-line] (165.center) to (167.center);
		\draw [style=thick-line] (167.center) to (168.center);
		\draw [style=thick-line, in=90, out=0] (166.center) to (169.center);
		\draw [style=thick-line, in=-90, out=0] (168.center) to (169.center);
		\draw [style=thick-line] (174.center) to (175.center);
		\draw [style=thick-line] (174.center) to (176.center);
		\draw [style=thick-line] (176.center) to (177.center);
		\draw [style=thick-line, in=90, out=0] (175.center) to (178.center);
		\draw [style=thick-line, in=-90, out=0] (177.center) to (178.center);
		\draw [style=thick-line] (178.center) to (171.center);
		\draw [style=thick-line] (182.center) to (181.center);
		\draw [style=big dashes] (184.center)
			 to (186.center)
			 to (187.center)
			 to [in=-90, out=0] (188.center)
			 to [in=0, out=90] (185.center)
			 to cycle;
		\draw [style=thick-line, in=180, out=90, looseness=0.75] (191) to (192.center);
		\draw [style=thick-line] (192.center) to (190.center);
		\draw [style=big dashes] (197.center)
			 to (196.center)
			 to (194.center)
			 to (195.center)
			 to cycle;
		\draw [style=thick-line] (198.center) to (199.center);
		\draw [style=thick-line] (198.center) to (200.center);
		\draw [style=thick-line] (200.center) to (201.center);
		\draw [style=thick-line, in=90, out=0] (199.center) to (202.center);
		\draw [style=thick-line, in=-90, out=0] (201.center) to (202.center);
		\draw [style=thick-line] (202.center) to (206.center);
		\draw [style=thick-line] (193) to (205.center);
		\draw [style=thick-line, in=165, out=-90, looseness=0.75] (207) to (210.center);
		\draw [style=thick-line] (210.center) to (208);
		\draw [style=thick-line, in=180, out=-90] (211) to (212.center);
		\draw [style=thick-line] (212.center) to (193);
		\draw [style=thick-line] (169.center) to (180);
	\end{pgfonlayer}
\end{tikzpicture}}
        \label{eqn:ConsistentPossibilisticUpdates}
    \end{equation}
    where the dashed lines show, informally, where we replaced the definitions above in \cref{eqn:conjugatePriorFiltering}.
\end{theorem}
\begin{proof}
    See \cref{prf:beliefmap}.
\end{proof}

We highlight that the hidden Markov model $\hiddenmodel$ in \cref{th:IMPmodelimpliesBayesmodel} is possibilistic as opposed to probabilistic since we are now in $\RelPlus$ rather than $\FinStoch$ or another Markov category.
This means, once again, that $\hiddenmodel$ includes transitions to sets of possible states without assigning numerical probabilities, and is related to, among other, standard work on (non-deterministic) labelled transition systems in automata theory, as previously highlighted by, for instance,~\cite{virgo2021interpreting, fritz2024hidden}.

We then note that updates $\diamondupdate{\SysX}$ constitute a kind of approximation of
$\update_\SysX$.
For any state $x \in X$, the deterministic result of updates $\update_\SysX(x)$ is in fact replaced by an approximate, possibilistic one: the set $\diamondsuit_{\model_\sts}(\update_\SysX(x))$ of all states that are mapped to the same state $\model_\sts(\update_\SysX(x)) \in \sysM$ (i.e. to the fibre over $\model_\sts(\update_\SysX(x))$.)
In this sense, these states are indistinguishable from the perspective of $\SysM$.

This approximation is derived from the map on states $\model_\sts$ defined as a surjective function, see \cref{eq:closure}.
If the map on states was \emph{bijective}, then  $\diamondupdate{\SysX}=\update_\SysX$ and
beliefs would be (trivially) concentrated on a single hidden state.

Our result also shows a rather simplistic form of Bayesian filtering where observations are essentially ignored.
This is to be expected due to our definition of model in \cref{def:model}.
As mentioned in \cref{rmk:HWIMP}, systems that model autonomous systems must also be autonomous.
We thus only have an update function $\update_{\SysM}: \sysM \to \sysM$ that take no inputs to build a Bayesian filtering interpretation.

Nonetheless, even while ignoring observations, the Bayesian filtering interpretation is still consistent. This is possible since, due to the approximation of $\update_\SysX$, for any prior given by $\indmodel_\sts:M \markto X$, the observations do not affect the posterior.

There is another way in which Bayesian filtering interpretations induced by models of autonomous systems are special. Their beliefs are necessarily disjoint. This is a direct consequence of the interpretation map being given by a right inverse $\indmodel_\sts$ of the surjective map on states $\model_\sts$.

\subsection{Applications to the Internal Model Principle}
Recall that, in the IMP settings, the error feedback structure of \Cref{ass:autonomousController} let us define the autonomous attracting controller $\Contstaraut$, which is an autonomous system.
According to \Cref{th:impFullSystem} then, the autonomous attracting controller models the attracting full system $\System^*$.
\Cref{th:IMPmodelimpliesBayesmodel} then tells us that we then have the following Bayesian filtering interpretation.

\begin{example}
    Define $\reasoner:  \system^* \otimes \contstaraut \to \contstaraut$ as
    \begin{equation}
        \scalebox{0.8}{\begin{tikzpicture}
	\begin{pgfonlayer}{nodelayer}
		\node [style=none] (0) at (-9.75, 1) {};
		\node [style=none] (1) at (-7.5, 1) {};
		\node [style=none] (2) at (-1.45, 0.25) {};
		\node [style=none] (3) at (-7.5, -0.5) {};
		\node [style=none] (4) at (-8.75, 0) {$\contstaraut$};
		\node [style=none] (5) at (-2.5, 0.75) {$\contstaraut$};
		\node [style=none] (6) at (-8.75, 1.5) {$\system^*$};
		\node [style=none] (7) at (-9.75, -0.5) {};
		\node [style=none] (8) at (0, 0.25) {$\coloneqq$};
		\node [style=none] (9) at (1.5, 0.75) {};
		\node [style=function black] (10) at (3.25, 0.75) {};
		\node [style=none] (11) at (9.8, 0.25) {};
		\node [style=none] (12) at (3.75, -0.25) {};
		\node [style=none] (13) at (2.5, 0.25) {$\contstaraut$};
		\node [style=none] (14) at (8.75, 0.75) {$\contstaraut$};
		\node [style=none] (15) at (2.5, 1.25) {$\system^*$};
		\node [style=none] (16) at (7.75, 0.25) {};
		\node [style=box-2-1] (17) at (5.75, 0.25) {$\update_{\Contstaraut}$};
		\node [style=none] (18) at (1.5, -0.25) {};
		\node [style=none] (19) at (7.75, 0.25) {};
		\node [style=none] (20) at (-7.5, 1.5) {};
		\node [style=none] (21) at (-3.7, 1.5) {};
		\node [style=none] (22) at (-7.5, -1) {};
		\node [style=none] (23) at (-3.7, -1) {};
		\node [style=none] (24) at (-3.7, 0.25) {};
		\node [style=none] (25) at (-5.5, 0.25) {$\reasoner$};
	\end{pgfonlayer}
	\begin{pgfonlayer}{edgelayer}
		\draw [style=thick-line] (0.center) to (1.center);
		\draw [style=thick-line] (7.center) to (3.center);
		\draw [style=thick-line] (9.center) to (10);
		\draw [style=thick-line] (18.center) to (12.center);
		\draw [style=thick-line] (11.center) to (19.center);
		\draw [style=thick-line] (20.center) to (21.center);
		\draw [style=thick-line] (20.center) to (22.center);
		\draw [style=thick-line] (22.center) to (23.center);
		\draw [style=thick-line] (2.center) to (24.center);
		\draw [style=thick-line] (21.center) to (23.center);
	\end{pgfonlayer}
\end{tikzpicture}}
    \end{equation}
    and $\hiddenmodel: \system^* \markto \system^* \otimes \system^*$ as:
    \begin{equation}
        \scalebox{0.8}{\begin{tikzpicture}
	\begin{pgfonlayer}{nodelayer}
		\node [style=none] (0) at (1.5, -0.5) {};
		\node [style=none] (1) at (3.9, 0) {};
		\node [style=none] (2) at (5.65, 0) {};
		\node [style=none] (3) at (3.9, -1) {};
		\node [style=none] (4) at (5.65, -1) {};
		\node [style=none] (5) at (6.9, -0.5) {};
		\node [style=none] (6) at (5.275, -0.5) {$\diamondupdate{\System^*}$};
		\node [style=none] (7) at (3.9, -0.5) {};
		\node [style=function black] (8) at (3, -0.5) {};
		\node [style=none] (9) at (3.9, 1) {};
		\node [style=none] (10) at (8.5, -0.5) {};
		\node [style=none] (11) at (7.75, 1.5) {$\system^*$};
		\node [style=none] (12) at (7.75, 0) {$\system^*$};
		\node [style=none] (13) at (8.5, 1) {};
		\node [style=none] (14) at (2.25, 0) {$\system^*$};
		\node [style=none] (15) at (-9.75, 0.25) {};
		\node [style=none] (16) at (-7.5, 0.25) {};
		\node [style=none] (17) at (-1.5, 1) {};
		\node [style=none] (18) at (-1.5, -0.5) {};
		\node [style=none] (19) at (-8.75, 0.75) {$\system^*$};
		\node [style=none] (20) at (-2.25, 1.5) {$\system^*$};
		\node [style=none] (21) at (-2.25, 0) {$\system^*$};
		\node [style=none] (22) at (-7.5, 1.5) {};
		\node [style=none] (23) at (-5, 1.5) {};
		\node [style=none] (24) at (-7.5, -1) {};
		\node [style=none] (25) at (-5, -1) {};
		\node [style=none] (26) at (-3.5, 0.25) {};
		\node [style=none] (27) at (-5.5, 0.25) {$\hiddenmodel$};
		\node [style=none] (28) at (-3.8, -0.5) {};
		\node [style=none] (29) at (-3.8, 1) {};
		\node [style=none] (30) at (0.1, 0.25) {$\coloneqq$};
	\end{pgfonlayer}
	\begin{pgfonlayer}{edgelayer}
		\draw [style=thick-line] (1.center) to (2.center);
		\draw [style=thick-line] (1.center) to (3.center);
		\draw [style=thick-line] (3.center) to (4.center);
		\draw [style=thick-line, in=105, out=0] (2.center) to (5.center);
		\draw [style=thick-line, in=-105, out=0] (4.center) to (5.center);
		\draw [style=thick-line, in=180, out=90, looseness=0.75] (8) to (9.center);
		\draw [style=thick-line] (15.center) to (16.center);
		\draw [style=thick-line] (22.center) to (23.center);
		\draw [style=thick-line] (22.center) to (24.center);
		\draw [style=thick-line] (24.center) to (25.center);
		\draw [style=thick-line, in=90, out=0] (23.center) to (26.center);
		\draw [style=thick-line, in=-90, out=0] (25.center) to (26.center);
		\draw [style=thick-line] (28.center) to (18.center);
		\draw [style=thick-line] (17.center) to (29.center);
		\draw [style=thick-line] (0.center) to (7.center);
		\draw [style=thick-line] (5.center) to (10.center);
		\draw [style=thick-line] (9.center) to (13.center);
	\end{pgfonlayer}
\end{tikzpicture}}
    \end{equation}
    Then $\hiddenmodel$ is the hidden Markov model, and $\indmodelenv : \contstaraut \markto \system^*$ the interpretation map of a Bayesian filtering interpretation of $\reasoner$, i.e. we have:
    \begin{equation}
        \scalebox{0.7}{\begin{tikzpicture}
	\begin{pgfonlayer}{nodelayer}
		\node [style=none] (0) at (-4.5, 2.75) {};
		\node [style=none] (1) at (-3, 2.75) {};
		\node [style=none] (2) at (-4.5, 1.75) {};
		\node [style=none] (3) at (-3, 1.75) {};
		\node [style=none] (4) at (-2, 2.25) {};
		\node [style=none] (5) at (-3.25, 2.25) {$\diamondupdate{\System^*}$};
		\node [style=none] (6) at (-4.5, 2.25) {};
		\node [style=none] (7) at (-6.25, 2.75) {$\states{\system^*}$};
		\node [style=none] (8) at (1.75, 2.75) {$\states{\system^*}$};
		\node [style=none] (9) at (-8.75, 2.75) {};
		\node [style=none] (10) at (-7.75, 2.75) {};
		\node [style=none] (11) at (-8.75, 1.75) {};
		\node [style=none] (12) at (-7.75, 1.75) {};
		\node [style=none] (13) at (-6.75, 2.25) {};
		\node [style=none] (14) at (-7.75, 2.25) {$\indmodelsys$};
		\node [style=none] (15) at (3.75, 2.75) {=};
		\node [style=none] (16) at (2.25, 2.25) {};
		\node [style=none] (17) at (-8.75, 2.25) {};
		\node [style=none] (18) at (-10.25, 2.25) {};
		\node [style=none] (19) at (-9.75, 2.75) {$\states{\contstaraut}$};
		\node [style=none] (20) at (-5.75, 4.75) {};
		\node [style=none] (21) at (-2.75, 4.75) {};
		\node [style=none] (22) at (-5.75, 1.25) {};
		\node [style=none] (23) at (-2.75, 1.25) {};
		\node [style=none] (24) at (-1, 3) {};
		\node [style=none] (25) at (1.75, 4.25) {$\states{\system^*}$};
		\node [style=none] (26) at (2.25, 3.75) {};
		\node [style=function black] (27) at (-5.25, 2.25) {};
		\node [style=none] (28) at (-4.35, 3.75) {};
		\node [style=none] (29) at (-4.5, -1.75) {};
		\node [style=none] (30) at (-3, -1.75) {};
		\node [style=none] (31) at (-4.5, -2.75) {};
		\node [style=none] (32) at (-3, -2.75) {};
		\node [style=none] (33) at (-2, -2.25) {};
		\node [style=none] (34) at (-3.25, -2.25) {$\diamondupdate{\System^*}$};
		\node [style=none] (35) at (-4.5, -2.25) {};
		\node [style=none] (36) at (-6.25, -1.75) {$\states{\system^*}$};
		\node [style=none] (37) at (-0.5, -1.75) {$\states{\system^*}$};
		\node [style=none] (38) at (-8.75, -1.75) {};
		\node [style=none] (39) at (-7.75, -1.75) {};
		\node [style=none] (40) at (-8.75, -2.75) {};
		\node [style=none] (41) at (-7.75, -2.75) {};
		\node [style=none] (42) at (-6.75, -2.25) {};
		\node [style=none] (43) at (-7.75, -2.25) {$\indmodelsys$};
		\node [style=function black] (44) at (0, -2.25) {};
		\node [style=none] (45) at (-8.75, -2.25) {};
		\node [style=none] (46) at (-10.25, -2.25) {};
		\node [style=none] (47) at (-9.75, -1.75) {$\states{\contstaraut}$};
		\node [style=none] (48) at (-5.75, 0) {};
		\node [style=none] (49) at (-2.75, 0) {};
		\node [style=none] (50) at (-5.75, -3.5) {};
		\node [style=none] (51) at (-2.75, -3.5) {};
		\node [style=none] (52) at (-1, -1.75) {};
		\node [style=none] (53) at (9.25, -0.25) {$\states{\system^*}$};
		\node [style=none] (54) at (9.75, -0.75) {};
		\node [style=function black] (55) at (-5.25, -2.25) {};
		\node [style=none] (56) at (-4.35, -0.75) {};
		\node [style=box-1-.5] (57) at (3.375, -3.75) {$\update_{\Contstaraut}$};
		\node [style=none] (58) at (1.25, -2.75) {};
		\node [style=none] (59) at (5.5, -2.75) {};
		\node [style=none] (60) at (1.25, -4.75) {};
		\node [style=none] (61) at (5.5, -4.75) {};
		\node [style=none] (62) at (6.75, -3.25) {};
		\node [style=none] (63) at (7.75, -3.25) {};
		\node [style=none] (64) at (6.75, -4.25) {};
		\node [style=none] (65) at (7.75, -4.25) {};
		\node [style=none] (66) at (8.75, -3.75) {};
		\node [style=none] (67) at (7.75, -3.75) {$\indmodelsys$};
		\node [style=none] (68) at (9.25, -3.25) {$\states{\system^*}$};
		\node [style=none] (69) at (6.75, -3.75) {};
		\node [style=none] (70) at (9.75, -3.75) {};
		\node [style=function black] (71) at (0.5, -0.75) {};
		\node [style=function black] (72) at (1.775, -3.275) {};
		\node [style=none] (73) at (1.275, -3.275) {};
		\node [style=function black] (74) at (-9.5, -2.25) {};
		\node [style=none] (75) at (-6.25, -3.75) {};
		\node [style=none] (76) at (6, -3.25) {$\states{\contstaraut}$};
	\end{pgfonlayer}
	\begin{pgfonlayer}{edgelayer}
		\draw [style=thick-line] (0.center) to (1.center);
		\draw [style=thick-line] (0.center) to (2.center);
		\draw [style=thick-line] (2.center) to (3.center);
		\draw [style=thick-line, in=90, out=0] (1.center) to (4.center);
		\draw [style=thick-line, in=-90, out=0] (3.center) to (4.center);
		\draw [style=thick-line] (9.center) to (10.center);
		\draw [style=thick-line] (9.center) to (11.center);
		\draw [style=thick-line] (11.center) to (12.center);
		\draw [style=thick-line, in=90, out=0] (10.center) to (13.center);
		\draw [style=thick-line, in=-90, out=0] (12.center) to (13.center);
		\draw [style=thick-line] (13.center) to (6.center);
		\draw [style=thick-line] (18.center) to (17.center);
		\draw [style=big dashes] (24.center)
			 to [in=0, out=-90] (23.center)
			 to (22.center)
			 to (20.center)
			 to (21.center)
			 to [in=90, out=0] cycle;
		\draw [style=thick-line] (4.center) to (16.center);
		\draw [style=thick-line, in=180, out=90, looseness=0.75] (27) to (28.center);
		\draw [style=thick-line] (28.center) to (26.center);
		\draw [style=thick-line] (29.center) to (30.center);
		\draw [style=thick-line] (29.center) to (31.center);
		\draw [style=thick-line] (31.center) to (32.center);
		\draw [style=thick-line, in=90, out=0] (30.center) to (33.center);
		\draw [style=thick-line, in=-90, out=0] (32.center) to (33.center);
		\draw [style=thick-line] (38.center) to (39.center);
		\draw [style=thick-line] (38.center) to (40.center);
		\draw [style=thick-line] (40.center) to (41.center);
		\draw [style=thick-line, in=90, out=0] (39.center) to (42.center);
		\draw [style=thick-line, in=-90, out=0] (41.center) to (42.center);
		\draw [style=thick-line] (42.center) to (35.center);
		\draw [style=thick-line] (46.center) to (45.center);
		\draw [style=big dashes] (52.center)
			 to [in=0, out=90] (49.center)
			 to (48.center)
			 to (50.center)
			 to (51.center)
			 to [in=-90, out=0] cycle;
		\draw [style=thick-line, in=180, out=90, looseness=0.75] (55) to (56.center);
		\draw [style=thick-line] (56.center) to (54.center);
		\draw [style=big dashes] (60.center)
			 to (58.center)
			 to (59.center)
			 to (61.center)
			 to cycle;
		\draw [style=thick-line] (62.center) to (63.center);
		\draw [style=thick-line] (62.center) to (64.center);
		\draw [style=thick-line] (64.center) to (65.center);
		\draw [style=thick-line, in=90, out=0] (63.center) to (66.center);
		\draw [style=thick-line, in=-90, out=0] (65.center) to (66.center);
		\draw [style=thick-line] (66.center) to (70.center);
		\draw [style=thick-line] (57) to (69.center);
		\draw [style=thick-line, in=165, out=-90, looseness=0.75] (71) to (73.center);
		\draw [style=thick-line] (73.center) to (72);
		\draw [style=thick-line, in=180, out=-90] (74) to (75.center);
		\draw [style=thick-line] (75.center) to (57);
		\draw [style=thick-line] (33.center) to (44);
	\end{pgfonlayer}
\end{tikzpicture}}
    \end{equation}
\end{example}

Similarly, \Cref{th:imp} shows that the autonomous attracting controller models the attracting environment $\Env^*$, and by using
\Cref{th:IMPmodelimpliesBayesmodel} we then also have the following Bayesian filtering interpretation.

\begin{example}
    Define $\reasoner:  \env^* \otimes \contstaraut \to \contstaraut$ as
    \begin{equation}
        \scalebox{0.8}{\begin{tikzpicture}
	\begin{pgfonlayer}{nodelayer}
		\node [style=none] (0) at (-9, 0.75) {};
		\node [style=none] (1) at (-6.75, 0.75) {};
		\node [style=none] (2) at (-0.7, 0) {};
		\node [style=none] (3) at (-6.75, -0.75) {};
		\node [style=none] (4) at (-8, -0.25) {$\contstaraut$};
		\node [style=none] (5) at (-1.75, 0.5) {$\contstaraut$};
		\node [style=none] (6) at (-8, 1.25) {$\env^*$};
		\node [style=none] (7) at (-9, -0.75) {};
		\node [style=none] (8) at (0.75, 0) {$\coloneqq$};
		\node [style=none] (9) at (2.25, 0.5) {};
		\node [style=function black] (10) at (4, 0.5) {};
		\node [style=none] (11) at (10.55, 0) {};
		\node [style=none] (12) at (4.5, -0.5) {};
		\node [style=none] (13) at (3.25, 0) {$\contstaraut$};
		\node [style=none] (14) at (9.5, 0.5) {$\contstaraut$};
		\node [style=none] (15) at (3.25, 1) {$\env^*$};
		\node [style=none] (16) at (8.5, 0) {};
		\node [style=box-2-1] (17) at (6.5, 0) {$\update_{\Contstaraut}$};
		\node [style=none] (18) at (2.25, -0.5) {};
		\node [style=none] (19) at (8.5, 0) {};
		\node [style=none] (20) at (-6.75, 1.25) {};
		\node [style=none] (21) at (-2.95, 1.25) {};
		\node [style=none] (22) at (-6.75, -1.25) {};
		\node [style=none] (23) at (-2.95, -1.25) {};
		\node [style=none] (24) at (-2.95, 0) {};
		\node [style=none] (25) at (-4.75, 0) {$\reasoner$};
	\end{pgfonlayer}
	\begin{pgfonlayer}{edgelayer}
		\draw [style=thick-line] (0.center) to (1.center);
		\draw [style=thick-line] (7.center) to (3.center);
		\draw [style=thick-line] (9.center) to (10);
		\draw [style=thick-line] (18.center) to (12.center);
		\draw [style=thick-line] (11.center) to (19.center);
		\draw [style=thick-line] (20.center) to (21.center);
		\draw [style=thick-line] (20.center) to (22.center);
		\draw [style=thick-line] (22.center) to (23.center);
		\draw [style=thick-line] (2.center) to (24.center);
		\draw [style=thick-line] (21.center) to (23.center);
	\end{pgfonlayer}
\end{tikzpicture}}
    \end{equation}
    and $\hiddenmodel: \env^* \markto \env^* \otimes \env^*$ as:
    \begin{equation}
        \label{eqn:attractingEnvironment}
        \scalebox{0.8}{\begin{tikzpicture}
	\begin{pgfonlayer}{nodelayer}
		\node [style=none] (0) at (1.5, -0.75) {};
		\node [style=none] (1) at (3.9, -0.25) {};
		\node [style=none] (2) at (5.65, -0.25) {};
		\node [style=none] (3) at (3.9, -1.25) {};
		\node [style=none] (4) at (5.65, -1.25) {};
		\node [style=none] (5) at (6.9, -0.75) {};
		\node [style=none] (6) at (5.275, -0.75) {$\diamondupdate{\Env^*}$};
		\node [style=none] (7) at (3.9, -0.75) {};
		\node [style=function black] (8) at (3, -0.75) {};
		\node [style=none] (9) at (3.9, 0.75) {};
		\node [style=none] (10) at (8.5, -0.75) {};
		\node [style=none] (11) at (7.75, 1.25) {$\env^*$};
		\node [style=none] (12) at (7.75, -0.25) {$\env^*$};
		\node [style=none] (13) at (8.5, 0.75) {};
		\node [style=none] (14) at (2.25, -0.25) {$\env^*$};
		\node [style=none] (15) at (-9.75, 0) {};
		\node [style=none] (16) at (-7.5, 0) {};
		\node [style=none] (17) at (-1.5, 0.75) {};
		\node [style=none] (18) at (-1.5, -0.75) {};
		\node [style=none] (19) at (-8.75, 0.5) {$\env^*$};
		\node [style=none] (20) at (-2.25, 1.25) {$\env^*$};
		\node [style=none] (21) at (-2.25, -0.25) {$\env^*$};
		\node [style=none] (22) at (-7.5, 1.25) {};
		\node [style=none] (23) at (-5, 1.25) {};
		\node [style=none] (24) at (-7.5, -1.25) {};
		\node [style=none] (25) at (-5, -1.25) {};
		\node [style=none] (26) at (-3.5, 0) {};
		\node [style=none] (27) at (-5.5, 0) {$\hiddenmodel$};
		\node [style=none] (28) at (-3.8, -0.75) {};
		\node [style=none] (29) at (-3.8, 0.75) {};
		\node [style=none] (30) at (0.1, 0) {$\coloneqq$};
	\end{pgfonlayer}
	\begin{pgfonlayer}{edgelayer}
		\draw [style=thick-line] (1.center) to (2.center);
		\draw [style=thick-line] (1.center) to (3.center);
		\draw [style=thick-line] (3.center) to (4.center);
		\draw [style=thick-line, in=105, out=0] (2.center) to (5.center);
		\draw [style=thick-line, in=-105, out=0] (4.center) to (5.center);
		\draw [style=thick-line, in=180, out=90, looseness=0.75] (8) to (9.center);
		\draw [style=thick-line] (15.center) to (16.center);
		\draw [style=thick-line] (22.center) to (23.center);
		\draw [style=thick-line] (22.center) to (24.center);
		\draw [style=thick-line] (24.center) to (25.center);
		\draw [style=thick-line, in=90, out=0] (23.center) to (26.center);
		\draw [style=thick-line, in=-90, out=0] (25.center) to (26.center);
		\draw [style=thick-line] (28.center) to (18.center);
		\draw [style=thick-line] (17.center) to (29.center);
		\draw [style=thick-line] (0.center) to (7.center);
		\draw [style=thick-line] (5.center) to (10.center);
		\draw [style=thick-line] (9.center) to (13.center);
	\end{pgfonlayer}
\end{tikzpicture}}
    \end{equation}
    Then $\hiddenmodel$ is the hidden Markov model, and $\indmodelenv : \contstaraut \markto \env^*$ the interpretation map of a Bayesian filtering interpretation of $\reasoner$, i.e. we have:
    \begin{equation}
        \scalebox{0.7}{\begin{tikzpicture}
	\begin{pgfonlayer}{nodelayer}
		\node [style=none] (77) at (-5, 3) {};
		\node [style=none] (78) at (-3.5, 3) {};
		\node [style=none] (79) at (-5, 2) {};
		\node [style=none] (80) at (-3.5, 2) {};
		\node [style=none] (81) at (-2.5, 2.5) {};
		\node [style=none] (82) at (-3.75, 2.5) {$\diamondupdate{\Env^*}$};
		\node [style=none] (83) at (-5, 2.5) {};
		\node [style=none] (84) at (-6.75, 3) {$\states{\env^*}$};
		\node [style=none] (85) at (1.25, 3) {$\states{\env^*}$};
		\node [style=none] (86) at (-9.25, 3) {};
		\node [style=none] (87) at (-8.25, 3) {};
		\node [style=none] (88) at (-9.25, 2) {};
		\node [style=none] (89) at (-8.25, 2) {};
		\node [style=none] (90) at (-7.25, 2.5) {};
		\node [style=none] (91) at (-8.25, 2.5) {$\indmodelenv$};
		\node [style=none] (92) at (3.25, 3) {=};
		\node [style=none] (93) at (1.75, 2.5) {};
		\node [style=none] (94) at (-9.25, 2.5) {};
		\node [style=none] (95) at (-10.75, 2.5) {};
		\node [style=none] (96) at (-10.25, 3) {$\states{\contstaraut}$};
		\node [style=none] (97) at (-6.25, 5) {};
		\node [style=none] (98) at (-3.25, 5) {};
		\node [style=none] (99) at (-6.25, 1.5) {};
		\node [style=none] (100) at (-3.25, 1.5) {};
		\node [style=none] (101) at (-1.5, 3.25) {};
		\node [style=none] (102) at (1.25, 4.5) {$\states{\env^*}$};
		\node [style=none] (103) at (1.75, 4) {};
		\node [style=function black] (104) at (-5.75, 2.5) {};
		\node [style=none] (105) at (-4.85, 4) {};
		\node [style=none] (106) at (-5, -1.5) {};
		\node [style=none] (107) at (-3.5, -1.5) {};
		\node [style=none] (108) at (-5, -2.5) {};
		\node [style=none] (109) at (-3.5, -2.5) {};
		\node [style=none] (110) at (-2.5, -2) {};
		\node [style=none] (111) at (-3.75, -2) {$\diamondupdate{\Env^*}$};
		\node [style=none] (112) at (-5, -2) {};
		\node [style=none] (113) at (-6.75, -1.5) {$\states{\env^*}$};
		\node [style=none] (114) at (-1, -1.5) {$\states{\env^*}$};
		\node [style=none] (115) at (-9.25, -1.5) {};
		\node [style=none] (116) at (-8.25, -1.5) {};
		\node [style=none] (117) at (-9.25, -2.5) {};
		\node [style=none] (118) at (-8.25, -2.5) {};
		\node [style=none] (119) at (-7.25, -2) {};
		\node [style=none] (120) at (-8.25, -2) {$\indmodelenv$};
		\node [style=function black] (121) at (-0.5, -2) {};
		\node [style=none] (122) at (-9.25, -2) {};
		\node [style=none] (123) at (-10.75, -2) {};
		\node [style=none] (124) at (-10.25, -1.5) {$\states{\contstaraut}$};
		\node [style=none] (125) at (-6.25, 0.25) {};
		\node [style=none] (126) at (-3.25, 0.25) {};
		\node [style=none] (127) at (-6.25, -3.25) {};
		\node [style=none] (128) at (-3.25, -3.25) {};
		\node [style=none] (129) at (-1.5, -1.5) {};
		\node [style=none] (130) at (8.75, 0) {$\states{\env^*}$};
		\node [style=none] (131) at (9.25, -0.5) {};
		\node [style=function black] (132) at (-5.75, -2) {};
		\node [style=none] (133) at (-4.85, -0.5) {};
		\node [style=box-1-.5] (134) at (2.875, -3.5) {$\update_{\Contstaraut}$};
		\node [style=none] (135) at (0.75, -2.5) {};
		\node [style=none] (136) at (5, -2.5) {};
		\node [style=none] (137) at (0.75, -4.5) {};
		\node [style=none] (138) at (5, -4.5) {};
		\node [style=none] (139) at (6.25, -3) {};
		\node [style=none] (140) at (7.25, -3) {};
		\node [style=none] (141) at (6.25, -4) {};
		\node [style=none] (142) at (7.25, -4) {};
		\node [style=none] (143) at (8.25, -3.5) {};
		\node [style=none] (144) at (7.25, -3.5) {$\indmodelenv$};
		\node [style=none] (145) at (8.75, -3) {$\states{\env^*}$};
		\node [style=none] (146) at (6.25, -3.5) {};
		\node [style=none] (147) at (9.25, -3.5) {};
		\node [style=function black] (148) at (0, -0.5) {};
		\node [style=function black] (149) at (1.275, -3.025) {};
		\node [style=none] (150) at (0.775, -3.025) {};
		\node [style=function black] (151) at (-10, -2) {};
		\node [style=none] (152) at (-6.75, -3.5) {};
		\node [style=none] (153) at (5.5, -3) {$\states{\contstaraut}$};
	\end{pgfonlayer}
	\begin{pgfonlayer}{edgelayer}
		\draw [style=thick-line] (77.center) to (78.center);
		\draw [style=thick-line] (77.center) to (79.center);
		\draw [style=thick-line] (79.center) to (80.center);
		\draw [style=thick-line, in=90, out=0] (78.center) to (81.center);
		\draw [style=thick-line, in=-90, out=0] (80.center) to (81.center);
		\draw [style=thick-line] (86.center) to (87.center);
		\draw [style=thick-line] (86.center) to (88.center);
		\draw [style=thick-line] (88.center) to (89.center);
		\draw [style=thick-line, in=90, out=0] (87.center) to (90.center);
		\draw [style=thick-line, in=-90, out=0] (89.center) to (90.center);
		\draw [style=thick-line] (90.center) to (83.center);
		\draw [style=thick-line] (95.center) to (94.center);
		\draw [style=big dashes] (101.center)
			 to [in=0, out=-90] (100.center)
			 to (99.center)
			 to (97.center)
			 to (98.center)
			 to [in=90, out=0] cycle;
		\draw [style=thick-line] (81.center) to (93.center);
		\draw [style=thick-line, in=180, out=90, looseness=0.75] (104) to (105.center);
		\draw [style=thick-line] (105.center) to (103.center);
		\draw [style=thick-line] (106.center) to (107.center);
		\draw [style=thick-line] (106.center) to (108.center);
		\draw [style=thick-line] (108.center) to (109.center);
		\draw [style=thick-line, in=90, out=0] (107.center) to (110.center);
		\draw [style=thick-line, in=-90, out=0] (109.center) to (110.center);
		\draw [style=thick-line] (115.center) to (116.center);
		\draw [style=thick-line] (115.center) to (117.center);
		\draw [style=thick-line] (117.center) to (118.center);
		\draw [style=thick-line, in=90, out=0] (116.center) to (119.center);
		\draw [style=thick-line, in=-90, out=0] (118.center) to (119.center);
		\draw [style=thick-line] (119.center) to (112.center);
		\draw [style=thick-line] (123.center) to (122.center);
		\draw [style=big dashes] (129.center)
			 to [in=0, out=90] (126.center)
			 to (125.center)
			 to (127.center)
			 to (128.center)
			 to [in=-90, out=0] cycle;
		\draw [style=thick-line, in=180, out=90, looseness=0.75] (132) to (133.center);
		\draw [style=thick-line] (133.center) to (131.center);
		\draw [style=big dashes] (137.center)
			 to (135.center)
			 to (136.center)
			 to (138.center)
			 to cycle;
		\draw [style=thick-line] (139.center) to (140.center);
		\draw [style=thick-line] (139.center) to (141.center);
		\draw [style=thick-line] (141.center) to (142.center);
		\draw [style=thick-line, in=90, out=0] (140.center) to (143.center);
		\draw [style=thick-line, in=-90, out=0] (142.center) to (143.center);
		\draw [style=thick-line] (143.center) to (147.center);
		\draw [style=thick-line] (134) to (146.center);
		\draw [style=thick-line, in=165, out=-90, looseness=0.75] (148) to (150.center);
		\draw [style=thick-line] (150.center) to (149);
		\draw [style=thick-line, in=180, out=-90] (151) to (152.center);
		\draw [style=thick-line] (152.center) to (134);
		\draw [style=thick-line] (110.center) to (121);
	\end{pgfonlayer}
\end{tikzpicture}}
    \end{equation}
\end{example}

These examples show the precise sense in which one can understand a controller as modelling 1) the full system that contains it and 2) its environment from a Bayesian perspective: under the IMP assumptions~\cite{wonham1976towards, hepburn1981internal, hepburn1984structurally, hepburn1984error}, a model in the control theoretic sense described by \cref{def:model} admits a Bayesian filtering interpretation with reasoner and (hidden Markov) model, as described by \cref{def:interpretation} following work by~\cite{virgo2021interpreting, biehl2022interpreting}, given above.

The second example also provides a clearer understanding of the statement that the environment can be seen as a ``convenient fiction by which the designer may specify (or analyst describe) precisely the class of tracking and disturbance rejection tasks which the controlled system is to accomplish with zero (asymptotic) error''~\cite{wonham1976towards}, since the environment now just appears in \cref{eqn:attractingEnvironment} as a special case of a full-fledged \emph{epistemic} (see~\cite{virgo2023unifilar}) Bayesian model $\hiddenmodel$ for a reasoner, i.e. while it could in principle capture properties of the physical world where the controller is instantiated, it effectively only needs to obey the consistency equation expressed by a Bayesian filtering interpretation.

\section{Conclusions and future work}
\label{sec:conclusion}
The idea of ``internal models'' has appeared in a number of different research fields, including control theory, biology, artificial intelligence and cognitive science~\cite{conant1970every, francis1976internal, kawato1999internal, yi2000robust, russell2009artificial, seth2014cybernetic, ha2018world}.
These notions of internal models appeal, at least on the surface, to a common intuition of a system modelling another system in order to achieve a goal, e.g.~a reinforcement learning agent forming a model of its world to maximise the sum of expected rewards, or a cognitive system performing Bayesian inference on the hidden states generating its observations with a model of the environment to perform a certain task.
It is however unclear in the literature whether this goes beyond a simple analogy, and whether different notions of internal models can be captured by a common mathematical theory.

In this work we provided a first investigation on different notions of ``models'': one that is used in the control theoretic context of the internal model principle~\cite{francis1976internal, sontag2003adaptation, bin2022internal}, and one from work on Bayesian interpretations~\cite{virgo2021interpreting, biehl2022interpreting}.
We formally connected the two by showing that the notion of model for autonomous systems in the internal model principle as formulated in~\cite{hepburn1981internal, hepburn1984structurally, hepburn1984error} can be seen as a special case of the more general Bayesian filtering interpretations proposed  by~\cite{virgo2021interpreting, biehl2022interpreting}.

More specifically, \cref{th:IMPmodelimpliesBayesmodel} tells us that the definition of a model used in the IMP literature, i.e. the special case of our \cref{def:model} for autonomous systems~\cite{wonham1976towards, hepburn1981internal, hepburn1984structurally, hepburn1984error, wonham2019supervisory}, induces a Bayesian filtering interpretation.

Moreover, it also tells us that the IMP definition of a model is too restrictive from a Bayesian perspective: Bayesian filtering interpretations are a more general formalisation of what it means for a system to model another one.
The Bayesian filtering interpretations induced by the control-theoretic models are such that (1) the observations \emph{need not} be taken into account in order to update beliefs about the hidden states, and (2) those beliefs are always disjoint.

{\appendix
\section{Proofs}
\subsection{Proof of \cref{th:IndexedModel}}
\begin{proof}
\label{prf:IndexedModel}
    For this proof, we will use string diagrams as presented in \cref{sec:interpretationMap} to denote the arrows of the Markov category $\RelPlus$ since all the maps we need are morphisms in that category.
    First, we translate \cref{def:model} for autonomous systems (see \cref{rmk:HWIMP}) into string diagrams, obtaining:
    \begin{equation}
        \label{eqn:proofCommutative}
        \scalebox{0.8}{\begin{tikzpicture}
	\begin{pgfonlayer}{nodelayer}
		\node [style=box-1-.5] (0) at (7.25, 0) {$\update_\SysM$};
		\node [style=none] (1) at (9.75, 0) {};
		\node [style=box-1-.5] (2) at (-7.25, 0) {$\update_\SysX$};
		\node [style=none] (3) at (-9.75, 0) {};
		\node [style=box-1-.5] (4) at (-3.75, 0) {$\model_\sts$};
		\node [style=none] (5) at (-1.25, 0) {};
		\node [style=none] (6) at (0, 0) {=};
		\node [style=box-1-.5] (7) at (3.75, 0) {$\model_\sts$};
		\node [style=none] (8) at (1.25, 0) {};
		\node [style=none] (9) at (-9, 0.5) {$\sysX$};
		\node [style=none] (10) at (-5.5, 0.5) {$\sysX$};
		\node [style=none] (11) at (2, 0.5) {$\sysX$};
		\node [style=none] (12) at (5.5, 0.5) {$\sysM$};
		\node [style=none] (13) at (-2, 0.5) {$\sysM$};
		\node [style=none] (14) at (9, 0.5) {$\sysM$};
	\end{pgfonlayer}
	\begin{pgfonlayer}{edgelayer}
		\draw [style=thick-line] (0) to (1.center);
		\draw [style=thick-line] (3.center) to (2);
		\draw [style=thick-line] (4) to (5.center);
		\draw [style=thick-line, in=180, out=0] (2) to (4);
		\draw [style=thick-line] (8.center) to (7);
		\draw [style=thick-line] (7) to (0);
	\end{pgfonlayer}
\end{tikzpicture}
}
    \end{equation}
    Then we note that in $\RelPlus$, we have that $\indmodel_\sts: \sysM \markto \sysX$ is a right inverse of the surjective map $\model_\sts: \sysX \rrightarrow \sysM$, i.e.:
    \begin{equation}
        \label{eqn:proofSurjective}
        \scalebox{0.8}{\begin{tikzpicture}
	\begin{pgfonlayer}{nodelayer}
		\node [style=none] (0) at (-6, 0) {};
		\node [style=none] (1) at (-6, 0.5) {};
		\node [style=none] (2) at (-5, 0.5) {};
		\node [style=none] (3) at (-6, -0.5) {};
		\node [style=none] (4) at (-5, -0.5) {};
		\node [style=none] (5) at (-4, 0) {};
		\node [style=none] (6) at (-5, 0) {$\indmodel_\sts$};
		\node [style=none] (7) at (-7.5, 0) {};
		\node [style=box-1-.5] (8) at (-1.5, 0) {$\model_\sts$};
		\node [style=none] (9) at (1.25, 0) {};
		\node [style=none] (10) at (2.5, 0) {=};
		\node [style=none] (11) at (3.75, 0) {};
		\node [style=none] (12) at (7.5, 0) {};
		\node [style=none] (13) at (-6.75, 0.5) {$\sysM$};
		\node [style=none] (14) at (-3.25, 0.5) {$\sysX$};
		\node [style=none] (15) at (0.5, 0.5) {$\sysM$};
		\node [style=none] (16) at (5.5, 0.5) {$\sysM$};
	\end{pgfonlayer}
	\begin{pgfonlayer}{edgelayer}
		\draw [style=thick-line] (1.center) to (2.center);
		\draw [style=thick-line] (1.center) to (3.center);
		\draw [style=thick-line] (3.center) to (4.center);
		\draw [style=thick-line, in=90, out=0] (2.center) to (5.center);
		\draw [style=thick-line, in=-90, out=0] (4.center) to (5.center);
		\draw [style=thick-line] (7.center) to (0.center);
		\draw [style=thick-line] (8) to (9.center);
		\draw [style=thick-line] (5.center) to (8);
		\draw [style=thick-line] (11.center) to (12.center);
	\end{pgfonlayer}
\end{tikzpicture}
}
    \end{equation}
    Furthermore, we also have the definition of the updates induced by the Hepburn--Wonham IMP, $\diamondupdate{\SysX}$, see \cref{eqn:coarsedUpdates}:
    \begin{equation}
        \label{eqn:proofCoarseUpdatesAppx}
        \scalebox{0.8}{\begin{tikzpicture}
	\begin{pgfonlayer}{nodelayer}
		\node [style=none] (4) at (5, 0.5) {};
		\node [style=none] (5) at (6, 0.5) {};
		\node [style=none] (6) at (5, -0.5) {};
		\node [style=none] (7) at (6, -0.5) {};
		\node [style=none] (8) at (7, 0) {};
		\node [style=none] (9) at (6, 0) {$\indmodel_\sts$};
		\node [style=none] (10) at (8.5, 0) {};
		\node [style=none] (11) at (-9.5, 0) {};
		\node [style=none] (12) at (-9.5, 0.5) {};
		\node [style=none] (13) at (-8.5, 0.5) {};
		\node [style=none] (14) at (-9.5, -0.5) {};
		\node [style=none] (15) at (-8.5, -0.5) {};
		\node [style=none] (16) at (-7.5, 0) {};
		\node [style=none] (17) at (-8.5, 0) {$\diamondupdate{\SysX}$};
		\node [style=none] (18) at (-11, 0) {};
		\node [style=none] (19) at (-6, 0) {};
		\node [style=none] (20) at (-4.75, 0) {$\coloneqq$};
		\node [style=box-1-.5] (21) at (-1, 0) {$\update_\SysX$};
		\node [style=none] (22) at (-3.5, 0) {};
		\node [style=box-1-.5] (23) at (2.5, 0) {$\model_\sts$};
		\node [style=none] (24) at (5, 0) {};
		\node [style=none] (25) at (-2.75, 0.5) {$\sysX$};
		\node [style=none] (26) at (0.75, 0.5) {$\sysX$};
		\node [style=none] (27) at (4.25, 0.5) {$\sysM$};
		\node [style=none] (28) at (-10.25, 0.5) {$\sysX$};
		\node [style=none] (29) at (-6.75, 0.5) {$\sysX$};
		\node [style=none] (30) at (7.75, 0.5) {$\sysX$};
	\end{pgfonlayer}
	\begin{pgfonlayer}{edgelayer}
		\draw [style=thick-line] (4.center) to (5.center);
		\draw [style=thick-line] (4.center) to (6.center);
		\draw [style=thick-line] (6.center) to (7.center);
		\draw [style=thick-line, in=90, out=0] (5.center) to (8.center);
		\draw [style=thick-line, in=-90, out=0] (7.center) to (8.center);
		\draw [style=thick-line] (8.center) to (10.center);
		\draw [style=thick-line] (12.center) to (13.center);
		\draw [style=thick-line] (12.center) to (14.center);
		\draw [style=thick-line] (14.center) to (15.center);
		\draw [style=thick-line, in=90, out=0] (13.center) to (16.center);
		\draw [style=thick-line, in=-90, out=0] (15.center) to (16.center);
		\draw [style=thick-line] (18.center) to (11.center);
		\draw [style=thick-line] (16.center) to (19.center);
		\draw [style=thick-line] (22.center) to (21);
		\draw [style=thick-line] (23) to (24.center);
		\draw [style=thick-line] (21) to (23);
	\end{pgfonlayer}
\end{tikzpicture}
}
    \end{equation}

    We then prove the following, starting from \cref{eqn:proofCommutative}, we can (i) post-compose both sides with $\indmodel_\sts$, (ii) apply \cref{eqn:proofCoarseUpdatesAppx} on the left hand side, (iii) pre-compose both sides with $\indmodel_\sts$, and then (iv) use the surjectivity of $\model_\sts$ (\cref{eqn:proofSurjective}) to get:
    \begin{equation}
        \label{eqn:proofBelief}
        \scalebox{0.55}{\begin{tikzpicture}
	\begin{pgfonlayer}{nodelayer}
		\node [style=none] (89) at (0, 1) {=};
		\node [style=none] (111) at (-4.75, 1) {};
		\node [style=none] (112) at (-4.75, 1.5) {};
		\node [style=none] (113) at (-3.75, 1.5) {};
		\node [style=none] (114) at (-4.75, 0.5) {};
		\node [style=none] (115) at (-3.75, 0.5) {};
		\node [style=none] (116) at (-2.75, 1) {};
		\node [style=none] (117) at (-3.75, 1) {$\diamondupdate{\SysX}$};
		\node [style=none] (118) at (-6.25, 1) {};
		\node [style=none] (119) at (-1.25, 1) {};
		\node [style=none] (121) at (0, -1) {=};
		\node [style=none] (132) at (-4.75, -1) {};
		\node [style=none] (133) at (-4.75, -0.5) {};
		\node [style=none] (134) at (-3.75, -0.5) {};
		\node [style=none] (135) at (-4.75, -1.5) {};
		\node [style=none] (136) at (-3.75, -1.5) {};
		\node [style=none] (137) at (-2.75, -1) {};
		\node [style=none] (138) at (-3.75, -1) {$\diamondupdate{\SysX}$};
		\node [style=none] (140) at (-1.25, -1) {};
		\node [style=none] (158) at (0, -3) {=};
		\node [style=none] (185) at (1.25, -3) {};
		\node [style=box-1-.5] (186) at (-10.75, 3) {$\update_\SysX$};
		\node [style=none] (187) at (-13.25, 3) {};
		\node [style=box-1-.5] (188) at (-7.25, 3) {$\model_\sts$};
		\node [style=none] (190) at (0, 3) {=};
		\node [style=none] (191) at (-12.5, 3.5) {$\sysX$};
		\node [style=none] (192) at (-9, 3.5) {$\sysX$};
		\node [style=none] (193) at (-5.5, 3.5) {$\sysM$};
		\node [style=box-1-.5] (194) at (3.75, 3) {$\model_\sts$};
		\node [style=none] (195) at (1.25, 3) {};
		\node [style=box-1-.5] (196) at (7.25, 3) {$\update_\SysM$};
		\node [style=none] (197) at (9.75, 3) {};
		\node [style=none] (198) at (2, 3.5) {$\sysX$};
		\node [style=none] (199) at (5.5, 3.5) {$\sysM$};
		\node [style=none] (200) at (9, 3.5) {$\sysM$};
		\node [style=none] (201) at (9.75, 3.5) {};
		\node [style=none] (202) at (10.75, 3.5) {};
		\node [style=none] (203) at (9.75, 2.5) {};
		\node [style=none] (204) at (10.75, 2.5) {};
		\node [style=none] (205) at (11.75, 3) {};
		\node [style=none] (206) at (10.75, 3) {$\indmodel_\sts$};
		\node [style=none] (207) at (13.25, 3) {};
		\node [style=none] (209) at (12.5, 3.5) {$\sysX$};
		\node [style=none] (210) at (-4.75, 3.5) {};
		\node [style=none] (211) at (-3.75, 3.5) {};
		\node [style=none] (212) at (-4.75, 2.5) {};
		\node [style=none] (213) at (-3.75, 2.5) {};
		\node [style=none] (214) at (-2.75, 3) {};
		\node [style=none] (215) at (-3.75, 3) {$\indmodel_\sts$};
		\node [style=none] (216) at (-1.25, 3) {};
		\node [style=none] (217) at (-4.75, 3) {};
		\node [style=none] (218) at (-2, 3.5) {$\sysX$};
		\node [style=none] (219) at (-2, 1.5) {$\sysX$};
		\node [style=none] (220) at (-5.5, 1.5) {$\sysX$};
		\node [style=none] (221) at (-9, -0.5) {$\sysM$};
		\node [style=none] (222) at (-8.25, -0.5) {};
		\node [style=none] (223) at (-7.25, -0.5) {};
		\node [style=none] (224) at (-8.25, -1.5) {};
		\node [style=none] (225) at (-7.25, -1.5) {};
		\node [style=none] (226) at (-6.25, -1) {};
		\node [style=none] (227) at (-7.25, -1) {$\indmodel_\sts$};
		\node [style=none] (228) at (-4.75, -1) {};
		\node [style=none] (229) at (-8.25, -1) {};
		\node [style=none] (230) at (-5.5, -0.5) {$\sysX$};
		\node [style=none] (231) at (2, -0.5) {$\sysM$};
		\node [style=none] (232) at (2.75, -0.5) {};
		\node [style=none] (233) at (3.75, -0.5) {};
		\node [style=none] (234) at (2.75, -1.5) {};
		\node [style=none] (235) at (3.75, -1.5) {};
		\node [style=none] (237) at (3.75, -1) {$\indmodel_\sts$};
		\node [style=none] (239) at (2.75, -1) {};
		\node [style=none] (241) at (-9.75, -1) {};
		\node [style=box-1-.5] (242) at (3.75, 1) {$\model_\sts$};
		\node [style=none] (243) at (1.25, 1) {};
		\node [style=box-1-.5] (244) at (7.25, 1) {$\update_\SysM$};
		\node [style=none] (245) at (9.75, 1) {};
		\node [style=none] (246) at (2, 1.5) {$\sysX$};
		\node [style=none] (247) at (5.5, 1.5) {$\sysM$};
		\node [style=none] (248) at (9, 1.5) {$\sysM$};
		\node [style=none] (249) at (9.75, 1.5) {};
		\node [style=none] (250) at (10.75, 1.5) {};
		\node [style=none] (251) at (9.75, 0.5) {};
		\node [style=none] (252) at (10.75, 0.5) {};
		\node [style=none] (253) at (11.75, 1) {};
		\node [style=none] (254) at (10.75, 1) {$\indmodel_\sts$};
		\node [style=none] (255) at (13.25, 1) {};
		\node [style=none] (256) at (12.5, 1.5) {$\sysX$};
		\node [style=none] (257) at (1.25, -1) {};
		\node [style=box-1-.5] (258) at (7.25, -1) {$\model_\sts$};
		\node [style=none] (259) at (4.75, -1) {};
		\node [style=box-1-.5] (260) at (10.75, -1) {$\update_\SysM$};
		\node [style=none] (261) at (13.25, -1) {};
		\node [style=none] (262) at (5.5, -0.5) {$\sysX$};
		\node [style=none] (263) at (9, -0.5) {$\sysM$};
		\node [style=none] (264) at (12.5, -0.5) {$\sysM$};
		\node [style=none] (265) at (13.25, -0.5) {};
		\node [style=none] (266) at (14.25, -0.5) {};
		\node [style=none] (267) at (13.25, -1.5) {};
		\node [style=none] (268) at (14.25, -1.5) {};
		\node [style=none] (269) at (15.25, -1) {};
		\node [style=none] (270) at (14.25, -1) {$\indmodel_\sts$};
		\node [style=none] (271) at (16.75, -1) {};
		\node [style=none] (272) at (16, -0.5) {$\sysX$};
		\node [style=none] (273) at (-2, -0.5) {$\sysX$};
		\node [style=box-1-.5] (274) at (3.75, -3) {$\update_\SysM$};
		\node [style=none] (275) at (6.25, -3) {};
		\node [style=none] (276) at (5.5, -2.5) {$\sysM$};
		\node [style=none] (277) at (6.25, -2.5) {};
		\node [style=none] (278) at (7.25, -2.5) {};
		\node [style=none] (279) at (6.25, -3.5) {};
		\node [style=none] (280) at (7.25, -3.5) {};
		\node [style=none] (281) at (8.25, -3) {};
		\node [style=none] (282) at (7.25, -3) {$\indmodel_\sts$};
		\node [style=none] (283) at (9.75, -3) {};
		\node [style=none] (284) at (9, -2.5) {$\sysX$};
		\node [style=none] (285) at (-4.75, -3) {};
		\node [style=none] (286) at (-4.75, -2.5) {};
		\node [style=none] (287) at (-3.75, -2.5) {};
		\node [style=none] (288) at (-4.75, -3.5) {};
		\node [style=none] (289) at (-3.75, -3.5) {};
		\node [style=none] (290) at (-2.75, -3) {};
		\node [style=none] (291) at (-3.75, -3) {$\diamondupdate{\SysX}$};
		\node [style=none] (292) at (-1.25, -3) {};
		\node [style=none] (293) at (-9, -2.5) {$\sysM$};
		\node [style=none] (294) at (-8.25, -2.5) {};
		\node [style=none] (295) at (-7.25, -2.5) {};
		\node [style=none] (296) at (-8.25, -3.5) {};
		\node [style=none] (297) at (-7.25, -3.5) {};
		\node [style=none] (298) at (-6.25, -3) {};
		\node [style=none] (299) at (-7.25, -3) {$\indmodel_\sts$};
		\node [style=none] (300) at (-4.75, -3) {};
		\node [style=none] (301) at (-8.25, -3) {};
		\node [style=none] (302) at (-5.5, -2.5) {$\sysX$};
		\node [style=none] (303) at (-9.75, -3) {};
		\node [style=none] (304) at (-2, -2.5) {$\sysX$};
		\node [style=none] (305) at (-12.75, 1) {$\implies$};
		\node [style=none] (306) at (-12.75, -1) {$\implies$};
		\node [style=none] (307) at (-12.75, -3) {$\implies$};
		\node [style=none] (308) at (2, -2.5) {$\sysM$};
	\end{pgfonlayer}
	\begin{pgfonlayer}{edgelayer}
		\draw [style=thick-line] (112.center) to (113.center);
		\draw [style=thick-line] (112.center) to (114.center);
		\draw [style=thick-line] (114.center) to (115.center);
		\draw [style=thick-line, in=90, out=0] (113.center) to (116.center);
		\draw [style=thick-line, in=-90, out=0] (115.center) to (116.center);
		\draw [style=thick-line] (118.center) to (111.center);
		\draw [style=thick-line] (116.center) to (119.center);
		\draw [style=thick-line] (133.center) to (134.center);
		\draw [style=thick-line] (133.center) to (135.center);
		\draw [style=thick-line] (135.center) to (136.center);
		\draw [style=thick-line, in=90, out=0] (134.center) to (137.center);
		\draw [style=thick-line, in=-90, out=0] (136.center) to (137.center);
		\draw [style=thick-line] (137.center) to (140.center);
		\draw [style=thick-line] (187.center) to (186);
		\draw [style=thick-line] (186) to (188);
		\draw [style=thick-line] (195.center) to (194);
		\draw [style=thick-line] (196) to (197.center);
		\draw [style=thick-line] (194) to (196);
		\draw [style=thick-line] (201.center) to (202.center);
		\draw [style=thick-line] (201.center) to (203.center);
		\draw [style=thick-line] (203.center) to (204.center);
		\draw [style=thick-line, in=90, out=0] (202.center) to (205.center);
		\draw [style=thick-line, in=-90, out=0] (204.center) to (205.center);
		\draw [style=thick-line] (205.center) to (207.center);
		\draw [style=thick-line] (210.center) to (211.center);
		\draw [style=thick-line] (210.center) to (212.center);
		\draw [style=thick-line] (212.center) to (213.center);
		\draw [style=thick-line, in=90, out=0] (211.center) to (214.center);
		\draw [style=thick-line, in=-90, out=0] (213.center) to (214.center);
		\draw [style=thick-line] (214.center) to (216.center);
		\draw [style=thick-line] (188) to (217.center);
		\draw [style=thick-line] (222.center) to (223.center);
		\draw [style=thick-line] (222.center) to (224.center);
		\draw [style=thick-line] (224.center) to (225.center);
		\draw [style=thick-line, in=90, out=0] (223.center) to (226.center);
		\draw [style=thick-line, in=-90, out=0] (225.center) to (226.center);
		\draw [style=thick-line] (226.center) to (228.center);
		\draw [style=thick-line] (232.center) to (233.center);
		\draw [style=thick-line] (232.center) to (234.center);
		\draw [style=thick-line] (234.center) to (235.center);
		\draw [style=thick-line] (241.center) to (229.center);
		\draw [style=thick-line] (243.center) to (242);
		\draw [style=thick-line] (244) to (245.center);
		\draw [style=thick-line] (242) to (244);
		\draw [style=thick-line] (249.center) to (250.center);
		\draw [style=thick-line] (249.center) to (251.center);
		\draw [style=thick-line] (251.center) to (252.center);
		\draw [style=thick-line, in=90, out=0] (250.center) to (253.center);
		\draw [style=thick-line, in=-90, out=0] (252.center) to (253.center);
		\draw [style=thick-line] (253.center) to (255.center);
		\draw [style=thick-line] (257.center) to (239.center);
		\draw [style=thick-line] (259.center) to (258);
		\draw [style=thick-line] (260) to (261.center);
		\draw [style=thick-line] (258) to (260);
		\draw [style=thick-line] (265.center) to (266.center);
		\draw [style=thick-line] (265.center) to (267.center);
		\draw [style=thick-line] (267.center) to (268.center);
		\draw [style=thick-line, in=90, out=0] (266.center) to (269.center);
		\draw [style=thick-line, in=-90, out=0] (268.center) to (269.center);
		\draw [style=thick-line] (269.center) to (271.center);
		\draw [style=thick-line, in=90, out=0] (233.center) to (259.center);
		\draw [style=thick-line, in=-90, out=0] (235.center) to (259.center);
		\draw [style=thick-line] (274) to (275.center);
		\draw [style=thick-line] (277.center) to (278.center);
		\draw [style=thick-line] (277.center) to (279.center);
		\draw [style=thick-line] (279.center) to (280.center);
		\draw [style=thick-line, in=90, out=0] (278.center) to (281.center);
		\draw [style=thick-line, in=-90, out=0] (280.center) to (281.center);
		\draw [style=thick-line] (281.center) to (283.center);
		\draw [style=thick-line] (274) to (185.center);
		\draw [style=thick-line] (286.center) to (287.center);
		\draw [style=thick-line] (286.center) to (288.center);
		\draw [style=thick-line] (288.center) to (289.center);
		\draw [style=thick-line, in=90, out=0] (287.center) to (290.center);
		\draw [style=thick-line, in=-90, out=0] (289.center) to (290.center);
		\draw [style=thick-line] (290.center) to (292.center);
		\draw [style=thick-line] (294.center) to (295.center);
		\draw [style=thick-line] (294.center) to (296.center);
		\draw [style=thick-line] (296.center) to (297.center);
		\draw [style=thick-line, in=90, out=0] (295.center) to (298.center);
		\draw [style=thick-line, in=-90, out=0] (297.center) to (298.center);
		\draw [style=thick-line] (298.center) to (300.center);
		\draw [style=thick-line] (303.center) to (301.center);
	\end{pgfonlayer}
\end{tikzpicture}
}
    \end{equation}
    which gives us the string diagram version of \cref{eqn:beliefModel}.
\end{proof}

\subsection{Proof of \cref{th:IMPmodelimpliesBayesmodel}}
\begin{proof}
    \label{prf:beliefmap}
    For this proof, we will use string diagrams for the Markov category $\RelPlus$.
    We start with the following equation, derived from post-composing both sides of \cref{eqn:proofBelief} with $\model_\sts$ and subsequently using the surjectivity property of $\model_\sts$, see \cref{eqn:proofSurjective},
    \begin{equation}
        \label{eqn:proofDeterminisitic}
        \scalebox{0.8}{\begin{tikzpicture}
	\begin{pgfonlayer}{nodelayer}
		\node [style=none] (20) at (-3.5, 0) {$=$};
		\node [style=none] (31) at (-9.75, 0) {};
		\node [style=box-1-.5] (32) at (-7.25, 0) {$\update_\SysM$};
		\node [style=none] (33) at (-4.75, 0) {};
		\node [style=none] (34) at (-5.5, 0.5) {$\sysM$};
		\node [style=none] (35) at (-9, 0.5) {$\sysM$};
		\node [style=none] (37) at (2.75, 0) {};
		\node [style=none] (38) at (2.75, 0.5) {};
		\node [style=none] (39) at (3.75, 0.5) {};
		\node [style=none] (40) at (2.75, -0.5) {};
		\node [style=none] (41) at (3.75, -0.5) {};
		\node [style=none] (43) at (3.75, 0) {$\diamondupdate{\SysX}$};
		\node [style=none] (45) at (-1.5, 0.5) {$\sysM$};
		\node [style=none] (46) at (-0.75, 0.5) {};
		\node [style=none] (47) at (0.25, 0.5) {};
		\node [style=none] (48) at (-0.75, -0.5) {};
		\node [style=none] (49) at (0.25, -0.5) {};
		\node [style=none] (50) at (1.25, 0) {};
		\node [style=none] (51) at (0.25, 0) {$\indmodel_\sts$};
		\node [style=none] (52) at (2.75, 0) {};
		\node [style=none] (53) at (-0.75, 0) {};
		\node [style=none] (54) at (2, 0.5) {$\sysX$};
		\node [style=none] (55) at (-2.25, 0) {};
		\node [style=box-1-.5] (57) at (7.25, 0) {$\model_\sts$};
		\node [style=none] (58) at (4.75, 0) {};
		\node [style=none] (59) at (5.5, 0.5) {$\sysX$};
		\node [style=none] (60) at (9, 0.5) {$\sysM$};
		\node [style=none] (61) at (10, 0) {};
	\end{pgfonlayer}
	\begin{pgfonlayer}{edgelayer}
		\draw [style=thick-line] (32) to (33.center);
		\draw [style=thick-line] (32) to (31.center);
		\draw [style=thick-line] (38.center) to (39.center);
		\draw [style=thick-line] (38.center) to (40.center);
		\draw [style=thick-line] (40.center) to (41.center);
		\draw [style=thick-line] (46.center) to (47.center);
		\draw [style=thick-line] (46.center) to (48.center);
		\draw [style=thick-line] (48.center) to (49.center);
		\draw [style=thick-line, in=90, out=0] (47.center) to (50.center);
		\draw [style=thick-line, in=-90, out=0] (49.center) to (50.center);
		\draw [style=thick-line] (50.center) to (52.center);
		\draw [style=thick-line] (55.center) to (53.center);
		\draw [style=thick-line] (58.center) to (57);
		\draw (57) to (61.center);
		\draw [style=thick-line, in=90, out=0] (39.center) to (58.center);
		\draw [style=thick-line, in=-90, out=0] (41.center) to (58.center);
	\end{pgfonlayer}
\end{tikzpicture}
}
    \end{equation}
    This tells us that the composite $\indmodel_\sts \comp \diamondupdate{\SysX} \comp \model_\sts$ is deterministic, and thus allows us to apply the definition of \textit{positivity} for Markov categories~\cite[Definition 11.22]{fritz2020synthetic} (since every category with conditionals, thus including $\RelPlus$~\cite{fritz2023d}, is positive).
    We briefly recall the definition of a positive Markov category.
    \begin{definition}
        A Markov category $\Cat$ is positive if given two morphisms $f : A \to B$ and $g: B \to C$ such that their composite $f \comp g$ is deterministic implies the following
        \begin{equation}
            \label{eqn:proofPositive}
            \scalebox{0.8}{\begin{tikzpicture}
	\begin{pgfonlayer}{nodelayer}
		\node [style=none] (13) at (-9.25, 1) {$A$};
		\node [style=none] (14) at (-8.5, 1) {};
		\node [style=none] (15) at (-7.5, 1) {};
		\node [style=none] (16) at (-8.5, 0) {};
		\node [style=none] (17) at (-7.5, 0) {};
		\node [style=none] (18) at (-7.5, 0.5) {$f$};
		\node [style=none] (19) at (-8.5, 0.5) {};
		\node [style=none] (20) at (-10, 0.5) {};
		\node [style=none] (22) at (-6.5, 0.5) {};
		\node [style=none] (23) at (-5.75, 1) {$B$};
		\node [style=none] (24) at (-2.25, 1) {$C$};
		\node [style=none] (25) at (-5, -0.75) {};
		\node [style=function black] (26) at (-5.75, 0.5) {};
		\node [style=none] (43) at (-1.25, 0.5) {};
		\node [style=none] (44) at (-1.25, -0.75) {};
		\node [style=none] (45) at (-2.25, -0.25) {$B$};
		\node [style=none] (47) at (-4.75, 1) {};
		\node [style=none] (48) at (-3.75, 1) {};
		\node [style=none] (49) at (-4.75, 0) {};
		\node [style=none] (50) at (-3.75, 0) {};
		\node [style=none] (51) at (-3.75, 0.5) {$g$};
		\node [style=none] (52) at (-4.75, 0.5) {};
		\node [style=none] (53) at (-2.75, 0.5) {};
		\node [style=none] (54) at (2, 1) {$A$};
		\node [style=none] (55) at (2.75, 1) {};
		\node [style=none] (56) at (3.75, 1) {};
		\node [style=none] (57) at (2.75, 0) {};
		\node [style=none] (58) at (3.75, 0) {};
		\node [style=none] (59) at (3.75, 0.5) {$f$};
		\node [style=none] (60) at (2.75, 0.5) {};
		\node [style=none] (61) at (1.25, 0.5) {};
		\node [style=none] (62) at (4.75, 0.5) {};
		\node [style=none] (63) at (5.5, 1) {$B$};
		\node [style=none] (64) at (9, 1) {$C$};
		\node [style=none] (65) at (4.75, -0.75) {};
		\node [style=function black] (66) at (2, 0.5) {};
		\node [style=none] (67) at (2.75, -0.25) {};
		\node [style=none] (68) at (3.75, -0.25) {};
		\node [style=none] (69) at (2.75, -1.25) {};
		\node [style=none] (70) at (3.75, -1.25) {};
		\node [style=none] (71) at (3.75, -0.75) {$f$};
		\node [style=none] (72) at (2.75, -0.75) {};
		\node [style=none] (73) at (10, 0.5) {};
		\node [style=none] (74) at (10, -0.75) {};
		\node [style=none] (75) at (9, -0.25) {$B$};
		\node [style=none] (76) at (6.5, 1) {};
		\node [style=none] (77) at (7.5, 1) {};
		\node [style=none] (78) at (6.5, 0) {};
		\node [style=none] (79) at (7.5, 0) {};
		\node [style=none] (80) at (7.5, 0.5) {$g$};
		\node [style=none] (81) at (6.5, 0.5) {};
		\node [style=none] (82) at (8.5, 0.5) {};
		\node [style=none] (83) at (0, 0) {=};
	\end{pgfonlayer}
	\begin{pgfonlayer}{edgelayer}
		\draw [style=thick-line] (14.center) to (15.center);
		\draw [style=thick-line] (14.center) to (16.center);
		\draw [style=thick-line] (16.center) to (17.center);
		\draw [style=thick-line] (20.center) to (19.center);
		\draw [style=thick-line, in=90, out=0] (15.center) to (22.center);
		\draw [style=thick-line, in=-90, out=0] (17.center) to (22.center);
		\draw [style=thick-line] (25.center) to (44.center);
		\draw [style=thick-line] (47.center) to (48.center);
		\draw [style=thick-line] (47.center) to (49.center);
		\draw [style=thick-line] (49.center) to (50.center);
		\draw [style=thick-line, in=90, out=0] (48.center) to (53.center);
		\draw [style=thick-line, in=-90, out=0] (50.center) to (53.center);
		\draw [style=thick-line] (22.center) to (52.center);
		\draw [style=thick-line] (53.center) to (43.center);
		\draw [style=thick-line] (55.center) to (56.center);
		\draw [style=thick-line] (55.center) to (57.center);
		\draw [style=thick-line] (57.center) to (58.center);
		\draw [style=thick-line] (61.center) to (60.center);
		\draw [style=thick-line, in=90, out=0] (56.center) to (62.center);
		\draw [style=thick-line, in=-90, out=0] (58.center) to (62.center);
		\draw [style=thick-line] (67.center) to (68.center);
		\draw [style=thick-line] (67.center) to (69.center);
		\draw [style=thick-line] (69.center) to (70.center);
		\draw [style=thick-line, in=90, out=0] (68.center) to (65.center);
		\draw [style=thick-line, in=-90, out=0] (70.center) to (65.center);
		\draw [style=thick-line, in=-180, out=-90, looseness=0.75] (66) to (72.center);
		\draw [style=thick-line] (65.center) to (74.center);
		\draw [style=thick-line] (76.center) to (77.center);
		\draw [style=thick-line] (76.center) to (78.center);
		\draw [style=thick-line] (78.center) to (79.center);
		\draw [style=thick-line, in=90, out=0] (77.center) to (82.center);
		\draw [style=thick-line, in=-90, out=0] (79.center) to (82.center);
		\draw [style=thick-line] (62.center) to (81.center);
		\draw [style=thick-line] (82.center) to (73.center);
		\draw [style=thick-line, in=-180, out=-90, looseness=0.75] (26) to (25.center);
	\end{pgfonlayer}
\end{tikzpicture}
}
        \end{equation}
    \end{definition}
    It then follows that, starting from \cref{eqn:proofCoarseUpdatesAppx}, we can (i) parallel compose both sides with $\id_\SysX$ and pre-compose both sides with $\indmodel_\sts \comp \Copy_\sysX$ (where $\Copy_\SysX$ is the copy map for the system $\SysX$) (ii) apply the definition of positivity to the deterministic composite in \cref{eqn:proofDeterminisitic} on the right hand side, and (iii) use the surjectivity of $\model_\sts$, see \cref{eqn:proofSurjective}, to obtain the following:
    \begin{equation}
        \scalebox{0.6}{\input{figures/tikzit/proofBeliefMap/ProofMartin.tikz}}
    \end{equation}
\end{proof}
}

\section*{Acknowledgments}
Manuel Baltieri was supported by JST, Moonshot R\&D, Grant Number JPMJMS2012.
Matteo Capucci is an Independent Researcher funded by the Advanced Research+Innovation Agency (ARIA).
Work by Martin Biehl on this paper was made possible through the support of Grants 62828 and 62229 from the John Templeton Foundation.
Nathaniel Virgo's work on this paper was made possible through the support of grant 62229 from the John Templeton Foundation.
The opinions expressed in this publication are those of the authors and do not necessarily reflect the views of the John Templeton Foundation.

\bibliographystyle{IEEEtran}
\bibliography{AllEntries}

\vfill

\end{document}